\documentclass[a4paper,11pt]{amsart}
\usepackage[utf8]{inputenc}
\usepackage{csquotes}
\usepackage{amsmath, amssymb, amsthm, xcolor,enumerate, enumitem}
\usepackage[english]{babel}
\usepackage{tikz-cd}
\usepackage[all]{xy}
\usepackage{microtype}
\usepackage[colorlinks=true, citecolor=blue, linkcolor=blue, bookmarks=true]{hyperref}
\usepackage{bbm}
\usepackage{graphicx}
\usepackage{adjustbox}
\usepackage{comment}

\newcounter{theorem}
\newtheorem{theorem}[theorem]{Theorem}
\newtheorem{lemma}[theorem]{Lemma}
\newtheorem{prop}[theorem]{Proposition}
\newtheorem{cor}[theorem]{Corollary}

\theoremstyle{definition}
\newtheorem{defn}[theorem]{Definition}

\theoremstyle{remark}
\newtheorem*{remark*}{Remark}
\newtheorem{rmk}[theorem]{Remark}
\newtheorem{example}[theorem]{Example}

\numberwithin{equation}{section}
\allowdisplaybreaks

\newcommand{\C}{\mathrm{C}^*}
\newcommand{\cU}{\mathcal{U}}
\newcommand{\cM}{\mathcal{M}}
\newcommand{\bv}{\mathbbm{v}}
\newcommand{\bw}{\mathbbm{w}}
\newcommand{\sep}{\mathrm{sep}}
\newcommand{\F}{\mathcal{F}}
\newcommand{\U}{\mathcal{U}}
\newcommand{\M}{\mathcal{M}}
\newcommand{\T}{\mathbb{T}}

\newcommand{\G}{\mathcal{G}}

\newcommand{\id}{\mathrm{id}}
\newcommand{\KK}{\mathrm{KK}}
\newcommand{\II}{\mathrm{II}}

\newcommand{\N}{\mathbb{N}}
\newcommand{\cC}{\mathcal{C}}
\newcommand{\Corr}{\mathrm{Corr}}
\newcommand{\Aut}{\mathrm{Aut}}
\newcommand{\Ad}{\mathrm{Ad}}
\newcommand{\Hilb}{\mathrm{Hilb}}
\newcommand{\Irr}{\mathrm{Irr}}
\newcommand{\Hom}{\mathrm{Hom}}

\newcommand{\AND}{\mathrm{and}}
\newcommand{\dist}{\mathrm{dist}}

\newcommand{\asym}{\cong_u}
\newcommand{\amp}{\mathrm{amp}}
\newcommand{\Ker}{\mathrm{Ker}}
\newcommand{\rev}{\mathrm{rev}}

\title[Elliott intertwining for actions of tensor categories]{An Elliott intertwining approach to classifying actions of $\C$-tensor categories}

\author{Sergio Girón Pacheco}
\address{\hskip-\parindent Sergio Girón Pacheco, Department of mathematics, KU Leuven, Celestijnenlaan 200B, 3001, Leuven, Belgium.}
\email{sergio.gironpacheco@kuleuven.be}

\author{Robert Neagu}
\address{\hskip-\parindent Robert Neagu, Department of mathematics, KU Leuven, Celestijnenlaan 200B, 3001, Leuven, Belgium. }
\email{robert.neagu@kuleuven.be}

\thanks{RN funded by the European Union. Views and opinions expressed are however those of the authors only and do not necessarily reflect those of the European Union or the European Research Council. Neither the EU nor the ERC can be held responsible for them.}

\begin{document}

\begin{abstract}
We introduce a categorical approach to classifying actions of C$^*$-tensor categories on $\C$-algebras up to cocycle conjugacy. We show that, in this category, inductive limits exist and there is a natural notion of approximate unitary equivalence. Then, we generalise classical Elliott intertwining results to the tensor category equivariant case, in the same fashion as done by Szabó for the group equivariant case. 
\end{abstract}
\maketitle

\numberwithin{theorem}{section}	

\section*{Introduction}
\renewcommand*{\thetheorem}{\Alph{theorem}}
The study of existence and classification of symmetries on operator algebras has been a ubiquitous theme in the field. In the case of von Neumann algebras, this can be traced back to the work of Connes (\cite{connes1,connes2}), where he classified automorphisms of the hyperfinite $\II_1$ factor $\mathcal{R}$ up to outer conjugacy. Subsequently, Jones classified actions of finite groups on $\mathcal{R}$ (\cite{jonesactions}), while Ocneanu generalised these results to the case of amenable groups acting on $\mathcal{R}$ (\cite{ocneanu}).
\par Later, Popa classified amenable subfactors $N\subset\mathcal{R}$ by their standard invariant (\cite{popa1}). Following subsequent reformulations (\cite{LO94,MU03}), the standard invariant of a finite index subfactor $N\subset\mathcal{R}$ can be understood as a pair $(F,Q)$ with $F$ an action of a unitary tensor category $\mathcal{C}$ on $N$ and $Q$ a special object in $\cC$ that allows one to recover the inclusion $N\subset\mathcal{R}$ as a generalised crossed product.\footnote{Precisely, $Q$ is a $Q$-system as defined by Longo in \cite{LO94}.} The general setting of $\C$-tensor categories represents a unifying framework for studying group actions and more general quantum symmetries arising from subfactor theory. Following the recent spectacular classification results for group actions on Kirchberg algebras by group equivariant KK-theory (\cite{DynamicalKP}), and the development of a $\cC$-equivariant KK-theory (\cite{EqKK}), it is natural to consider to what extent one can classify actions of more general C$^*$-tensor categories on simple, amenable C$^*$-algebras.
\par Prior and closely linked to the classification of symmetries of operator algebras was the classification of the operator algebras themselves. First, Connes classified amenable factors in \cite{CO76} with the exception of one type which was later completed by Haagerup (\cite{HAA87}). In the case of C$^*$-algebras, Elliott classified C$^*$-inductive limits of finite dimensional C$^*$-algebras by their ordered K-theory in \cite{ElliottAF}. Elliott's methods laid the foundations for a roadmap to a classification of more general simple, amenable C$^*$-algebras. Following Elliott's strategy and building on decades of work by many mathematicians, the classification programme of simple, amenable $\C$-algebras successfully culminated in \cite{Kirch00, Phil00, GLN1, GLN2, TWW, ClassifMorphisms23}.

\par The general strategy for a classification of operator algebras is to first achieve existence and uniqueness results for morphisms with respect to the proposed classification invariant. Loosely speaking, to construct an isomorphism between two objects $A$ and $B$ in the category of $\C$-algebras, by virtue of an existence type statement, one obtains morphisms $\phi:A\to B$ and $\psi:B\to A$ which induce mutually inverse elements at the level of the chosen abstract invariant. Then, a suitable uniqueness result will give that $\psi\circ \phi$ is equivalent to $\id_A$ and $\phi\circ\psi$ is equivalent to $\id_B$. In the category of separable $\C$-algebras, the required notion of equivalence is approximate unitary equivalence. The final step is to  perform a so-called approximate intertwining argument. Essentially, one can tweak the morphisms $\phi$ and $\psi$ by unitaries until they become mutually inverse isomorphisms (see \cite[Corollary 2.3.4]{rordambook} and the introduction of \cite{cocyclecategszabo} for a more detailed breakdown of this roadmap to classification). Furthermore, in \cite{AbstractElliott}, Elliott proposes a similar classification strategy for arbitrary objects in more general categories that have a suitable notion of inner automorphisms. 
\par Instances of the implementation of Elliott's strategy for the classification of compact (quantum) group actions on C$^*$-algebras appear in \cite{GASA16,BASZVO17}, where the actions are assumed to have the Rokhlin property. In \cite{cocyclecategszabo}, Szab\'o articulates Elliott's proposed strategy in the generality of actions of locally compact groups as an alternative to the Evans--Kishimoto intertwining type arguments (see \cite{EvansKish}). In his work, Szab\'o defines an appropriate category of $\Gamma$-C$^*$-algebras, that is C$^*$-algebras carrying an action of $\Gamma$, and produces a vast collection of intertwining results for this category. Szab\'o's construction is a key ingredient facilitating the recent groundbreaking classification 
of amenable actions on Kirchberg algebras in \cite{DynamicalKP}. 
\par In this paper, we will develop the necessary techniques to perform approximate intertwining arguments for C$^*$-algebras carrying an action of a tensor category $\cC$ ($\cC$-$\C$-algebras). Glimpses of these techniques can be found in \cite{limitfusioncategclassif}, where exact intertwining arguments are performed to classify inductive limit actions of fusion categories on AF-algebras. However, for classification, it is often necessary to develop a more general framework which allows to perform \emph{approximate intertwining arguments}. In the setting of $\C$-algebras, Elliott used exact intertwining arguments to classify AF-algebras (\cite{ElliottAF}), but more refined approximate versions were needed to classify simple $A\mathbb{T}$-algebras in \cite{elliott}. Approximate intertwining arguments in the setting of actions of tensor category have already appeared in the literature; see \cite{IntertTomatsu} for an adaptation of the Evans-Kishimoto intertwining argument introduced in \cite{EvansKish}.

Primarily, the categorical framework developed by Szabó in \cite{cocyclecategszabo} provides the conceptual skeleton for our construction. However, unlike in the group action case, a tensor category might not act by automorphisms, so our techniques differ. In \cite{cocyclecategszabo}, Szabó defines a cocycle morphism as a pair consisting of an extendible $^*$-homomorphism and a unitary cocycle. In the setting of tensor categories, the action on a $\C$-algebra is given by a family of bimodules acting compatibly with respect to the tensor product. Therefore, the cocycles will be given by certain bimodule maps satisfying a family of commuting diagrams. Hence, it is not apparent how to adapt Szabó's arguments.

\par To perform intertwining techniques, we first need to introduce a category whose objects are $\cC$-C$^*$-algebras.\ Multiple notions of morphisms between $\cC$-C$^*$-algebras have appeared in the literature (see \cite[Definition 2.11]{EqKK} or \cite[Definition 3.2]{limitfusioncategclassif} for example).\ Any such notion has a common flavour: Given a C$^*$-tensor category $\cC$ acting on $\C$-algebras $A$ and $B$ via tensor functors $(F,J):\cC^{\rev}\to \Corr_0(A)$ and $(G,I):\cC^{\rev}\to\Corr_0(B)$ respectively, a morphism between $(A,F,J)$ and $(B,G,I)$ is given by an $A$-$B$-correspondence $E$ and a family of bimodule maps $\mathbbm{v}_X: F(X)\boxtimes E\rightarrow E\boxtimes G(X)$ satisfying coherence diagrams.\footnote{See Section \ref{sect: BuildingActions} for our conventions on actions of tensor categories.} The intricacies between the different definitions lie in how general correspondences we allow and what structure we expect the maps $\mathbbm{v}_X$ to have. We choose to work with correspondences which arise from (possibly degenerate) $^*$-homomorphisms $\varphi:A\rightarrow B$ and with (possibly non-adjointable) isometries $\mathbbm{v}_X$. We will say this data yields a \emph{cocycle morphism} and we denote the category consisting of $\cC$-$\C$-algebras and cocycle morphisms by $\C_{\cC}$. When the acting category is $\Hilb(\Gamma)$, and the action factors through automorphisms, we give an explicit formula for the family of bimodule maps $\{\mathbbm{v}_X\}_{X\in\cC}$ in Example \ref{example: intertHilb(G)}. 

In $\C_{\cC}$, we are able to construct inductive limits and define a suitable topology on the space of morphisms.\ An important challenge in constructing a category of cocycle morphisms is the composition. Composing two cocycle morphisms might not give a cocycle morphism. Therefore, we need to introduce a slightly different composition which will allow us to obtain a category. If the morphisms are non-degenerate, this new composition agrees with the canonical one, and with the one introduced by Szabó in the group action case ({\cite[Proposition 1.15]{cocyclecategszabo}}). Moreover, to avoid the need of restricting to non-degenerate morphisms, the maps $\mathbbm{v}_X$ are assumed to be \emph{isometries}. This assumption allows us, unlike in \cite{cocyclecategszabo}, to work in the possibly non-extendible setting, generalising the notion of a cocycle morphism from \cite{cocyclecategszabo} when restricting to twisted group actions.

An essential observation for our construction is that we can encode the information of a cocycle morphism into a family of linear maps satisfying certain conditions (see {\cite[Lemma 3.8]{limitfusioncategclassif}} where this is done for unital injective cocycle morphisms). Therefore, a cocycle morphism can be equivalently defined as follows.

\begin{defn}[Lemma \ref{linearmapspicture}]\label{defn: A}
Let $\cC$ be a $\C$-tensor category acting on $\C$-algebras $A$ and $B$ via $(A,F,J)$ and $(B,G,I)$ respectively. Then a cocycle morphism $(\phi,h): (A,F,J)\rightarrow (B,G,I)$  is given by a $^*$-homomorphism $\phi:A\to B$ and a family of linear maps: 
$$h= \{h^{X}: F(X)\rightarrow G(X)\}_{X\in\cC}$$ such that for any $X,Y\in\cC$
\begin{enumerate}
\item
for any morphism $f\in \Hom(X,Y)$, $G(f)\circ h^{X}=h^{Y}\circ F(f)$;
\item
$\phi(\langle x , y\rangle_A)=\langle h^{X}(x) , h^{X}(y)\rangle_B$ for any $x,y\in F(X)$; 
\item the following diagram commutes:
\[
\begin{tikzcd}
F(Y)\boxtimes F(X)
\arrow[swap]{d}{h^{Y}\boxtimes h^{X}}
\arrow{r}{J_{X,Y}}
& F(X\otimes Y)\arrow{d}{h^{X\otimes Y}}
 \\
G(Y)\boxtimes G(X)
\arrow{r}{I_{X,Y}}
& G(X\otimes Y);
\end{tikzcd}
\]

\item $h^{1_{\cC}}:A\to B$ is given by $h^{1_{\cC}}(a)=\phi(a)$ for any $a\in A$. 
\end{enumerate}   
\end{defn}

In the case of a group action, we give an explicit formula for the family of linear maps $\{h^X\}_{X\in\cC}$ in Example \ref{exmp: linearmapsHilb(G)}. Compared with the usual definition of a cocycle morphism given by pairs $(\phi,\mathbbm{v})$ consisting of a $^*$-homomorphism and a family of bimodule maps $\{\mathbbm{v}_X\}_{X\in\cC}$ (see Definition \ref{cocyclemorphism}), Definition \ref{defn: A} provides a more direct framework for setting up the intertwining arguments. In particular, to compose two cocycle morphisms, one composes the $^*$-homomorphisms, as well as the corresponding linear maps. Furthermore, the natural topology on cocycle morphisms is obtained by considering pointwise differences of the linear maps. When the acting category is semisimple and has countably many isomorphism classes of simple objects (which is assumed for the rest of the introduction), we can phrase convergence in this topology purely in terms of the linear maps.

\begin{defn}[Lemma \ref{lemma: topologycocyclemor}]
Let $(\phi_{\lambda},h_{\lambda}), (\phi,h):(A,F,J)\to (B,G,I)$ be cocycle morphisms. Then $(\phi_{\lambda},h_{\lambda})$ converges to $(\phi,h)$ if and only if for any $X\in\Irr(\cC)$, $h_{\lambda}^X$ converges pointwise to $h^X$ in the norm induced by the right inner product.\footnote{The collection $\Irr(\mathcal{C})$ is a choice of representatives for isomorphism classes of simple objects in $\cC$.}
\end{defn}

In general, this topology is coarser than the one used by Szabó in the group action case ({\cite[Definition 2.5]{cocyclecategszabo}}) (see Example \ref{exmp: DiffTop}). However, it is the same when restricted to non-degenerate cocycle morphisms. With this topology in hand, we can then define approximate unitary conjugation for cocycle morphisms. Moreover, unlike Szabó's definition in {\cite[Definition 2.8]{cocyclecategszabo}}, this is a symmetric relation (see Lemma \ref{lem: conjisequivalence}), which simplifies the intertwining arguments. This can be formulated as follows.

\begin{defn}[Lemma \ref{lemma: approxunitconj}]
If $(\phi,h),(\psi,l):(A,F,J)\to (B,G,I)$ are cocycle morphisms, then $(\psi,l)$ is \emph{approximately unitarily equivalent} to $(\phi,h)$ if and only if there exists a net of unitaries $u_\lambda\in\mathcal{U}(\mathcal{M}(B))$ such that $$\|l^X(x)-u_\lambda\rhd h^X(x)\lhd u_\lambda^*\| \overset{\lambda}{\longrightarrow} 0$$ for any $X\in \Irr(\cC)$ and any $x\in F(X)$.\footnote{Note that if $A$ is separable, and $\cC$ has countably many simple objects, then $u_{\lambda}$ can be chosen to be a sequence. The symbols $\rhd$ and $\lhd$ denote the left and right actions.}
\end{defn}

With the introduced topology at the level of morphism spaces and notion of unitary equivalence, we can place the subcategory of $\C_{\cC}$ consisting of separable $\cC$-$\C$-algebras and extendible cocycle morphisms in Elliott's abstract classification framework from \cite{AbstractElliott}. Precisely, the quotient category of $\C_{\cC}$ by approximate unitary equivalence is a classification category in the sense of \cite{AbstractElliott}. Hence, we get the following theorem.

\begin{theorem}[Corollary \ref{intertidentity}]\label{thm: IntroThmIntert}
Let $(F,J): \cC\curvearrowright A$ and $(G,I): \cC\curvearrowright B$ be actions on separable $\C$-algebras. Let
\[
(\phi, h): (A,F,J) \to (B,G,I) \quad
and \quad
(\psi, l): (B,G,I) \to (A,F,J)
\] be two extendible cocycle morphisms such that the compositions $(\psi,l)\circ (\phi, h)$ and $(\phi,h)\circ (\psi, l)$ are approximately inner. Then $(\phi, h)$ and $(\psi, l)$ are approximately unitarily equivalent to mutually inverse cocycle conjugacies.
\end{theorem}

Along with Theorem \ref{thm: IntroThmIntert}, in Section \ref{sect: TwoSidedInt} we use our construction to obtain various $\cC$-equivariant versions of classical intertwining arguments. For instance, we establish a general Elliott two-sided intertwining argument (see \cite[Proposition 2.3.2]{rordambook} for a non-equivariant version) as well as an intertwining through reparametrisation in Theorem \ref{thm: IntThroughRepar}. These results are essential ingredients in our upcoming classification results for actions of unitary tensor category on Kirchberg algebras in joint work with Kitamura.

In Section \ref{section: onesidedintert}, the techniques we develop allow us to perfom a $\cC$-equivariant one-sided intertwining argument (see Theorem \ref{thm: onesidedintert}), generalising the classical one-sided intertwining arguments of \cite[Proposition 2.3.5]{rordambook} and \cite[Proposition 4.3]{cocyclecategszabo}.\ This result is a key ingredient in the work of Evington, C. Jones, and the first named author in obtaining a McDuff-type characterisation of equivariant $\mathcal{D}$-stability of an action of a unitary tensor category, for strongly self-absorbing $\mathcal{D}$ (\cite{EqZ}).
The remaining part of Section \ref{section: onesidedintert} is concerned with asymptotic versions of the results in Section \ref{sect: TwoSidedInt}.

\section{Preliminaries}
\numberwithin{theorem}{section}

\subsection{Hilbert bimodules}

In this subsection, we collect a few basics on the theory of Hilbert bimodules. We refer the reader to \cite{Hilbertmodules} for a more detailed exposition. First, we start by recalling the definition of a Hilbert module as introduced by Paschke in \cite{PA73}.
\begin{defn}\label{defn:rightHilbertmodules}
Let $X$ be a vector space over $\mathbb{C}$ and $B$ be a $\C$-algebra. We say that $X$ is a \emph{(right-)Hilbert $B$-module} if $X$ is a right $B$-module equipped with a function $\langle\cdot, \cdot\rangle_B : X \times X \to B$ satisfying the following properties:
\begin{enumerate}[label=\textit{(\roman*)}]
\item $\langle\cdot, \cdot\rangle_B$ is left conjugate linear and right linear.\label{item:1Hilbert}
\item For any $x,y \in X$ and $b \in B$, one has that $\langle x,yb\rangle_B = \langle x,y\rangle_Bb$.\label{item:2Hilbert}
\item For any $x\in X$, $\langle x,x\rangle_B \geq 0$ and $\langle x,x\rangle_B =0$ if and only if $x=0$.\label{item:3Hilbert}
\item For any $x,y\in X$, $\langle x,y\rangle_B = \langle y,x\rangle_B^*$.\label{item:4Hilbert} 
\item $X$ is complete with respect to the norm induced by $\|\langle x,x\rangle_B\|^{1/2}$.\label{item:5Hilbert}
\end{enumerate}
\end{defn}

If $X$ only satisfies the properties \ref{item:1Hilbert}-\ref{item:4Hilbert} above, then we say $X$ is a \emph{pre-Hilbert $B$-module}. For $x\in X$, we denote by $|x|^2$ the positive element of $B$ given by $\langle x,x\rangle_B$. We denote the adjointable operators of $X$ by $\mathcal{L}(X)$. A \emph{(right-)Hilbert $A$-$B$-bimodule}, is a right Hilbert $B$-module with a left $A$ action by adjointable operators. For a Hilbert $A$-$B$-bimodule $X$, we often write the left action of an element $a\in A$ on a vector $x\in X$ by $a\rhd x$ and the right action of an element $b\in B$ on $x$ by $x\lhd b$. The morphisms we consider between Hilbert $A$-$B$-bimodules are adjointable linear maps which commute with the left $A$-action. In particular, Hilbert $A$-$B$-bimodules are called isomorphic if there is a unitary bimodule map between them.

\begin{example}\label{example: homsbimodules}
For a C$^*$-algebra $B$, the vector space $X=B$ is a Hilbert $B$-module with the right $B$-module structure given by multiplication and the inner product $\langle a,b\rangle_B=a^*b$ for $a,b\in B$. In this case, any $^*$-homomorphism $\phi:A\rightarrow M(B)$ induces a right-Hilbert $A$-$B$-bimodule which we denote by $_{\phi}B$.
\end{example}
Let $A,B,C$ be $\C$-algebras. If $X$ is a Hilbert $A$-$B$-bimodule and $E$ is a Hilbert $B$-$C$-bimodule we may form their \emph{internal tensor product} $X\boxtimes E$ that is a Hilbert $A$-$C$-bimodule. We sketch this construction and refer to \cite[Section 4]{Hilbertmodules} for details. To perform the internal tensor product one starts by considering the algebraic tensor product of vector spaces $X\odot E$. We identify the elements of the form $x\lhd b\odot y$ with $x\odot b\rhd y$ to form the quotient
\begin{equation*}
    V=X\odot E/\mathrm{span}\{x\lhd b\odot y-x\odot b\rhd y:\ x\in X,\ y\in E,\ b\in B\}.
\end{equation*}
We denote the image of the elementary tensor $x\odot y$ of $X\odot E$ in $V$ under the canonical quotient map by $x\boxtimes y$. One may define a right $C$-action and a right $C$-inner product on $V$ by
\begin{equation*}
    (x\boxtimes y)\lhd c=x\boxtimes (y \lhd c),
\end{equation*}
\begin{equation*}
    \langle x\boxtimes y, z\boxtimes w\rangle_C=\langle y,\langle x,z\rangle_B \rhd w\rangle_C, 
\end{equation*}
for any $x,z\in X$ and $y,w\in E$. It follows that $V$ equipped with this $C$-action and $C$-valued inner product satisfies \ref{item:1Hilbert}-\ref{item:4Hilbert} of Definition \ref{defn:rightHilbertmodules}. We produce a Hilbert $C$-module $X\boxtimes E$ by completing $V$ under the norm defined by the inner product. Moreover, one can induce a left $A$ action on $X\boxtimes E$ through
\begin{equation*}
    a\rhd (x\boxtimes y)=(a\rhd x)\boxtimes y
\end{equation*}
for all $a\in A$, $x\in X$ and $y\in E$. This equips $X\boxtimes E$ with the structure of a Hilbert $A$-$C$-bimodule.

If the bimodules are given by $^*$-homomorphisms into the multiplier algebra as in Example \ref{example: homsbimodules}, one may get greater insight into the structure of their tensor product. First we recall the definition of non-degenerate and extendible $^*$-homomorphisms.

\begin{defn}\label{defn: non-deg}
Let $A$ and $B$ be $\C$-algebras and $\phi:A\to \mathcal{M}(B)$ a $^*$-homomorphism.\ Then $\phi$ is said to be \emph{non-degenerate} if $\phi(A)B$ is dense in $B$.
\end{defn}

If $\phi:A\to\mathcal{M}(B)$ is non-degenerate, then $\phi$ extends to a unital $^*$-homomorphism $\phi:\mathcal{M}(A)\to\mathcal{M}(B)$ ({\cite[Proposition 2.5]{Hilbertmodules}}). 

\begin{defn}\label{defn: extendible}
Let $A$ and $B$ be $\C$-algebras.\ A $^*$-homomorphism $\phi:A\to \mathcal{M}(B)$ is called \emph{extendible} if for any increasing approximate unit $e_{\lambda}\in A$ (throughout this paper all approximate units are assumed to be increasing), the net $\phi(e_{\lambda})\in\mathcal{M}(B)$ converges strictly to a projection $p\in\mathcal{M}(B)$.\footnote{Any non-degenerate $^*$-homomorphism $\phi: A\rightarrow \M(B)$ is extendible as for an approximate unit $e_{\lambda}\in A$ the net $\phi(e_\lambda)$ converges strictly to $1_{\mathcal{M}(B)}$ (see e.g. \cite[II.7.3.8.]{BL06}).} As discussed in {\cite[Definition 1.7]{cocyclecategszabo}}, $\phi$ then factorises through $\mathcal{M}(pBp)\cong p\mathcal{M}(B)p\subseteq \mathcal{M}(B)$, with the $^*$-homomorphism $\phi_p:A\to\mathcal{M}(pBp)$ now non-degenerate. In this case, following \cite{cocyclecategszabo}, we let $\phi^{\dagger}:\mathcal{U}(\mathcal{M}(A))\to\mathcal{U}(\mathcal{M}(B))$ be the unital group homomorphism defined by $\phi^{\dagger}(u)=\phi_p(u)+(1_{\M(B)}-p)$ for all $u\in\mathcal{U}(\mathcal{M}(A))$. Note that for any $a\in A$ and $u\in \U(\M(A))$ one has that $\phi(ua)=\phi^{\dagger}(u)\phi(a)$.
\end{defn}

\begin{rmk}\label{rmk: extendible}
Let $\phi:A\rightarrow \M(B)$, $\psi:B\rightarrow \M(C)$  be $^*$-homomorphisms with $\psi$ extendible, and $q\in \M(C)$ be the projection corresponding to $\psi$. One may compose them to get the $^*$-homomorphism $\psi_q\circ \phi:A\rightarrow \M(C)$ that we simply denote by $\psi\circ\phi$ (see \cite[Remark 1.9]{cocyclecategszabo}). If $\phi$ is also extendible, then the composition $\psi\circ\phi$ remains extendible and $(\psi\circ\phi)^{\dagger}=\psi^{\dagger}\circ \phi^{\dagger}$.
\end{rmk}

The proposition below yields a more explicit form for tensor products of bimodules arising from $^*$-homomorphisms into the multiplier algebra as in Example \ref{example: homsbimodules}. If the homomorphisms are non-degenerate, then the result is folklore (see \cite{correspondences}).

\begin{prop}\label{extendiblecase}
Suppose that $A,B$, and $C$ are $\C$-algebras and $\phi:A\to \mathcal{M}(B)$, $\psi:B\to\mathcal{M}(C)$ are $^*$-homomorphisms with $\psi$ extendible, but possibly degenerate. Then ${}_{\phi}B\boxtimes{}_{\psi}C\cong{}_{\psi\circ\phi}\overline{\psi(B)C}={}_{\psi\circ\phi}\psi(B)C$.  
\end{prop}

\begin{proof}
First, note that $\overline{\psi(B)C}=\psi(B)C$ by Cohen's factorisation (see \cite[Proposition 2.33]{raeburnwilliamsbook}). Let $T:{}_{\phi}B\boxtimes{}_{\psi}C\to{}_{\psi\circ\phi}\psi(B)C$ be the continuous linear map given by $T(b\boxtimes c)=\psi(b)c$ for any $b\in B$ and $c\in C$. We claim that $T$ is a bimodule isomorphism.\ Using the definition of the inner product, a standard check shows that $T$ is a well-defined bimodule map.\ Taking $(\eta_\lambda)_{\lambda\in \Lambda}$ to be an approximate unit for $B$ let $S:{}_{\psi\circ\phi}\psi(B)C\to {}_{\phi}B\boxtimes{}_{\psi}C$ be given by $S(c)=\lim\limits_{\lambda}\eta_\lambda\boxtimes c$ for any $c\in C$. The map $S$ is well-defined precisely because $\psi$ is extendible. Indeed, for $\lambda,\mu\in \Lambda$
\begin{align*}
|(\eta_\lambda-\eta_\mu)\boxtimes c|^2&=\langle c, \langle \eta_\lambda-\eta_\mu,\eta_\lambda-\eta_\mu\rangle_B \rhd c\rangle_C\\&=\langle c, (\eta_\lambda-\eta_\mu)(\eta_\lambda-\eta_\mu)\rhd c\rangle\\ &= \langle \psi(\eta_\lambda-\eta_\mu)c,\psi(\eta_\lambda-\eta_\mu)c\rangle.
\end{align*} Therefore, $S$ is well defined if and only if $\psi(\eta_\lambda-\eta_\mu)c$ converges to $0$ for any $c\in C$, which is equivalent to saying that $\psi(\eta_\lambda)$ converges strictly in $\mathcal{M}(C)$. Hence $S$ is well-defined. Moreover, it can be seen that $S$ is the adjoint of $T$ and hence continuous. Now for any $b\in B$ and $c\in C$, $$S(T(b\boxtimes c))=\lim\limits_{n\to\infty}\eta_n\boxtimes \psi(b)c=\lim\limits_{n\to\infty}\eta_n\boxtimes b\rhd c=\lim\limits_{n\to\infty}\eta_nb\boxtimes c=b\boxtimes c.$$ Then, by continuity, $S\circ T$ is the identity on the domain of $T$. Therefore, $T$ is injective. As $T$ is surjective by construction, it follows that $T$ is a bimodule isomorphism.
\end{proof}

\begin{rmk}\label{rmk: IndepApproxUnit}
By uniqueness of the adjoint, the map $S:{}_{\psi\circ\phi}\psi(B)C\to {}_{\phi}B\boxtimes{}_{\psi}C$ of the proof of Proposition \ref{extendiblecase} given by $S(c)=\lim\limits_{\lambda}\eta_\lambda\boxtimes c$ for any $c\in C$ and a choice of approximate unit $(\eta_\lambda)_{\lambda\in \Lambda}$ for $B$, does not depend on the choice of the approximate unit $\eta_\lambda$.
\end{rmk}
\begin{defn}\label{def:nondegbim}
    Let $A$ be a C$^*$-algebra. A Hilbert $A$-$B$-bimodule $X$ is called \emph{non-degenerate} if $XB$ is dense in $X$ and $AX$ is dense in $X$.
\end{defn}
\begin{rmk}
    Note that if $X$ is a right-Hilbert $B$-module then $XB$ is always dense in $X$ by the argument in \cite[Proposition 2.16]{nonunitalcateg}. Therefore, a Hilbert $A$-$B$-bimodule $E$ may only fail to be non-degenerate if $AX$ is not dense in $X$.
\end{rmk}
We end this subsection by showing that if $A$ is a $\C$-algebra and $E$ is a non-degenerate Hilbert $A$-bimodule, then we can obtain a well-defined action of $\mathcal{M}(A)$ on $E$. 

Let $(\eta_\lambda)_{\lambda\in\Lambda}$ be an approximate unit of $A$ and let $L_E:E\to A\boxtimes E$ and $R_E:E\to E\boxtimes A$ be the $A$-bimodule maps given by $L_E(x)=\lim\limits_{\lambda}\eta_\lambda\boxtimes x$ and $R_E(x)=\lim\limits_{\lambda}x\boxtimes\eta_\lambda$ for all $x\in E$.\ Note that, since $E$ is non-degenerate, $L_E$ and $R_E$ are unitary bimodule isomorphisms.\ This follows similarly to the proof of Proposition \ref{extendiblecase}. The inverses of $L_E$ and $R_E$ are given by their adjoints $L_E^{-1}(a\boxtimes x)=a\rhd x$ and $R_E^{-1}(x\boxtimes a)=x\lhd a$, for all $a\in A$ and $x\in E$ respectively.  

\begin{lemma}\label{lemma:ActionMultiplierAlg}
Let $A$ be a $\C$-algebra and $E$ be a non-degenerate Hilbert $A$-bimodule. Then one may extend the left and right actions of $A$ on $E$ to left and right actions of $\mathcal{M}(A)$ on $E$. These extended actions equip $E$ with the structure of a Hilbert $\mathcal{M}(A)$-bimodule.
\end{lemma}

\begin{proof}
For any $v\in\mathcal{M}(A)$ and any $x\in E$, let us define 
\begin{equation}\label{eq: ActionsM(A)}
v\rhd x=L_E^{-1}(v\rhd L_E(x)) \quad \AND \quad x\lhd v = R_E^{-1}(R_E(x)\lhd v).
\end{equation} We claim that these formulae define left and right actions of $\mathcal{M}(A)$ on $E$. Moreover, it is clear that (\ref{eq: ActionsM(A)}) restricted to $A$ coincides with the $A$-bimodule structure of $E$.

First, note that the left action of $\mathcal{M}(A)$ on $A\boxtimes E$ is given by left multiplication on $A$. Similarly, the right action of $\mathcal{M}(A)$ on $E\boxtimes A$ is given by right multiplication on $A$. Since $L_E, R_E$, and their inverses are bimodule maps, it is straightforward to see that the formulae in \eqref{eq: ActionsM(A)} define left and right actions of $\mathcal{M}(A)$ on $E$. To see that $E$ with its right $A$-valued inner product is a Hilbert $\mathcal{M}(A)$-bimodule it suffices to check that the right $\mathcal{M}(A)$-action commutes with the inner product and the left $\mathcal{M}(A)$-action consists of adjointable operators. First, for $x,y\in E$ and $v\in \mathcal{M}(A)$ we have that
\begin{align*}
\langle y,x\lhd v\rangle_A &=\langle y, R_{E}^{-1}(R_E(x)\lhd v)\rangle_A\\
&=\lim\limits_{\lambda}\lim\limits_{\mu}\langle y\boxtimes\eta_{\lambda},x\boxtimes\eta_{\mu}v\rangle_A\\
&=\lim\limits_{\lambda}\langle \eta_{\lambda},\langle y,x\rangle_Av\rangle_A\\
&=\langle y,x\rangle_A v,
\end{align*}so the right $\mathcal{M}(A)$-action commutes with the right inner product.\ Moreover, the operator of left multiplication by $v\in \mathcal{M}(A)$ has as an adjoint the operator of left multiplication by $v^*$.\ Indeed left multiplication by $v^*$ is the adjoint of the operator of left multiplication by $v$ on the Hilbert $A$-module $A$ so for $x,y\in E$
$$\langle v\rhd x,y\rangle_A=\langle L_E^{-1}(v\rhd L_E(x)),y\rangle_A=\langle L_E(x),v^*\rhd L_E(y)\rangle_A =\langle x,v^*\rhd y\rangle_A$$
as required.
\end{proof}

\subsection{Correspondences}
A reformulation of the theory of Hilbert bimodules is the language of C$^*$-correspondences.
\begin{defn}
Let $A,B$ be $\C$-algebras. An \emph{$A$-$B$-correspondence} is a Hilbert $B$-module $X_B$ together with a $^*$-homomorphism $\phi:A\to\mathcal{L}(X_B)$. We often simply denote a correspondence by its underlying $^*$-homomorphism. Moreover, we denote the collection of $A$-$B$-correspondences by $\Corr(A,B)$. 
\end{defn}

\begin{rmk}
Note that any $A$-$B$-correspondence induces a right-Hilbert $A$-$B$-bimodule and vice versa. If $\phi:A\to\mathcal{L}(X_B)$ is an $A$-$B$-correspondence, then $X_B$ becomes a right-Hilbert $A$-$B$-bimodule, with the left action given by $\phi$. We will often denote this bimodule by ${}_{\phi}X$, forgetting the right $B$-action. Conversely, given a right-Hilbert $A$-$B$-bimodule $X$, the left action by $A$ induces an $A$-$B$-correspondence. Therefore, we will freely flip between the two pictures. 
\end{rmk}

One may compose correspondences through the tensor product of bimodules.\ Precisely, let $X$ be a  Hilbert $A$-$B$-bimodule inducing $\phi\in\Corr(A,B)$, and let $E$ be a Hilbert $B$-$C$-bimodule inducing $\psi\in\Corr(B,C)$.\ Then $X\boxtimes E$ gives a Hilbert $A$-$C$-bimodule which induces an element in $\Corr(A,C)$ denoted by $\psi\circ\phi$. Although the theories of bimodules and correspondences are equivalent, we sometimes choose to work with correspondences as the composition resembles composition of $^*$-homomorphisms between $\C$-algebras in a covariant manner.
\subsection{\texorpdfstring{$\C$}{C}-tensor categories}
Throughout this section we will assume that the reader is familiar with the standard language of category theory. For a category $\cC$ we will use capital letters e.g. $X,Y,$ and $Z$ to denote objects of the category. The space of morphisms between two objects $X,Y\in \cC$ will be denoted by $\Hom(X,Y)$. All categories in this section will be \emph{$\mathbb{C}$-linear}, that is an additive category such that its space of morphisms $\Hom(X,Y)$ between any two objects $X,Y$ is a $\mathbb{C}$-vector space and the composition of morphisms yields a bilinear map $\Hom(X,Y)\times \Hom(Y,Z)\rightarrow \Hom(X,Z)$ for any $X,Y,Z\in \cC$. 

We invite the reader to recall the definition of a $\C$-category from \cite{GLR85}.
Let $\cC$ and $\mathcal{D}$ be $\C$-categories.\ A functor $F:\cC\rightarrow \mathcal{D}$ is called a \emph{$\C$-functor} if the induced mappings  $\Hom(X,Y)\to\Hom(F(X),F(Y))$ are $\mathbb{C}$-linear and $^*$-preserving.\ A natural transformation $\nu:F\rightarrow G$ between $\C$-functors is called an \emph{isometry} if $\nu_X^*\nu_X=\id_{F(X)}$ for all $X\in \cC$. Moreover, $\nu$ is called a \emph{unitary} if it is a surjective isometry in which case $\nu_X\nu_X^*=\id_{G(X)}$ for all $X\in\cC$.\ We are interested in C$^*$-categories that admit tensor product structures.
\begin{defn}{(see for example \cite{JP16})}\label{defn:C*tensor}
    A \emph{$\rm{C}^*$-tensor category} is a $\rm{C}^*$-category $\mathcal{C}$ together with a $\mathbb{C}$-linear bifunctor $-\otimes-:\mathcal{C}\times\mathcal{C} \rightarrow \mathcal{C}$, a distinguished object $1_\mathcal{C} \in \cC$ and unitary natural isomorphisms   
        \begin{align}
            \alpha_{X,Y,Z}&: (X \otimes Y) \otimes Z \rightarrow X \otimes (Y \otimes Z),\\
            \lambda_X&: (1_\mathcal{C} \otimes X) \rightarrow X,\nonumber\\
            \rho_X&: ( X \otimes 1_\mathcal{C}) \rightarrow X,\nonumber
        \end{align} 
such that $(\phi \otimes \psi)^* = (\phi^* \otimes \psi^*)$ and the following diagrams commute for any $X,Y,Z,W\in \cC$
    
\begin{equation}
\begin{tikzcd}[column sep=0.3em]
 {}
&((W\otimes X)\otimes Y)\otimes Z \arrow{dl}{\alpha_{W,X,Y}\otimes \operatorname{id_Z}} \arrow{drr}{\alpha_{W\otimes X,Y,Z}}
& {}  \\
(W\otimes(X\otimes Y))\otimes Z \arrow{d}{\alpha_{W,X\otimes Y,Z}}
& & & (W\otimes X)\otimes (Y\otimes Z) \arrow{d}{\alpha_{W,X,Y\otimes Z}} \\
W\otimes ((X\otimes Y)\otimes Z) \arrow{rrr}{\operatorname{id_W}\otimes \alpha_{X,Y,Z}}
& & & W\otimes (X\otimes(Y\otimes Z)),
\end{tikzcd}
\end{equation}
\vspace{-1em}

\begin{equation}
\begin{tikzcd}
(X\otimes 1_{\mathcal{C}})\otimes\arrow{rr}{\alpha_{X,1,Y}}\arrow{dr}[swap]{\rho_X\otimes \id_Y} Y& {} & X\otimes (1_{\mathcal{C}}\otimes Y)\arrow{dl}{\id_X\otimes\lambda_Y}\\
{}&X\otimes Y.&{}
\end{tikzcd}
\end{equation}
and $\Hom(1_{\cC},1_{\cC})\cong \mathbb{C}$. Moreover, $\cC$ is said to be \emph{semisimple} if there exists a collection of objects $X_i\in \cC$ with $\Hom(X_i,X_j)\cong \delta_{ij}\mathbb{C}$ and any object $X\in \cC$ can be decomposed uniquely as a finite direct sum $X\cong \bigoplus_i X_i^{\oplus m_i}$, where $m_i$ is the multiplicity of each irreducible $X_i$.
\end{defn}
We call the structure morphism $\alpha$ as in Definition \ref{defn:C*tensor} associated to a C$^*$-tensor category its \emph{associator} and the maps $\lambda$ and $\rho$ the \emph{unitors}. We call an object $X\in \cC$ such that $\Hom(X,X)\cong \mathbb{C}$ \emph{irreducible} or \emph{simple}. For a C$^*$-tensor category $\cC$, we will denote by $\Irr(\cC)$ a collection of isomorphism class representatives for simple objects in $\cC$.
\begin{example}\label{exm:HilbG}
    Let $\Gamma$ be a countable discrete group (or more generally a countable, discrete monoid). We denote by $\Hilb(\Gamma)$ the semisimple C$^*$-tensor category whose objects are finite-dimensional $\Gamma$-graded Hilbert spaces, i.e. finite-dimensional Hilbert spaces $\mathcal{H}$ endowed with a decomposition $\mathcal{H}=\bigoplus_{g\in \Gamma} \mathcal{H}_g$. The morphisms are linear maps that preserve the $\Gamma$-grading. The tensor product is the usual Hilbert space tensor product with the grading defined by $$(\mathcal{H}\otimes\mathcal{K})_g=\bigoplus_{hk=g}\mathcal{H}_h\otimes \mathcal{K}_{k}.$$The isomorphism classes of simple objects in this category are indexed by group elements in $\Gamma$. We denote these graded Hilbert spaces by $\mathbb{C}_g$ where
    $$(\mathbb{C}_g)_h=\begin{cases}\mathbb{C}\ \text{if}\ g=h,\\
    0\ \text{otherwise}.
    \end{cases}$$ If $\Gamma$ is a countable discrete group and $\omega\in Z^3(\Gamma,\T)$ is a normalised $3$-cocycle (see \cite[Section III.1]{BR82} for definitions), the category $\Hilb(\Gamma,\omega)$ is defined exactly as is $\Hilb(\Gamma)$ but with associators now given by
    $$\alpha_{\mathcal{H},\mathcal{K},\mathcal{L}}:(\xi\boxtimes\eta)\boxtimes\mu \mapsto \omega(g,h,k) \xi\boxtimes(\eta\boxtimes\mu)$$
    for $\xi\in \mathcal{H}_g,\eta\in \mathcal{K}_h$ and $\mu\in \mathcal{H}_k$.
    \par For the remainder of this paper we will only make use of this example when $\Gamma$ is a countable discrete group.\ However, our results also apply to the case of monoids.
\end{example}
We now introduce the most important example for our purposes. 
\begin{example}
 Let $A$ be a C$^*$-algebra. We denote by $\Corr_0(A)$ the C$^*$-tensor category whose objects are non-degenerate $A$-$A$-correspondences and whose morphisms are adjointable bimodule maps between the underlying Hilbert bimodules. The tensor product of two $A$-$A$-correspondences $\varphi$ with $\psi$ is given by their composition $\varphi\circ\psi$. The tensor identity of $\Corr_0(A)$ is given by the identity homomorphism $\id_A$, the associator is given by the rebracketing morphism associated to the underlying tensor product of Hilbert bimodules.
\end{example}
\begin{rmk}\label{rmk: composingcorr}
In general, $\Corr(A)$ is not a C$^*$-tensor category as there is no tensor unit on degenerate correspondences; $A$ does not act as a unit on a degenerate correspondence. It is a non-unital C$^*$-tensor category; this is a weakening of Definition \ref{defn:C*tensor} which omits the necessity of a tensor unit.
\end{rmk}
\subsection{Szabó's cocycle category}
Before we begin our discussion on actions of $\C$-tensor categories, we recall the case of a twisted action by a second-countable locally compact group $\Gamma$. We shall later identify this, in the case when $\Gamma$ is countable discrete, with actions of the C$^*$- tensor category $\Hilb(\Gamma)$.

\begin{defn}[cf. {\cite[Definition 1.1]{cocyclecategszabo}}]
Let $\Gamma$ be a locally compact group and $A$ be a $\C$-algebra.\ A \emph{twisted action} of $\Gamma$ on $A$ is a pair $(\alpha,\mathfrak{u})$, where $\alpha:\Gamma\to \Aut(A)$ is a point-norm continuous map, and $\mathfrak{u}:\Gamma\times \Gamma\to \mathcal{U}(\mathcal{M}(A))$ is a strictly continuous map satisfying 
\begin{equation}\label{twistedaction1}
    \alpha_1=\id_A, \quad \Ad(\mathfrak{u}_{g,h})\circ\alpha_g\circ\alpha_h= \alpha_{gh}
\end{equation} and 

\begin{equation}\label{twistedaction2}
    \mathfrak{u}_{g,1}=\mathfrak{u}_{1,g}= 1, \quad \mathfrak{u}_{k,gh}\alpha_k(\mathfrak{u}_{g,h})=\mathfrak{u}_{kg,h}\mathfrak{u}_{k,g}
    \end{equation}
for all $k,g,h\in \Gamma$.
\end{defn}

\begin{rmk}\label{rmk: DiffFormulaeGroupCase}
Note that the formulae above differ slightly from the definition of a twisted action in \cite{cocyclecategszabo}. In fact, the sole difference is that our \emph{unitary cocycles} $\mathfrak{u}_{g,h}$ for $g,h\in \Gamma$ are the adjoints of the cocycles in \cite{cocyclecategszabo}. The reason for this change of conventions will be discussed in Section \ref{sect: BuildingActions}.
\end{rmk}
A triple $(A,\alpha,\mathfrak{u})$ as above is called a \emph{twisted $\Gamma$-$\C$-algebra}.\ If the cocycle is trivial i.e. $\mathfrak{u}_{g,h}=\mathfrak{1}$ for all $g,h\in \Gamma$, then $(A,\alpha)$ is said to be a $\Gamma$-$\C$-algebra. 

\begin{defn}[cf. {\cite[Definition 1.10]{cocyclecategszabo}}]\label{def: cocyclemorGroups}
    Let $(\alpha,\mathfrak{u}): \Gamma\curvearrowright A$ and $(\beta,\mathfrak{v}): \Gamma\curvearrowright B$ be two twisted actions on $\C$-algebras $A$ and $B$ respectively.
\begin{enumerate}
 \item A \emph{cocycle representation} from $(A,\alpha,\mathfrak{u})$ to $(B,\beta,\mathfrak{v})$ is a pair $(\psi,\mathbbm{v})$, where
$\psi: A\to \mathcal{M}(B)$ is an extendible $^*$-homomorphism and $\mathbbm{v}: \Gamma\to\mathcal{U}(\mathcal{M}(B))$ is a strictly continuous map such that 
\begin{equation} \label{cocyclerep1}
\beta_g\circ\psi = \Ad(\mathbbm{v}_g)\circ\psi\circ\alpha_g
\end{equation}
and
\begin{equation} \label{cocyclerep2}
\psi^{\dagger}(\mathfrak{u}_{g,h}) = \mathbbm{v}_{gh}^*\mathfrak{v}_{g,h}\beta_g(\mathbbm{v}_h)\mathbbm{v}_g
\end{equation}
for all $g,h\in \Gamma$.

\item A \emph{cocycle morphism} from $(A,\alpha,\mathfrak{u})$ to $(B,\beta,\mathfrak{v})$ is a cocycle representation $(\psi,\mathbbm{v})$ as above, with the additional requirement that $\psi(A)\subseteq B$.
\end{enumerate}
\end{defn}

\begin{rmk}
Due to our change of conventions when defining twisted actions we need to change the definition of a cocycle representation. The formula in \eqref{cocyclerep2} differs from {\cite[Definition 1.10]{cocyclecategszabo}} precisely by taking adjoints.
\end{rmk}

As shown in \cite{cocyclecategszabo}, there exists a category with objects being twisted $\Gamma$-$\C$-algebras and morphisms being cocycle morphisms with composition defined in {\cite[Proposition 1.15]{cocyclecategszabo}}. This category is denoted by $\C_{\Gamma,t}$ (\cite[Definition 1.16]{cocyclecategszabo}).

Our attention will now focus on generalising this construction to actions of semisimple $\C$-tensor categories.

\section{Actions of \texorpdfstring{$\C$}{C}-tensor categories}\label{sect: BuildingActions}

For $\C$-tensor categories $(\mathcal{C},\otimes)$ and $(\mathcal{D},\boxtimes)$, $F:\mathcal{C}\to\mathcal{D}$ is said to be a $\C$-tensor functor if it is a $\C$-functor such that $F(1_{\cC})=1_{\mathcal{D}}$ and there exists a unitary natural isomorphism $J_{X,Y}:F(X)\boxtimes F(Y)\to F(X\otimes Y)$ such that 

\begin{equation}\label{diagramJmaps}
\begin{tikzcd}[column sep=5em]
(F(X)\boxtimes F(Y))\boxtimes F(Z)\arrow[r,"\alpha_{F(X),F(Y),F(Z)}"]\arrow[d,"J_{X,Y}\boxtimes\id_{F(Z)}"]  & F(X)\boxtimes (F(Y) \boxtimes F(Z)) \arrow[d,"\id_{F(X)}\boxtimes J_{Y,Z}"] \\
F(X \otimes Y)\boxtimes F(Z) \arrow[d,"J_{X \otimes Y,Z}"] & F(X)\boxtimes F(Y \otimes Z) \arrow[d,"J_{X, Y\otimes Z}"]\\
F((X \otimes Y)\otimes Z )\arrow[r,"F(\alpha_{X,Y,Z})"] & F(X\otimes (Y \otimes Z))
\end{tikzcd}
\end{equation}
commutes for all $X,Y,Z\in \mathcal{C}$.

\par In the following definition we denote by $\Corr_0^{\sep}(A)$ the full subcategory of $\Corr_0(A)$ consisting of bimodules with countable dense subsets. Moreover, for a $\C$-tensor category $\cC$, we denote by $\cC^{\rev}$ the $\C$-tensor category whose underlying category is $\cC$, but the tensor product is reversed i.e. $X\otimes^{\rev}Y=Y\otimes X$.
\begin{defn}\label{categaction}
A $\C$-tensor category $\mathcal{C}$ is said to act on a $\C$-algebra $A$ if there exists a $\C$-tensor functor $F:\mathcal{C}^{\rev}\to \Corr_0(A)$. If $A$ is separable we further impose that $F$ is valued in $\Corr_0^{\sep}(A)$. We will often denote this by $\mathcal{C}\overset{F}{\curvearrowright} A$ or by the triple $(A,F,J)$, where $$J=\{J_{X,Y}:F(Y)\boxtimes F(X)\to F(X\otimes Y)\}_{X,Y\in \cC }$$ is the natural isomorphism associated with the functor $F$. In this case, we say that the triple $(A,F,J)$ is a $\cC$-$\C$-algebra.
\end{defn}

\begin{rmk}
Often in the literature an action of $\cC$ on $A$ is given instead by a C$^*$-tensor functor $F:\cC\rightarrow \Corr_0(A)$. We choose to define it as a functor from $\cC^{\rev}$ so that an action of a countable discrete group $\Gamma$ induces an action of the category $\Hilb(\Gamma)$ (see Example \ref{exn:groupaction}).
\end{rmk}

\begin{rmk}\label{rmk: functornondegimage}
In the literature, the main interest is actions of \emph{unitary} tensor categories on C$^*$-algebras (see e.g. \cite{EqKK,limitfusioncategclassif,JP16,nonunitalcateg,IntertTomatsu,RHPBN23}). This is because unitary tensor categories axiomatise the standard invariant in subfactor theory and can be thought of as the mathematical objects encoding the symmetry in finite index inclusions of C$^*$-algebras. If $\cC$ is a unitary tensor category, $A$ is separable, and $(F,J):\cC\rightarrow \Corr_0(A)$ is a C$^*$-tensor functor, then the bimodule associated to the correspondence $F(X)$ for any $X\in \cC$ is of finite index and hence it has a countable dense subset by
{\cite[Corollary 2.24]{nonunitalcateg}}.
Therefore, $(F,J)$ automatically falls into $\Corr_0^{\sep}(A)$.
\end{rmk}
\begin{example}\label{exn:groupaction}
Let $\Gamma$ be a a countable discrete group and let $(\alpha,\mathfrak{u}):\Gamma\rightarrow \Aut(A)$ be a twisted action. The pair $(\alpha,\mathfrak{u})$ will induce a C$^*$-tensor functor $(\alpha,\mathfrak{u}):\Hilb(\Gamma)^{\rev}\rightarrow \Corr_0(A)$ by setting 
\begin{align}
&\alpha(\mathbb{C}_g)={}_{\alpha_g}A,\label{formulaeHilbGfunctor:1}\\
&\mathfrak{u}_{\mathbb{C}_g,\mathbb{C}_h}(a\boxtimes b)=\mathfrak{u}_{g,h}\alpha_g(a)b.\label{formulaeHilbGfunctor:2}
\end{align}
The functor may then be extended by linearity to all of $\Hilb(\Gamma)^{\rev}$ in a similar manner to \cite[Proposition 5.6]{EVGI21}. In general, actions of $\Hilb(\Gamma)$ on $A$ correspond to twisted actions of $\Gamma$ on $A\otimes\mathbb{K}$.
\end{example}

If $A$ is a $\C$-algebra, then its \emph{sequence algebra} $A_{\infty}$ is defined by $$A_{\infty}=\ell^{\infty}(\mathbb{N},A)/\Big\{(a_n)_{n\geq 1} : \lim\limits_{n\to\infty}\|a_n\|=0\Big\}.$$

We will end this section by showing that if $\mathcal{C}\overset{F}{\curvearrowright} A$ is an action on a separable $\C$-algebra $A$, then we can induce an action of $\cC$ on its sequence algebra $A_{\infty}$ whenever the image bimodules are finite rank in the following sense.
\begin{defn}
We call a right Hilbert module $E$ over a C$^*$-algebra $A$ \emph{finite rank} if there exist $n\in \N$ and a projection $p\in M_n(\mathcal{M}(A))$ such that $E\cong pA^n$.
In fact, $E$ is finite rank if and only if there exists a collection of elements $\xi_i\in \mathcal{L}(A,E)$ for $1\leq i\leq n$ such that
\begin{equation}\label{eqn:basis}
\sum_{i=1}^n\xi_i\xi_i^*=\id_{E}.
\end{equation}
\end{defn}
We call $\xi_i$ a \emph{basis} for $E$. We say that $(F,J)$ is \emph{finite rank} if each $F(X)$ for $X\in \Irr(\cC)$ is finite rank.

Suppose that $F:\cC^{\rev}\to\Corr_0^{\sep}(A)$ is a $\C$-tensor functor. Our goal is to build a $\C$-tensor functor $F_{\infty}:\cC^{\rev}\to \Corr_0(A_{\infty})$.

For any $X\in\cC$, we can view $F(X)$ as a non-degenerate Hilbert $A$-bimodule. Define 
\begin{equation}\label{eq: SeqAlgBimodule}
F_{\infty}(X)= \ell^{\infty}(\mathbb{N},F(X))/ \Big\{(\xi_n)_{n\geq 1} : \lim_{n\to\infty}\|\xi_n\|=0\Big\},
\end{equation} with the left and right actions of $A_\infty$ given pointwise. Precisely, for any $a\in A_\infty$ and any $\xi\in F_{\infty}(X)$ represented by $(a_n)_{n\geq 1}$ and $(\xi_n)_{n\geq 1}$, we let
\begin{equation}\label{eq: SeqAlgActions}
a\rhd\xi = (a_n\rhd \xi_n)_{n\geq 1} \quad \AND \quad \xi\lhd a=(\xi_n\lhd a_n)_{n\geq 1}.
\end{equation} Similarly, for any $\xi,\eta\in F_{\infty}(X) $ we define 
\begin{equation}\label{eq: InnerProdSeqAlg}
\langle \xi,\eta\rangle_{A_\infty} = \big(\langle \xi_n,\eta_n\rangle_A\big)_{n\geq 1}.
\end{equation}

\begin{lemma}\label{lemma: ConstructionCheck}
$F_{\infty}(X)$ is a right-Hilbert $A_{\infty}$-module.  
\end{lemma}

\begin{proof}
First, we need to check that the formulae in \eqref{eq: SeqAlgActions} are well-defined. For this, suppose the sequences $(a_n)_{n\geq 1}$ and $(a_n')_{n\geq 1}$ induce the same element in $A_\infty$, and let $\xi\in F_{\infty}(X)$ be represented by the sequence $(\xi_n)_{n\geq 1}$. Then, a direct calculation shows that $$\|\langle \xi_n\lhd (a_n-a_n'), \xi_n\lhd (a_n-a_n')\rangle_A\|\leq\|\langle \xi_n,\xi_n\rangle_A\|\|a_n-a_n'\|^2,$$ which converges to $0$ as $n\to\infty$. Exactly the same calculation shows that if $(\xi_n)_{n\geq 1}$ and $(\xi_n')_{n\geq 1}$ induce the same element in $F_{\infty}(X)$ and $(a_n)_{n\geq 1}\in A_\infty$, then $(\xi_n\lhd a_n)_{n\geq 1}=(\xi_n'\lhd a_n)_{n\geq 1}$ as elements in $F_{\infty}(X)$. Thus, \eqref{eq: SeqAlgActions} gives a well-defined right action of $A_\infty$. 

We now check that \eqref{eq: InnerProdSeqAlg} gives a well-defined right inner product. First, if $(\xi_n)_{n\geq 1}$ and $(\xi_n')_{n\geq 1}$ induce $\xi$, and $(\eta_n)_{n\geq 1}$ and $(\eta_n)_{n\geq 1}'$ induce $\eta$, then $$\langle \xi_n,\eta_n\rangle_{A}- \langle \xi_n',\eta_n'\rangle_A= \langle \xi_n-\xi_n',\eta_n\rangle_A+\langle \xi_n',\eta_n-\eta_n'\rangle_A,$$ which converges to $0$ by Cauchy-Schwarz.\ Moreover, the sequences $\big(\langle \xi_n,\eta_n\rangle_A\big)_{n\geq 1}$ and $\big(\langle \xi_n',\eta_n'\rangle_A\big)_{n\geq 1}$ are bounded, so induce the same element in $A_\infty$.

Therefore, the function $\langle \cdot, \cdot\rangle_{A_\infty}: F_{\infty}(X)\times F_{\infty}(X)\to A_{\infty}$ is well-defined. It is now straightforward to check that this function is right linear, left conjugate linear, and antisymmetric since all these properties are satisfied pointwise for each $n\in\mathbb{N}$. Finally, it is clear that $\langle \xi,\xi\rangle_{A_\infty} \geq 0$ and that $\langle \xi,\xi\rangle_{A_\infty} = 0$ if and only if $\langle \xi_n,\xi_n\rangle_A$ converges to $0$ i.e. $\xi=0$ in $F_{\infty}(X)$. Since, $F_{\infty}(X)$ is complete with respect to the inner product defined in \eqref{eq: InnerProdSeqAlg} (as a quotient of a complete space by a closed subspace), it is a right-Hilbert $A_{\infty}$-bimodule. 
\end{proof}

\begin{lemma}
$F_{\infty}(X)$ is a non-degenerate Hilbert $A_\infty$-bimodule.
\end{lemma}

\begin{proof}
As in the proof of Lemma \ref{lemma: ConstructionCheck} we have that \eqref{eq: SeqAlgActions} gives a well-defined left action by $A_\infty$.\ Since the left action of $A$ on $F(X)$ is adjointable, the left action of $(a_n)_{n\geq 1}\in A_\infty$ is an adjointable operator with adjoint given by the action of $(a_n^*)_{n\geq 1}$. Moreover, $F_\infty(X)$ is non-degenerate for any $X\in \cC$ as $F(X)$ is. 
\end{proof}

Consider the functor $F_\infty:\cC^{\rev}\to \Corr_0(A_\infty)$ defined by sending any $X\in\cC$ to the correspondence induced by the Hilbert $A_\infty$-bimodule $F_\infty(X)$, and any $f\in \Hom(X,Y)$ for $X,Y\in \cC$ to the intertwiner defined on $F_\infty(X)$ by $F_{\infty}(f)((\xi_n)_{n\geq 1})=(F(f)(\xi_n))_{n\geq 1}$. Also, one can define $J_{X,Y}^{\infty}:F_\infty(Y)\boxtimes F_\infty(X)\to F_\infty(X\otimes Y)$ as the unique continuous extension of the map $$J_{X,Y}^{\infty}(\xi\boxtimes\eta)=(J_{X,Y}(\xi_n\boxtimes\eta_n))_n,$$ for any $\xi\in F_\infty(Y)$, any $\eta\in F_\infty(X)$ (that this is well defined follows in a similar fashion as the arguments in the proof of Lemma \ref{lemma: ConstructionCheck}). Note that $J^{\infty}_{X,Y}$ is isometric for all $X,Y\in \cC$.

\begin{lemma}\label{lemma: TensorFunctor}
Let $(F,J,A)$ be a finite rank action of a C$^*$-tensor category $\cC$. Then $\cC$ acts on $A_\infty$ via the triple $(A_\infty, F_\infty, J^{\infty})$.
\end{lemma}

\begin{proof}
It follows from construction that $F_\infty$ is a C$^*$-functor. Moreover, the naturality of $J^{\infty}$ and that $J^{\infty}$ satisfies commuting diagrams as in \eqref{diagramJmaps} follows from direct computations. It remains to show that $J_{X,Y}^\infty$ is surjective. Let $\iota_{X,Y}:F_{\infty}(Y) \boxtimes F_{\infty}(X) \rightarrow (F(Y)\boxtimes F(X))_\infty$ be the canonical inclusion maps and $T_{X,Y}:(F(Y)\boxtimes F(X))_\infty\rightarrow F_{\infty}(X\otimes Y)$ be the bounded bimodule map defined by $T(\xi)_n=J_{X,Y}(\xi_n)$ for $\xi=(\xi_n)_{n\geq 0}\in (F(Y)\boxtimes F(X))_\infty$. The composition $T_{X,Y}\circ \iota_{X,Y}$ coincides with $J_{X,Y}^{\infty}$ so it suffices to show the surjectivity of both $T_{X,Y}$ and $\iota_{X,Y}$. 
\par Firstly $T_{X,Y}$ is an adjointable unitary with adjoint defined by the mapping $(\xi_n)_{n\geq 0}\mapsto (J_{X,Y}^*(\xi_n))_{n\geq 0}$ for $(\xi_n)_{n\geq 0}\in F_{\infty}(X\otimes Y)$. It remains to show the surjectivity of $\iota_{X,Y}$. Let $X,Y \in \cC$ and let $\xi_i^{Y}$ with $i\in I_{Y}$ finite such that
    \[
    \sum_{i\in I_{Y}}\xi_i^{Y}{\xi_i^Y}^*=\id_{F(Y)}.
    \]
    Any $\zeta \in (F(Y)\boxtimes F(X))_{\infty}$ is represented by a sequence $(\sum_{l=1}^{k_n}y_l^{(n)}\boxtimes x_l^{(n)})_n$. By a standard reindexation one can choose an approximate unit $e_n$ for $A$ such that
    \[
    {e_n\xi_i^Y}^*(y_l^{(n)})-{\xi_i^Y}^*(y_l^{(n)})\rightarrow 0
    \]
    as $n$ tends to infinity for all $i\in I_{Y}$, $1\leq l\leq k_n$. Then one has that
    \begin{align*}
    \zeta&=\left(\sum_{l=1}^{k_n}y_l^{(n)}\boxtimes x_l^{(n)}\right)_n\\
    &=\left(\sum_{l=1}^{k_n}\sum_{i\in I_Y}\xi_i^{Y}(e_n{\xi_i^Y}^*(y_l^{(n)}))\boxtimes x_l^{(n)}\right)_n\\
    &=\left(\sum_{i\in I_Y}\xi_i^{Y}(e_n)\boxtimes \left(\sum_{l=1}^{k_n}{\xi_i^Y}^*(y_l^{(n)})x_l^{(n)}\right)\right)_n.
    \end{align*}
    Hence, it suffices to show that 
    \[
    \sum_{l=1}^{k_n}{\xi_i^Y}^*(y_l^{(n)})x_l^{(n)}
    \]
    is a bounded sequence. This follows precisely as $\zeta$ is a bounded sequence. Indeed,
    \begin{align*}
        \lVert \sum_{l=1}^{k_n}{\xi_i^Y}^*(y_l^{(n)})x_l^{(n)}\rVert^2&=\|\sum_{l,l'}\langle x_l^{(n)},({\xi_i^{Y}}^*(y_l^{(n)}))^*{\xi_i^{Y}}^*(y_{l'}^{(n)})x_{l'}^{(n)}\rangle\|\\
        &=\|\sum_{l,l'}\langle x_l^{(n)},({y_l^{(n)}}^*\xi_i^{Y}{\xi_i^{Y}}^*\boxtimes \id_{F(X)}) (y_{l'}^{(n)}\boxtimes {x_{l'}}^{(n)})\rangle\|\\
        &\leq \|\sum_{l,l'} \langle x_l^{(n)},({y_l^{(n)}}^*\boxtimes \id_{F(X)}) (y_{l'}^{(n)}\boxtimes {x_{l'}}^{(n)})\rangle\|\\
        &=\|\sum_{l,l'} \langle x_l^{(n)},\langle {y_l^{(n)}},y_{l'}^{(n)}\rangle x_{l'}^{(n)}\rangle\|\\
        &=\|\zeta_n\|^2
    \end{align*}
    and $\zeta_n$ is a bounded sequence.
\end{proof}
\begin{rmk}\label{rmk: FinGen}
If $\mathcal{C}\overset{F}{\curvearrowright} A$ is an action of a unitary tensor category (in the sense of \cite{limitfusioncategclassif} for example) on a unital C$^*$-algebra $A$, then it follows from \cite{nonunitalcateg, KW00} that $F$ is finite rank. 
\end{rmk}
\section{The generalised cocycle category}

In this section we introduce the category of $\cC$-$\C$-algebras for which we will later perform intertwining arguments. Throughout the rest of this paper $\cC$ is always assumed to be a $\C$-tensor category.

\begin{defn}\label{cocyclemorphism}
  Let $\mathcal{C}\overset{F}{\curvearrowright} A$ and $\mathcal{C}\overset{G}{\curvearrowright} B$ be actions of $\mathcal{C}$ on $\C$-algebras $A$ and $B$. 
  \begin{enumerate}
      \item A \emph{correspondence morphism} from $(A, F, J)$ to $(B, G, I)$ is a pair $(\phi,\{\mathbbm{v}_X\}_{X\in\mathcal{C}})$, where $\phi:A\to\mathcal{L}(E)$ is an $A$-$B$-correspondence and $\{\mathbbm{v}_X:\phi\circ F(X)\to G(X)\circ\phi\}_{X\in\cC}$ is a natural family of $A$-$B$-bimodule maps such that $\mathbbm{v}_X$ is an isometry (not necessarily adjointable) for any $X\in\cC$.\footnote{By an isometry, we mean a map which preserves the norm. By the proof of the theorem in \cite{Lance94}, it is equivalent to assume that it preserves the inner product.} Moreover, for all $X,Y\in\mathcal{C}$ the following pentagon diagram commutes 
\begin{equation}\label{cocyclemorphismdiagram}
\begin{adjustbox}{max width=\textwidth}
\begin{tikzcd}
& \phi\circ F(X)\circ F(Y) 
\ar{dr}{J_{X,Y}\boxtimes\id_{\phi}}
\ar[swap]{dl}{\id_{F(Y)}\boxtimes\mathbbm{v}_X} &    
\\   
G(X)\circ\phi\circ F(Y) \ar[swap]{dd}{\mathbbm{v}_Y\boxtimes\id_{G(X)}} 
& &\phi\circ F(X\otimes Y)\ar{dd}{\mathbbm{v}_{X\otimes Y}}
\\
& &
\\
G(X)\circ G(Y)\circ\phi
\ar{rr}{\id_{\phi}\boxtimes I_{X,Y}} 
&    
& G(X\otimes Y)\circ\phi, 
\end{tikzcd}
\end{adjustbox}
\end{equation} and $\mathbbm{v}_{1_\cC}:A\boxtimes {}_{\phi}E\to {}_{\phi} E\boxtimes B\cong {}_{\phi}E$ is given by $\mathbbm{v}_{1_\cC}(a\boxtimes x)=\phi(a)x$ for any $a\in A$ and $x\in {}_{\phi}E$. For convenience, we write $(\phi,\mathbbm{v})$, where $\mathbbm{v}$ denotes the collection of maps $\{\mathbbm{v}_X\}_{X\in\cC}$.\footnote{The notation of the maps in \eqref{cocyclemorphismdiagram} denotes the tensor product of bimodules, which is equivalent to composition of correspondences.}

\item A \emph{cocycle representation} $(\phi,\{\mathbbm{v}_X\}_{X\in\mathcal{C}}):(A,F,J)\to (B,G,I)$ is a correspondence morphism for which we further require that $\phi:A\to \mathcal{M}(B)$ is a $^*$-homomorphism.

\item A \emph{cocycle morphism} $(\phi,\{\mathbbm{v}_X\}_{X\in\mathcal{C}}):(A,F,J)\to (B,G,I)$ is a cocycle representation for which we further require that $\phi:A\to B$ is a $^*$-homomorphism.
  \end{enumerate}
\end{defn} 

\begin{rmk}\label{rmk: IrredDetermineCateg}
By naturality, if $\mathcal{C}$ is semisimple, any $\cC$-equivariant structure is uniquely determined by its values on $\Irr(\cC)$. In particular, for any cocycle morphism, the family of maps $\{\mathbbm{v}_X\}_{X\in\cC}$ is uniquely determined by the family of maps $\{\mathbbm{v}_X\}_{X\in\Irr(\cC)}$. 
\end{rmk}

In the case of group actions, Definition \ref{cocyclemorphism} recovers Szabó's notion of a cocycle morphism (see Definition \ref{def: cocyclemorGroups}).

\begin{example}\label{example: intertHilb(G)}
Suppose $(A,\alpha)$ and $(B,\beta)$ are actions of a countable discrete group $\Gamma$ on $\C$-algebras $A$ and $B$. Consider them as actions of $\Hilb(\Gamma)$ as in Example \ref{exn:groupaction}. Let $(\phi,\mathbbm{v}):(A,\alpha)\to (B,\beta)$ be an extendible cocycle morphism as in Definition \ref{cocyclemorphism} with $\mathbbm{v}_{\mathbb{C}g}$ being adjointable for all $g\in \Gamma$. Fix $g\in \Gamma$ and let $f_g:=T_g'\circ \mathbbm{v}_{\mathbb{C}_g}\circ S_g: {}_{\phi\circ\alpha_g}B\to {}_{\beta_g\circ\phi}B$, where $S_g:{}_{\phi\circ\alpha_g}B\to {}_{\alpha_g}A\boxtimes {}_{\phi}B $ is given by $S_g(b)=\lim\limits_{\lambda}\xi_\lambda\boxtimes b$ for any $b\in B$, where $\xi_\lambda$ is an approximate unit for $A$ and $T_g':{}_{\phi}B\boxtimes {}_{\beta_g}B\rightarrow {}_{\beta_g\circ\phi}B$ is given by $T_g'(b\boxtimes c)=\beta_g(b)c$ for $b,c\in B$.\footnote{Note that $S_g$ is well-defined by Proposition \ref{extendiblecase}.}\ Moreover, let $T_g:{}_{\alpha_g}A\boxtimes {}_{\phi}B\rightarrow {}_{\phi\circ\alpha_g}B$ be given by $T_g(a\boxtimes b)=\phi(a)b$ for any $a\in A$ and any $b\in B$.

Since $f_g$ is an adjointable bimodule map, it follows that $f_g(b)=\mathbbm{u}_gb$ for any $b\in B$, for some $\mathbbm{u}_g\in \mathcal{M}(B)$.\footnote{Note that $T_g$ is adjointable as $\phi$ is extendible (see Proposition \ref{extendiblecase}). Thus, $f_g$ is a composition of adjointable maps.}\ In particular, the equality $f_g(a\rhd b)=a\rhd f_g(b)$ implies that $$\beta_g(\phi(a))\mathbbm{u}_g=\mathbbm{u}_g\phi(\alpha_g(a))$$ for all $a\in A$. This gives \eqref{cocyclerep1}, although $\mathbbm{u}_g$ might not be a unitary.

As $S_g\circ T_g$ is the identity map, it follows that the diagram
\begin{equation}\label{eq: IntertwinersDiagramGroup}
\begin{tikzcd}
{}_{\alpha_g}A\boxtimes {}_{\phi}B
\arrow[swap]{d}{T_g}
\arrow{r}{\mathbbm{v}_{\mathbb{C}_g}}
& {}_{\phi}B\boxtimes {}_{\beta_g}B
 \\
{}_{\phi\circ\alpha_g}B
\arrow{r}{f_g}
& {}_{\beta_g\circ\phi}B\arrow{u}{(T_g')^{-1}}
\end{tikzcd}
\end{equation}commutes. Hence, $$\mathbbm{v}_{\mathbb{C}_g}(a\boxtimes b)=\lim\limits_{\lambda}\eta_\lambda\boxtimes \mathbbm{u}_g\phi(a)b,$$ for any $a\in A, b\in B$ and $\eta_\lambda$ an approximate unit for $B$. Following the pentagon diagram for $\mathbbm{v}$, one gets \eqref{cocyclerep2}. 

Conversely, if $(\phi,\mathbbm{u}):(A,\alpha)\to (B,\beta)$ is a cocycle morphism as in Definition \ref{def: cocyclemorGroups}, then define $f_g:{}_{\phi\circ\alpha_g}B\rightarrow {}_{\beta_g\circ\phi}B$ by $f_g(b)=\mathbbm{u}_gb$ for all $b\in B$ and $\mathbbm{v}_{\mathbb{C}_g}$ be given by \eqref{eq: IntertwinersDiagramGroup}. Note that \eqref{cocyclerep1} implies that $\mathbbm{v}_{\mathbb{C}_g}$ is a bimodule map, while \eqref{cocyclerep2} gives the pentagon diagram for $\mathbbm{v}_{\mathbb{C}_g}$. Hence $(\phi, \mathbbm{v})$ yields a cocycle morphism in the sense of Definition \ref{cocyclemorphism}.
\end{example}

\begin{rmk}
Note that, unlike in {\cite[Definition 1.10]{cocyclecategszabo}}, we do not require $\phi\in \Corr(A,B)$ to be extendible. In fact, this is precisely the reason why we consider the maps $\mathbbm{v}_X$ to be isometries (possibly non-adjointable) instead of unitaries. For example, the map $\mathbbm{v}_{1_{\cC}}:A\boxtimes{}_{\phi}B\to{}_{\phi}B\boxtimes B\cong{}_{\phi}B$ is given by $\mathbbm{v}_{1_{\cC}}(a\boxtimes b)=\phi(a)b$ for all $a\in A$ and $b\in B$. Therefore, it is not surjective unless $\phi$ is non-degenerate.\ Moreover, it might not be adjointable if $\phi$ is not extendible as seen in the proof of Proposition \ref{extendiblecase}.

Furthermore, the morphisms of Definition \ref{cocyclemorphism} fit into the $\cC$-equivariant $\KK$-theory developed in \cite{EqKK} (see \cite[Example 3.3]{EqKK}). In \cite{EqKK} a correspondence morphism is instead called a \emph{$\cC$-Hilbert $A$-$B$-bimodule} and a cocycle morphism is called a \emph{cocycle-$\cC$-$*$-homomorphism}.
\end{rmk}

We now define composition formulae for the various notions of morphisms in Definition \ref{cocyclemorphism}. Using the standard composition between correspondences, we can define composition of correspondence morphisms in the obvious way. Let $\mathcal{C}\overset{F}{\curvearrowright} A$, $\mathcal{C}\overset{G}{\curvearrowright} B$, and $\mathcal{C}\overset{H}{\curvearrowright} C$  be actions of $\mathcal{C}$ on $\C$-algebras $A$, $B$, and $C$ respectively. If $(\phi,\mathbbm{v}):(A,F,J)\to (B,G,I)$ and $(\psi,\mathbbm{w}):(B,G,I)\to (C,H,K)$ are correspondence morphisms, their composition is denoted by $(\psi\circ\phi, \mathbbm{w}\circ\mathbbm{v})$, where 
\begin{equation}\label{eq: CompCorrMor}
(\mathbbm{w}\circ\mathbbm{v})_X=(\id_{\phi}\boxtimes\mathbbm{w}_X)\circ(\mathbbm{v}_X\boxtimes\id_{\psi}).
\end{equation} By combining the pentagon diagrams for $\mathbbm{v}$ and $\mathbbm{w}$, one obtains that $(\psi\circ\phi,\mathbbm{w}\circ\mathbbm{v}):(A,F,J)\to (C,H,K)$ is indeed a correspondence morphism.

However, if $\phi$ and $\psi$ are possibly degenerate cocycle morphisms, the composition formula above will not give a cocycle morphism.\ This problem arises as the bimodule ${}_\phi B\boxtimes {}_\psi C$ may no longer be isomorphic to $C$ as a right $C$ module. Therefore, to form a category, we introduce a slightly different composition on cocycle morphisms.

\begin{defn}\label{def: CompCocMor}
Let $(\phi,\mathbbm{v}):(A,F,J)\to (B,G,I)$ and $(\psi,\mathbbm{w}):(B,G,I)\to (C,H,K)$ be cocycle morphisms. Let $\mathbbm{w}*\mathbbm{v}$ be the collection of isometries $\{(\mathbbm{w}*\mathbbm{v})_X\}_{X\in\cC}$ given by
\begin{equation}\label{eq: CompCocMor}
\begin{tikzcd}
F(X)\boxtimes {}_{\psi\circ\phi}C
\arrow[swap]{d}{S_X}
\arrow{r}{(\mathbbm{w}*\mathbbm{v})_X}
& {}_{\psi\circ\phi}C\boxtimes H(X)
 \\
F(X)\boxtimes {}_{\phi}B\boxtimes {}_{\psi}C
\arrow{r}{(\mathbbm{w}\circ\mathbbm{v})_X}
& {}_{\phi}B\boxtimes{}_{\psi}C\boxtimes H(X)\arrow{u}{T\boxtimes\id_{H(X)}}.
\end{tikzcd}
\end{equation} Here $S_X(x\boxtimes c)=\lim\limits_{\lambda}x\boxtimes\eta_\lambda\boxtimes c$ and $T(b\boxtimes c)=\psi(b)c$ for any $X\in\cC$, $x\in F(X)$, $b\in B$, and $c\in C$, with $\eta_\lambda$ being an approximate unit for $B$.\footnote{We have not shown that $S_X$ is well-defined at this point.}
\end{defn}

\begin{lemma}\label{lemma: SWellDef}
The continuous linear map $S_X: F(X)\boxtimes {}_{\psi\circ\phi}C\to F(X)\boxtimes{}_{\phi}B\boxtimes{}_{\psi}C$ given by $S_X(x\boxtimes c)=\lim\limits_{\lambda}x\boxtimes\eta_\lambda\boxtimes c$ for any $x\in F(X)$, any $c\in C$, and some approximate unit $\eta_\lambda$ of $B$ is a well-defined isometric bimodule isomorphism for any $X\in\cC$.
\end{lemma}

\begin{proof}
We will show that for any $x\in F(X)$ and any $c\in C$, the net $x\boxtimes\eta_\lambda\boxtimes c$ is Cauchy with respect to the norm induced by the right inner product. By definition, we have that $\langle x\boxtimes(\eta_\lambda-\eta_\mu)\boxtimes c, x\boxtimes(\eta_\lambda-\eta_\mu)\boxtimes c\rangle_C= \langle c, \langle x\boxtimes(\eta_\lambda-\eta_\mu), x\boxtimes(\eta_\lambda-\eta_\mu)\rangle_B\rhd c\rangle_C.$ Then, a direct computation shows that 
\begin{align*}
\langle x\boxtimes(\eta_\lambda-\eta_\mu), x\boxtimes(\eta_\lambda-\eta_\mu)\rangle_B &= \langle \eta_\lambda-\eta_\mu, \langle x,x\rangle_A\rhd (\eta_\lambda-\eta_\mu)\rangle_B \\ &= (\eta_\lambda-\eta_\mu)\phi(\langle x,x\rangle_A)(\eta_\lambda-\eta_\mu).
\end{align*} Therefore, $$\langle x\boxtimes(\eta_\lambda-\eta_\mu)\boxtimes c, x\boxtimes(\eta_\lambda-\eta_\mu)\boxtimes c\rangle_C=\langle c, \psi((\eta_\lambda-\eta_\mu)\phi(\langle x,x\rangle_A)(\eta_\lambda-\eta_\mu))c\rangle_C.$$ By the Cauchy-Schwarz inequality and since $(\eta_\lambda-\eta_\mu)\phi(\langle x,x\rangle_A)^{1/2}$ converges to $0$, it is readily seen that $\langle x\boxtimes(\eta_\lambda-\eta_\mu)\boxtimes c, x\boxtimes(\eta_\lambda-\eta_\mu)\boxtimes c\rangle_C$ converges to $0$, so $S_X$ is well-defined. Moreover, commutation with the left and right actions are immediate.

Let $T_X:F(X)\boxtimes{}_{\phi}B\boxtimes{}_{\psi}C\to F(X)\boxtimes {}_{\psi\circ\phi}C $ be the continuous linear map given by $T_X(x\boxtimes b\boxtimes c)=x\boxtimes \psi(b)c$. Then, 
\begin{align*}
S_X(T_X(x\boxtimes b\boxtimes c)) &= S_X(x\boxtimes \psi(b)c) \\ &= \lim\limits_{\lambda}x\boxtimes\eta_\lambda\boxtimes \psi(b)c \\ &= \lim\limits_{\lambda}x\boxtimes\eta_\lambda\boxtimes b\rhd c \\ &= \lim\limits_{\lambda}x\boxtimes\eta_\lambda\lhd b\boxtimes c \\ &= x\boxtimes b\boxtimes c. 
\end{align*}

Therefore, by linearity and continuity of both $S_X$ and $T_X$, it follows that $S_X$ is surjective and $T_X$ is injective. It now suffices to show that $T_X$ is surjective. This will imply that $S_X$ is invertible with $T_X$ being the inverse. Let $x\in F(X)$ and $c\in C$. Since $F(X)$ is a non-degenerate $A$-bimodule, the map $R_X^{-1}:F(X)\boxtimes A\to F(X)$ given by $R_X^{-1}(x\boxtimes a)=x\lhd a$ is a bimodule isomorphism (see the discussion above Lemma \ref{lemma:ActionMultiplierAlg}). Then, let $y\in F(X)$ and $a\in A$ such that $y\lhd a=x$.\ A straightforward calculation shows that $T_X(y\boxtimes \phi(a)\boxtimes c)=x\boxtimes c$, so $T_X$ is surjective. Hence, $S_X=T_X^{-1}$ is an isomorphism for any $X\in\cC$. 

Finally, for any $X\in\cC$, $T_X=\id_{F(X)}\boxtimes T$, where $T:{}_{\phi}B\boxtimes{}_{\psi}C\to {}_{\psi\circ\phi}C$ is given by $T(b\boxtimes c)=\psi(b)c$ for any $b\in B$ and $c\in C$. Since $T$ is an isometry, we conclude that $T_X$, and hence $S_X$ are isometric maps.
\end{proof}

Note that the proof of Lemma \ref{lemma: SWellDef} also shows that the map $S_X$ does not depend on the choice of approximate unit. The proof of the following Lemma is routine.

\begin{lemma}\label{lemma: NewCompExtCocRep}
Let $(\phi,\mathbbm{v}):(A,F,J)\to (B,G,I)$ and $(\psi,\mathbbm{w}):(B,G,I)\to (C,H,K)$ be cocycle morphisms. Then $(\psi\circ\phi, \mathbbm{w}*\mathbbm{v})$ is a cocycle morphism.
\end{lemma}

\begin{rmk}\label{rmk:NewComp}
We invite the reader to recall Remark~\ref{rmk: extendible}.
Then, with the same notation as in Definition \ref{def: CompCocMor}, if $(\phi,\mathbbm{v})$ and $(\psi,\mathbbm{w})$ are cocycle representations and $\psi$ is extendible, then it also follows that $(\psi\circ\phi, \mathbbm{w}*\mathbbm{v})$ is a well-defined cocycle representation.
\end{rmk}

\begin{rmk}
If the cocycle morphisms in Definition \ref{def: CompCocMor} are assumed to be non-degenerate, then the maps $S_X$ and $T$ are bimodule isomorphisms for any $X\in\cC$. Therefore, the composition in Definition \ref{def: CompCocMor} (denoted $*$) corresponds canonically to the composition of correspondence morphisms (denoted $\circ$).
\end{rmk}

The composition considered in Definition \ref{def: CompCocMor} defines a category. To show this we first reformulate the notion of cocycle morphism. Roughly speaking, all the information carried by the collection of isometries $\{\mathbbm{v}_X\}_{X\in \cC}$ can be encoded into a collection of linear maps $\{h^X:F(X)\to G(X)\}_{X\in\mathcal{C}}$ satisfying some conditions. This viewpoint will facilitate our constructions and proofs in later sections. Our approach is motivated by {\cite[Lemma 3.8]{limitfusioncategclassif}} which introduces this alternative viewpoint in the unital setting.
\par First, we recall a way of extending an action $(F,J)$ of a C$^*$-tensor category $\cC$ on $A$ to its matrix amplification $M_n(A)$.\ Consider the functor
$F^{\amp(n)}:\cC^{\rev}\rightarrow \Corr(M_n(A))$ that maps objects $X\in \cC$ to the correspondence with underlying bimodule $F(X)\otimes M_n(\mathbb{C})$ with the right inner product defined by $\langle (x_{ij}),(y_{ij})\rangle_{M_n(A)}=(\sum_l \langle x_{li},y_{lj}\rangle_A)$, right $M_n(A)$ action given by $(x_{ij})\lhd (a_{ij})=(\sum_l x_{il}\lhd a_{lj})$ and with the left action given by $(a_{ij})\rhd(x_{ij})=(\sum_l a_{il}\rhd x_{lj})$ for $a_{ij}\in A$ and $x_{ij}, y_{ij} \in F(X)$ for $1\leq i,j\leq n$.\footnote{This bimodule can be identified with the external tensor product of $F(X)$ and $M_n(\mathbb{C})$. Hence the inner product defines a Hilbert $M_n(A)$-bimodule (see \cite{Hilbertmodules}).}\ For a morphism $T\in \Hom(X,Y)$ we let $F^{\amp(n)}(T)=T\otimes \id_{M_n(\mathbb{C})}$.\ Moreover, letting $J_{X,Y}^{\amp(n)}((y_{ij})\boxtimes (x_{ij}))=(\sum_l J_{X,Y}(y_{il}\boxtimes x_{lj}))$ for $x_{ij}\in F(X)$ and $y_{ij}\in F(Y)$ for $1\leq i,j\leq n$, it is a straightforward calculation that $(F^{\amp(n)},J^{\amp(n)})$ is an action of $\cC$ on $M_n(A)$.
\begin{lemma}\label{linearmapspicture}
Let $\cC$ be a $\C$-tensor category acting on $\C$-algebras $A$ and $B$ via $(A,F,J)$ and $(B,G,I)$ respectively, and let $\phi: A\rightarrow B$ be a $^*$-homomorphism.\ Then there is a bijection between the families $\{\mathbbm{v}_X\}_{X\in\cC}$ corresponding to a cocycle morphism $(\phi,\mathbbm{v}):(A,F,J)\to (B,G,I)$ and families of linear maps: 
$$\{h^{X}: F(X)\rightarrow G(X)\}_{X\in\cC}$$ such that for any $X,Y\in\cC$
\begin{enumerate}[label=\textit{(\roman*)}]
\item $h^{X}(a\rhd x \lhd a')=\phi(a)\rhd h^{X}(x)\lhd \phi(a')$ for any $a,a'\in A$;\label{item:linearmap1}
\item
for any morphism $f\in \Hom(X,Y)$, $G(f)\circ h^{X}=h^{Y}\circ F(f)$;\label{item:linearmap2}
\item
$\phi(\langle x , y\rangle_A)=\langle h^{X}(x) , h^{X}(y)\rangle_B$ for any $x,y\in F(X)$; \label{item:linearmap3}
\item the diagram:\label{item:linearmap4}
\[
\begin{tikzcd}
F(Y)\boxtimes F(X)
\arrow[swap]{d}{h^{Y}\boxtimes h^{X}}
\arrow{r}{J_{X,Y}}
& F(X\otimes Y)\arrow{d}{h^{X\otimes Y}}
 \\
G(Y)\boxtimes G(X)
\arrow{r}{I_{X,Y}}
& G(X\otimes Y)
\end{tikzcd}
\]
commutes;
\item $h^{1_{\cC}}:A\to B$ is given by $h^{1_{\cC}}(a)=\phi(a)$ for any $a\in A$.\label{item:linearmap5}
\end{enumerate}  
\end{lemma}

\begin{proof}
Suppose we are given a collection of linear maps $\{h^X\}$ satisfying the conditions listed above. Fix $\zeta_\lambda$ an approximate unit for $A$. For all $X\in \cC$ let $\mathbbm{v}_X:F(X)\boxtimes {}_{\phi}B\to {}_{\phi}B\boxtimes G(X)$ be given by
\begin{equation}\label{eq: fromHtoV}
\mathbbm{v}_X(x\boxtimes b)=\lim\limits_{\lambda}\phi(\zeta_\lambda)\boxtimes h^X(x)\lhd b, \quad x\in F(X),\ b\in B.
\end{equation} 
First, we need to show that $\mathbbm{v}_X$ is well-defined. It suffices to show that the net $\phi(\zeta_\lambda)\boxtimes h^X(x)\lhd b$ is Cauchy. A straightforward computation using the definition of the inner product and \ref{item:linearmap1} gives that 
\begin{align}
\|\phi(\zeta_\lambda-\zeta_\mu)\boxtimes h^X(x)\lhd b\|^2 &= \|\phi(\zeta_\lambda-\zeta_\mu)\rhd h^X(x)\lhd b\|^2\nonumber\\ &= \|h^X(\zeta_\lambda\rhd x-\zeta_\mu\rhd x)\lhd b\|^2.\label{eqn:normcond}
\end{align} Since the bimodule $F(X)$ is non-degenerate, denoting the left action by $\sigma:A\to\mathcal{L}(F(X))$, we have that $\sigma(\zeta_\lambda)$ converges strictly to $1_{\mathcal{L}(F(X))}$. Then, for any $x\in F(X)$, $\zeta_\lambda\rhd x=\sigma(\zeta_\lambda)x$ converges to $x$. Moreover, condition \ref{item:linearmap3} implies that $h^X$ is continuous, so the right hand side of  \eqref{eqn:normcond} converges to zero and hence $\|\phi(\zeta_\lambda-\zeta_\mu)\boxtimes h^X(x)\lhd b\|^2$ converges to $0$. Therefore the formula in \eqref{eq: fromHtoV} gives a well-defined map $\mathbbm{v}_X$ for all $X\in \cC$. 

The maps $\mathbbm{v}_X$ are linear by linearity of $h^X$ and naturality follows by using naturality of the family $\{h^X\}$ given by condition \ref{item:linearmap2}. It is straightforward to see that $\mathbbm{v}_X$ commutes with the right $B$-action, and for any $a\in A$, $b\in B$, $X\in \cC$ and $x\in F(X)$
\begin{align*}
\mathbbm{v}_X(a\rhd (x\boxtimes b))&=\lim\limits_{\lambda}\phi(\zeta_\lambda)\boxtimes h^X(a\rhd x)\lhd b \\ &= \lim\limits_{\lambda}\phi(\zeta_\lambda)\boxtimes \phi(a)\rhd h^X(x)\lhd b \\ &= \lim\limits_{\lambda}\phi(\zeta_\lambda)\phi(a)\boxtimes h^X(x)\lhd b \\ &= \phi(a)\boxtimes h^X(x)\lhd b \\ &= a\rhd \mathbbm{v}_X(x\boxtimes b).
\end{align*} Therefore, $\mathbbm{v}_X$ is an $A$-$B$-bimodule map. Moreover, using \ref{item:linearmap5}, \begin{align*}
\mathbbm{v}_{1_{\cC}}(a\boxtimes b) &=\lim\limits_{\lambda}\phi(\zeta_\lambda)\boxtimes \phi(a) b \\ &= \lim\limits_{\lambda}\phi(\zeta_\lambda)\phi(a)\boxtimes b\\ &= \phi(a)\boxtimes b,\ \forall a\in A,\ b\in B.
\end{align*}

To prove that $(\phi,\mathbbm{v})$ defines a cocycle morphism, it remains to show that each map $\mathbbm{v}_X$ is an isometry and the family $\{\mathbbm{v}_X\}$ is such that diagram (\ref{cocyclemorphismdiagram}) commutes. Let us first show that each map $\mathbbm{v}_X$ is an isometry. For any $x\in F(X)$ and any $b\in B$, using \ref{item:linearmap1}, we have
\begin{align*}
|\mathbbm{v}_X(x\boxtimes b)|^2 &= \lim\limits_{\lambda}\langle\phi(\zeta_\lambda)\rhd h^X(x)\lhd b, \phi(\zeta_\lambda)\rhd h^X(x)\lhd b\rangle \\ &= \lim\limits_{\lambda}\langle h^X(\zeta_\lambda\rhd x)\lhd b, h^X(\zeta_\lambda\rhd x)\lhd b\rangle \\ &= \lim\limits_{\lambda} b^* \langle h^X(\zeta_\lambda\rhd x), h^X(\zeta_\lambda\rhd x)\rangle b.
\end{align*}
On the other hand, \ref{item:linearmap3} yields that
\begin{align*}
|x\boxtimes b|^2 &= \langle b, \langle x,x\rangle_A\rhd b\rangle \\ &= \langle \phi(|x|)b, \phi(|x|)b\rangle \\ &= b^*\phi(\langle x,x\rangle_A)b\\ &= b^*\langle h^X(x),h^X(x)\rangle_Bb.
\end{align*} Similarly, as $h^X$ is continuous by \ref{item:linearmap3} and $F(X)$ is non-degenerate, it follows that $h^X(\zeta_\lambda\rhd x)$ converges to $h^X(x)$, which shows that $\mathbbm{v}_X$ is an isometry when restricted to elementary tensors. To show that $\mathbbm{v}_X$ acts as an isometry on sums of the form $\sum_{i=1}^n x_i\boxtimes b_i$ for $x_i\in F(X)$, $b_i\in B$ and $X\in \cC$ we consider the amplified actions $(F^{\amp(n)},J^{\amp(n)})$ and $(G^{\amp(n)},I^{\amp(n)})$ on $M_n(A)$ and $M_n(B)$ respectively.\ It follows from a direct computation that the family of linear maps $h^{X,\amp(n)}:F^{\amp(n)}(X)\rightarrow G^{\amp(n)}(X)$ defined by $(x_{ij})\mapsto (h^X(x_{ij}))$ for $X\in \cC$ and $x_{ij}\in F(X)$ satisfies conditions \ref{item:linearmap1}-\ref{item:linearmap5} with the amplified homomorphism $\phi:M_n(A)\rightarrow M_n(B)$. Therefore $\mathbbm{v}_X^{\amp(n)}$ defined as in (\ref{eq: fromHtoV}) but instead with the pair $(\phi,h^{X,\amp(n)})$ is an isometry when restricted to elementary tensors.\ Choose $\mathbf{X}$ in $F^{\amp(n)}(X)$ with first row given by the vector $(x_1,x_2,\ldots ,x_n)$ and zero elsewhere and $\mathbf{B}$ in $M_n(B)$ have first column $(b_1,b_2,\ldots, b_n)$ and zero elsewhere. Now, by definition
\begin{align*}
    \|\mathbf{X}\boxtimes \mathbf{B}\|^2&=\|\langle \mathbf{B},\langle \mathbf{X},\mathbf{X}\rangle_{M_n(A)}\rhd \mathbf{B}\rangle\|\\
    &=\|\sum_{i=1}^n x_i\boxtimes b_i\|^2
\end{align*}
and similarly through a direct computation
$$\|\mathbbm{v}_X^{\amp(n)}(\mathbf{X}\boxtimes \mathbf{B})\|^2=\|\mathbbm{v}_X(\sum_{i=1}^n x_i\boxtimes b_i)\|^2.$$
As $\mathbbm{v}_X^{\amp(n)}$ is an isometry when restricted to elementary tensors, it follows that $\mathbbm{v}_X$ is an isometry.
\par It remains to check that the diagram\\
\begin{equation}\label{diag:prooflinmaps}
\begin{adjustbox}{max width=\textwidth}
\begin{tikzcd}
& F(Y)\boxtimes F(X)\boxtimes{}_{\phi}B 
\ar{dr}{J_{X,Y}\boxtimes\id_{B}}
\ar[swap]{dl}{\id_{F(Y)}\boxtimes\mathbbm{v}_X} &    
\\   
F(Y)\boxtimes {}_{\phi}B\boxtimes G(X) \ar[swap]{dd}{\mathbbm{v}_Y\boxtimes\id_{G(X)}} 
& &F(X\otimes Y)\boxtimes {}_{\phi}B\ar{dd}{\mathbbm{v}_{X\otimes Y}}
\\
& &
\\
{}_{\phi}B\boxtimes G(Y)\boxtimes G(X)
\ar{rr}{\id_{B}\boxtimes I_{X,Y}} 
&    
& {}_{\phi}B\boxtimes G(X\otimes Y).
\end{tikzcd}
\end{adjustbox}
\end{equation}
commutes for all $X,Y\in \cC$. 
\par Starting with an elementary tensor $y\boxtimes x\boxtimes b$ with $X,Y\in \cC,y\in F(Y), x\in F(X)$ and $b\in B$ and following the two rightmost maps of the diagram, we get that $$y\boxtimes x\boxtimes b\mapsto J_{X,Y}(y\boxtimes x)\boxtimes b\mapsto \lim\limits_{\lambda}\phi(\zeta_\lambda)\boxtimes h^{X\otimes Y}(J_{X,Y}(y\boxtimes x))\lhd b. $$ Moreover, using that the family of linear maps satisfies condition \ref{item:linearmap4}, this composition coincides with the mapping $$y\boxtimes x\boxtimes b\mapsto\lim\limits_{\lambda}\phi(\zeta_\lambda)\boxtimes I_{X,Y}(h^Y(y)\boxtimes h^X(x))\lhd b. $$
Again, starting with $y\boxtimes x\boxtimes b$ but now following the three leftmost arrows in diagram \eqref{diag:prooflinmaps} we get 
\begin{align*}
y\boxtimes x\boxtimes b &\mapsto \lim\limits_{\lambda}y\boxtimes \phi(\zeta_\lambda)\boxtimes h^X(x)\lhd b \\ &\mapsto \lim\limits_{\lambda}\lim\limits_{\mu} \phi(\zeta_\mu)\boxtimes h^Y(y)\lhd \phi(\zeta_\lambda)\boxtimes h^X(x)\lhd b \\ &= \lim\limits_{\mu} \phi(\zeta_\mu)\boxtimes h^Y(y)\boxtimes h^X(x)\lhd b \\ &\mapsto \lim\limits_{\mu} \phi(\zeta_\mu)\boxtimes I_{X,Y}(h^Y(y)\boxtimes h^X(x)\lhd b)\\ &= \lim\limits_{\mu} \phi(\zeta_\mu)\boxtimes I_{X,Y}(h^Y(y)\boxtimes h^X(x))\lhd b,
\end{align*} where the first equality holds since $h^Y(y)\lhd\phi(\zeta_\lambda)=h^Y(y\lhd\zeta_\lambda)$ converges to $h^Y(y)$. So \eqref{diag:prooflinmaps} commutes and $(\phi,\mathbbm{v})$ is a cocycle morphism.

Now, consider the map $\Psi:\{h^X\}\to\{\mathbbm{v}_X\}$ given by the formula in \eqref{eq: fromHtoV}. We claim that $\Psi$ is independent of the choice of approximate unit. Indeed let $\zeta_\lambda$ and $\xi_\lambda$ be two approximate units for $A$. Similarly as in (\ref{eqn:normcond}) we have that, $\|\phi(\zeta_\lambda-\xi_\lambda)\boxtimes h^X(x)\lhd b\|=\|h^X(\zeta_\lambda\rhd x-\xi_\lambda\rhd x)\lhd b\|$, which converges to $0$. Hence, the map $\mathbbm{v}_X$ is independent of the choice of approximate unit, and so $\Psi$ is well-defined.

Conversely, suppose we have a cocycle morphism $(\phi,\mathbbm{v}):(A,F,J)\to (B,G,I)$ and for each $X\in\cC$ let $h^X:F(X)\to G(X)$ be given by 
\begin{equation}\label{hcomp}
\begin{tikzcd}
F(X)\ar{r}{\iota}& F(X)\boxtimes{}_{\phi}B\ar{r}{\mathbbm{v}_X}& {}_{\phi}B\boxtimes G(X)\ar{r}{f} & G(X),  
\end{tikzcd}
\end{equation} where $\iota(x)=\lim_{\lambda}x\boxtimes\eta_\lambda$ for some approximate unit $\eta_\lambda$ of $B$ and all $x\in F(X)$, and $f$ is the map given by $f(b\boxtimes y)=b\rhd y$ for all $b\in B$ and $y\in G(X)$. Note that $f$ is an $A$-$B$-bimodule isomorphism if we see $G(X)$ as a left $A$-module through $\phi$ (i.e. $\phi(a)\rhd_B f(y)=a\rhd_A f(y)$ for all $y\in {}_\phi B\boxtimes G(X)$ and $a\in A$).

To check that $\iota$ is well-defined, we show that the net $x\boxtimes\eta_\lambda$ is Cauchy for all $X\in \cC$ and $x\in F(X)$. Precisely, one has that $\langle x\boxtimes(\eta_\lambda-\eta_\mu),x\boxtimes(\eta_\lambda-\eta_\mu)\rangle_B=\langle \eta_\lambda-\eta_\mu, \langle x,x\rangle_A\rhd(\eta_\lambda-\eta_\mu)\rangle_B= |\phi(\langle x,x\rangle_A)^{1/2}(\eta_\lambda-\eta_\mu)|^2$. This converges to $0$ since the image of $\phi$ is contained in $B$ and $\eta_\lambda$ is an approximate unit for $B$.

We now check that the family $\{h^X\}$ defined above satisfies the required compatibility conditions. Since each of the maps in \eqref{hcomp} are linear, we get that $h^X$ is linear. To see \ref{item:linearmap1}, note that $f$, $\mathbbm{v}_X$ and $\iota$ are left module maps so $\phi(a)\rhd h^X(x)=h^X(a\rhd x)$ for all $a\in A$ and $x\in F(X)$. Moreover, as $f$ and $\mathbbm{v}_X$ are right $B$-module maps and $\iota$ is a right $A$-module map
\begin{align*}
h^X(x)\lhd \phi(a)&=\lim\limits_{\lambda} f(\mathbbm{v}_X(x\boxtimes \eta_\lambda\phi(a)))=h^X(x\lhd a).
\end{align*} Hence, $h^X$ satisfies \ref{item:linearmap1}. It is straightforward to see that $h^X$ satisfies \ref{item:linearmap2} by naturality of $\mathbbm{v}$.

Note that $\mathbbm{v}_X$ and $f$ are isometries. So one has that for any $x\in F(X)$, $\langle h^{X}(x) , h^{X}(x)\rangle_B=\langle\iota(x),\iota(x)\rangle_B=\lim_{\lambda}\langle \eta_\lambda,\phi(\langle x,x\rangle_A)\eta_\lambda\rangle_B= \phi(\langle x,x\rangle_A)$ and \ref{item:linearmap3} follows from the polarisation identity. Condition \ref{item:linearmap4} follows from the fact that the maps $\mathbbm{v}_X$ satisfy the diagram in \eqref{cocyclemorphismdiagram}. Finally,
$$h^{1_{\cC}}(a) = \lim\limits_{\lambda} f(\mathbbm{v}_{1_{\cC}}(a\boxtimes \eta_\lambda)) = \lim\limits_{\lambda} f(\phi(a)\boxtimes\eta_\lambda) =\phi(a).$$

Now, let $\Phi:\{\mathbbm{v}_X\}\to\{h^X\}$ be the map induced by the formula in \eqref{hcomp}. Note that $\iota$ and hence $\Phi$ is independent of the choice of approximate unit. Indeed, let $\eta_\lambda$ and $\xi_\lambda$ be two approximate units for $B$. We show that the net $x\boxtimes(\eta_\lambda-\xi_\lambda)$ converges to $0$ for any $x\in F(X)$. Note that 
\begin{align*}
    |x\boxtimes(\eta_\lambda-\xi_\lambda)|^2&=\langle (\eta_\lambda-\xi_\lambda), \langle x, x\rangle_A\rhd(\eta_\lambda-\xi_\lambda)\rangle\\ &= |\phi(\langle x,x\rangle_A)^{1/2}(\eta_\lambda-\xi_\lambda)|^2,
\end{align*} which converges to $0$.
\par We claim that $\Phi$ and $\Psi$ are inverses to each other. First we show that $\Phi\circ\Psi$ is the identity map. For any $X\in\cC$ and $x\in F(X)$, it follows that

\begin{adjustbox}{max width=\textwidth}
$\Phi(\Psi(h^X))(x)=f\Big(\Psi(h^X)\Big(\lim\limits_{\lambda}x\boxtimes\eta_\lambda\Big)\Big)=f\Big(\lim\limits_{\mu}\lim\limits_{\lambda}\phi(\zeta_\mu)\boxtimes h^X(x)\lhd\eta_\lambda\Big).$
\end{adjustbox} 

Since $\eta_\lambda$ is an approximate unit for $B$ and $G(X)$ is non-degenerate, it follows that $$\Phi(\Psi(h^X))(x)=f\Big(\lim\limits_{\mu}\phi(\zeta_\mu)\boxtimes h^X(x)\Big)=\lim\limits_{\mu}\phi(\zeta_\mu)\rhd h^X(x).$$ As $h^X$ satisfies condition \ref{item:linearmap1}, $\Phi(\Psi(h^X))(x)=\lim\limits_{\mu}h^X(\zeta_\mu\rhd x)$. Thus, it suffices to show that $\|h^X(\zeta_\mu\rhd x-x)\|\to 0$. This follows by continuity of $h^X$ and that $F(X)$ is non-degenerate. 

To prove that $\Psi\circ\Phi$ is the identity map, note that
\begin{align*}
\Psi(\Phi(\mathbbm{v}_X))(x\boxtimes b)&=\lim\limits_{\mu}\phi(\zeta_\mu)\boxtimes\Phi(\mathbbm{v}_X)(x)\lhd b\\ &=\lim\limits_{\mu}\phi(\zeta_\mu)\boxtimes f\Big(\mathbbm{v}_X\Big(\lim\limits_{\lambda}x\boxtimes \eta_\lambda\Big)\Big)\lhd b\\&= \lim\limits_{\mu}\phi(\zeta_\mu)\boxtimes f\Big(\mathbbm{v}_X\Big(\lim\limits_{\lambda}x\boxtimes \eta_\lambda b\Big)\Big)\\&=\lim\limits_{\mu}\phi(\zeta_\mu)\boxtimes f(\mathbbm{v}_X(x\boxtimes b)).
\end{align*}
Applying $f$ to both sides
\begin{align*}
f(\Psi(\Phi(\mathbbm{v}_X))(x\boxtimes b))&=\lim\limits_{\mu}\phi(\zeta_\mu)\rhd_B f(\mathbbm{v}_X(x\boxtimes b))\\&= \lim\limits_{\mu}\zeta_\mu\rhd_A f(\mathbbm{v}_X(x\boxtimes b))\\
&=\lim\limits_{\mu} f(\mathbbm{v}_X((\zeta_\mu\rhd x)\boxtimes b))\\
&=f(\mathbbm{v}_X(x\boxtimes b))
\end{align*}
as $\zeta_\mu\rhd x$ converges to $x$. Hence, we reach the conclusion by composing with $f^{-1}: G(X)\to {}_{\phi}B\boxtimes G(X)$ given by $f^{-1}(x)=\lim\limits_{\lambda} \eta_\lambda\boxtimes x$. 
\end{proof}

Note that this alternative picture only holds for cocycle morphisms. In the generality of cocycle representations, the maps $h^X$ may not be well-defined. That is because if $(\phi,\mathbbm{v})$ is a cocycle representation, $\phi$ can land in $M(B)\setminus B$. As $\eta_\lambda$ is an approximate unit for $B$, $\|\phi(\langle x,x\rangle_A)^{1/2}(\eta_\lambda-\eta_\mu)\|^2$ need not converge to $0$.
\begin{rmk}
    Note that condition \ref{item:linearmap1} follows from \ref{item:linearmap4} as $J_{1,X}$ and $J_{X,1}$ correspond to the left and right actions of $A$ on $F(X)$.  Similarly from \ref{item:linearmap4} it is clear that $h^{1_{\cC}}$ is a $^*$-homomorphism so we may simply define a cocycle morphism satisfying conditions \ref{item:linearmap2}-\ref{item:linearmap4} by setting $h^{1_{\cC}}=\phi$.
\end{rmk}
\begin{rmk}\label{rmk: IrredDetLinearMaps}
As in Remark \ref{rmk: IrredDetermineCateg}, if the acting category $\mathcal{C}$ is semisimple, then the family of linear maps $\{h^{X}\}_{X\in\cC}$ is uniquely determined by the family $\{h^{X}\}_{X\in\Irr(\cC)}$. Precisely, if $X\cong\bigoplus_{i}X_i\in\cC$ is the decomposition as a direct sum of elements in $\Irr(\cC)$, then $F(X)$ is naturally isomorphic to $\bigoplus_{i}F(X_i)$ and $G(X)$ is naturally isomorphic to $\bigoplus_{i}G(X_i)$.\ Then, the map $h^X$ is $\bigoplus_{i}h^{X_i}$. In particular, it suffices to check that a family of linear maps $\{h^{X}: F(X)\rightarrow G(X)\}_{X\in\Irr(\cC)}$ satisfy the conditions of Lemma \ref{linearmapspicture} to yield a cocycle morphism, with understanding condition \ref{item:linearmap4} as $I_{X,Y}\circ (h^Y\boxtimes h^X)=\bigoplus_{i}h^{X_i}\circ J_{X,Y}$ for all $X,Y\in \Irr(\cC)$ and $X\otimes Y\cong \bigoplus_i X_i$ is the irreducible decomposition.
\end{rmk}

Lemma \ref{linearmapspicture} shows that any cocycle morphism can be equivalently represented by a pair $(\phi,h)$, where for convenience, we denote a cocycle morphism by $(\phi, h)$, where $h$ denotes the collection of linear maps $\{h^X\}_{X\in\cC}$.
Moreover, Lemma \ref{compositionlinearmaps} below shows that in the latter picture, the composition of cocycle morphisms translates to the composition of the underlying $^*$-homomorphisms and linear maps. From now on, we will freely identify these two pictures.

\begin{rmk}\label{rmk: ExtLinMaps}
When $\phi:A\rightarrow B$ is extendible and $(\phi,\{h^X\}_{X\in\cC})$ is a cocycle morphism, condition \ref{item:linearmap1} also follows for $a,a'\in \mathcal{M}(A)$. Precisely, if $a,a'\in \mathcal{M}(A)$ then $h^X(a\rhd x\lhd a')=\phi_p(a)\rhd h^X(x)\lhd \phi_p(a')$ for any $X\in\cC$ and $x\in F(X)$ where $p\in \M(B)$ is the projection associated to $\phi$. Similarly, for any $u,u'\in \U(\M(A))$ it follows that $h^X(u\rhd x\lhd u')=\phi^{\dagger}(u)\rhd h^X(x)\lhd \phi^{\dagger}(u')$. Indeed, for an approximate unit $e_{\lambda}$ of $A$, $\phi(e_{\lambda})$ converges to $p$ and $$p\rhd h^X(x)=\lim\limits_{\lambda}\phi(e_{\lambda})\rhd h^X(x)=\lim\limits_{\lambda}h^X(e_{\lambda}\rhd x)=h^X(x).$$
Likewise, we have $h^X(x)\lhd p=h^X(x)$.
Therefore, \[\phi^{\dagger}(u)\rhd h^X(x)\lhd \phi^{\dagger}(u')=\phi_p(u)\rhd h^X(x)\lhd \phi_p(u')=h^X(u\rhd x\lhd u').\]
\end{rmk}
\begin{example}\label{exmp: linearmapsHilb(G)}
Suppose $(A,\alpha)$ and $(B,\beta)$ are actions of a countable discrete group $\Gamma$ on $\C$-algebras $A$ and $B$. Consider them as actions of $\Hilb(\Gamma)$ as in Example \ref{exn:groupaction} and let $(\phi,\{h^g\}_{g\in \Gamma}):(A,\alpha)\to (B,\beta)$ be an extendible cocycle morphism. Recall from Example \ref{example: intertHilb(G)} that $\mathbbm{v}_{\mathbb{C}g}(a\boxtimes b)=\lim\limits_{\lambda}\eta_\lambda\boxtimes\mathbbm{u}_g\phi(a)b,$ where $\eta_\lambda$ is an approximate unit for $B$ and $(\mathbbm{u}_g)_{g\in \Gamma}\subseteq\mathcal{M}(B)$. Then for any $g\in \Gamma$ and $a\in A$, $$h^g(a)=f\left(\mathbbm{v}_{\mathbb{C}g}\left(\lim\limits_{\lambda}a\boxtimes\eta_\lambda\right)\right)=f\left(\lim\limits_{\lambda}\eta_\lambda\boxtimes\mathbbm{u}_g\phi(a)\right)=\mathbbm{u}_g\phi(a).$$
\end{example}

We now discuss the composition of cocycle morphisms.

\begin{lemma}\label{compositionlinearmaps}
If $(\phi,h):(A,F,J)\to (B,G,I)$ and $(\psi,l):(B,G,I)\to (C,H,K)$ are cocycle morphisms, then $(\psi\circ\phi,l\circ h):(A,F,J)\to (C,H,K)$ is a cocycle morphism and coincides with the composition of $(\phi,h)$ and $(\psi,l)$. 
\end{lemma}

\begin{proof}
Clearly $\psi\circ\phi:A\to C$ is a $^*$-homomorphism and $\{l^X\circ h^X: F(X)\to H(X)\}_{X\in\cC}$ is a family of linear maps. Conditions \ref{item:linearmap1}, \ref{item:linearmap2}, \ref{item:linearmap3}, and \ref{item:linearmap5} are immediate. The compatibility with the tensor product follows by stacking the diagrams in \ref{item:linearmap4} of Lemma \ref{linearmapspicture} for $h$ and $l$. Thus, $(\psi\circ\phi,l\circ h):(A,F,J)\to (C,H,K)$ induces a cocycle morphism by Lemma \ref{linearmapspicture}.

Suppose that $(\phi,\mathbbm{v})$ and $(\psi,\mathbbm{w})$ are cocycle morphisms associated to the families of linear maps $\{h^X\}_{X\in\cC}$ and $\{l^X\}_{X\in\cC}$ respectively. We claim that the cocycle morphism $(\psi\circ\phi,\mathbbm{w}*\mathbbm{v})$ has the associated family of linear maps $\{l^X\circ h^X\}_{X\in\cC}$. 

Recall from Definition \ref{def: CompCocMor} that $$(\mathbbm{w}*\mathbbm{v})_X=(T\boxtimes\id_{H(X)})\circ(\mathbbm{w}\circ\mathbbm{v})_X\circ S_X,$$ where $(\mathbbm{w}\circ\mathbbm{v})_X=(\id_{\phi}\boxtimes\mathbbm{w}_X)\circ(\mathbbm{v}_X\boxtimes\id_{\psi})$ (see \eqref{eq: CompCorrMor}). Start with an elementary tensor $x\boxtimes c$ and let $\zeta_\mu$ be an approximate unit of $A$ and $\eta_\lambda$ be an approximate unit of $B$. Following the composition of maps defining $(\mathbbm{w}*\mathbbm{v})_X$, we get that
\begin{align*}
x\boxtimes c &\mapsto \lim\limits_{\lambda}x\boxtimes \eta_\lambda\boxtimes c \\ &\mapsto\lim\limits_{\lambda}\lim\limits_{\mu}\phi(\zeta_\mu)\boxtimes h^X(x)\lhd \eta_\lambda\boxtimes c\\ &= \lim\limits_{\mu} \phi(\zeta_\mu)\boxtimes h^X(x)\boxtimes c \\ &\mapsto \lim\limits_{\mu}\lim\limits_{\lambda} \phi(\zeta_\mu)\boxtimes \psi(\eta_\lambda)\boxtimes l^X(h^X(x))\lhd c\\ &\mapsto \lim\limits_{\mu}\lim\limits_{\lambda}\psi(\phi(\zeta_\mu))\psi(\eta_\lambda)\boxtimes l^X(h^X(x))\lhd c\\ &= \lim\limits_{\mu} \psi(\phi(\zeta_\mu))\boxtimes l^X(h^X(x))\lhd c, 
\end{align*} where the last equality follows since $\eta_\lambda$ is an approximate unit, and in particular fixes $\phi(\zeta_\mu)$ in the limit. But this is precisely the formula in \eqref{eq: fromHtoV} corresponding to the family of linear maps $l\circ h$, so the two compositions agree.
\end{proof}

\begin{lemma}\label{lemma: CategoryCocMor}
The class of $\cC$-$\C$-algebras $(A,F,J)$, together with cocycle morphisms $(\phi,h):(A,F,J)\to (B,G,I)$ defines a category with respect to the composition in Lemma \ref{compositionlinearmaps}. 
\end{lemma}

\begin{proof}
The composition in Lemma \ref{compositionlinearmaps} is easily seen to be associative. Moreover, for any cocycle morphism $(\phi,h)$, the cocycle morphisms $(\id_A,\{\id_{F(X)}\})$ and $(\id_B,\{\id_{G(X)}\})$ are left and right identities respectively. 
\end{proof}

In the spirit of \cite[Definition 1.16]{cocyclecategszabo}, we have constructed a category of $\cC$-$\C$-algebras.

\begin{defn}\label{cocyclecatgeneralised}
  The \emph{generalised cocycle category} $\C_{\cC}$ is defined as the category whose objects are $\cC$-$\C$-algebras and whose morphisms are cocycle morphisms.\ Composition of morphisms in $\C_{\cC}$ is defined in Lemma \ref{compositionlinearmaps}. On any object $(A,F,J)$, the identity morphism in this category is given by $(\id_A,\{\id_{F(X)}\}_{X\in\cC})$.\ A cocycle morphism $(\phi,h):(A,F,J)\to (B,G,I)$ is invertible in this category if and only if $\phi:A\to B$ is an isomorphism and $h^X$ is bijective for any $X\in\cC$, in which case the inverse is given by $(\phi^{-1},h^{-1})$. Following the terminology in \cite{cocyclecategszabo}, we say that an invertible morphism in this category is a \emph{cocycle conjugacy}.
\end{defn}

\section{Inductive limits}\label{section: IndLimits}

In this section, we construct inductive limits in $\C_{\cC}$ for semisimple $\cC$.\ This is done in {\cite[Proposition 4.4]{limitfusioncategclassif}}, when restricted to unital $\C$-algebras and unital, injective connecting maps.\ Our approach is slightly different and does not need these assumptions.\ Before starting our construction, let us set up some notation.

Let $\cC$ be a semisimple $\C$-tensor category, and $A_n$ be a sequence of separable $\C$-algebras on which $\cC$ acts via the pair $(F_n,J^{(n)})$. Then, let
\begin{equation}\label{inductivelimitconnectingmaps}
(\phi_n,h_n):(A_n,F_n,J^{(n)})\to (A_{n+1},F_{n+1},J^{(n+1)})
\end{equation} be a sequence of cocycle morphisms. Recall that the $\C$-inductive limit $A=\lim\limits_{\longrightarrow}\{A_n,\phi_n\}$ is defined as the completion of 
\begin{equation*}
A^{(0)}=\frac{\left\{(a_n)_{n\geq 1}\in\bigoplus_{\ell^\infty}A_n: \lim_{n\to\infty}\|\phi_n(a_n)-a_{n+1}\|=0\right\}}{\bigoplus_{c_0}A_n}
\end{equation*} with respect to the topology induced by the norm $\|(a_n)_{n\geq 1}\|=\lim\limits_{n\to\infty}\|a_n\|_{A_n}$.\ Recall that the connecting maps $\phi_{n,\infty}:A_n\to A$ are given by 
\begin{equation}\label{eq: InfConnMaps}
(\phi_{n,\infty}(a_n))_k=
\begin{cases}
   \phi_{n,k}(a_n), \ k\geq n \\
   0, \quad \quad \quad k<n
\end{cases}
\end{equation} for all $n\geq 1$, where we adopt the standard notation $\phi_{n,m}:=\phi_{m-1}\circ\ldots\circ\phi_n$ and $\phi_{n,n}=\id_{A_n}$ for any $m>n\geq 1$.

Similarly, for any $X\in\cC$ and any $m>n\geq 1$, we consider the natural family $h_{n,m}^X:F_n(X)\to F_m(X)$ obtained by composition, with the convention that $h_{n,n+1}^X=h_n^X$.

To build an action on $A$, we start by constructing bimodules that will form the image of the functor.\ Essentially, for any $X\in\cC$, we can build a Hilbert $A$-$A$-bimodule as an inductive limit of $F_n(X)$.\ The construction is very similar to the one in \eqref{eq: SeqAlgBimodule}.

Define
\begin{equation*}
F^{(0)}(X)=\frac{\{(x_n)_{n\geq 1}\in \bigoplus_{\ell^\infty}F_n(X):  \lim_{n\to\infty}\|h_n^X(x_n)-x_{n+1}\|=0\}}{\bigoplus_{c_0}F_n(X)},
\end{equation*} where the norm on $F_n(X)$ is induced by the right inner product. 

For any $x=(x_n)_{n\geq 1},y=(y_n)_{n\geq 1} \in F^{(0)}(X)$ and any $a=(a_n)_{n\geq 1}\in A^{(0)}$, we can define 
\begin{equation}\label{Aactions}
 x\lhd a= (x_n\lhd a_n)_{n\geq 1}
\end{equation} and 
\begin{equation}\label{eq: InnerProdIndLim}
 \langle x,y\rangle_{A^{(0)}}=(\langle x_n,y_n\rangle_{A_n})_{n\geq 1}.  
\end{equation}

\begin{lemma}\label{checkconstruction}
For any $X\in\cC$, $F^{(0)}(X)$ equipped with the structure in \eqref{Aactions} and \eqref{eq: InnerProdIndLim} is a right pre-Hilbert-$A^{(0)}$-module.
\end{lemma}
\begin{proof}
We start by checking that the right action is well-defined. Firstly, if $x=(x_n)_{n\geq 1}\in F^{(0)}(X)$ and $a=(a_n)_{n\geq 1}\in A^{(0)}$, then $(x_n\lhd a_n)_{n\geq 1}$ induces an element in $F^{(0)}(X)$. By \ref{item:linearmap1} of Lemma \ref{linearmapspicture}, $h_n^X(x_n\lhd a_n)=h_n^X(x_n)\lhd \phi_n(a_n)$ for any $n\in\mathbb{N}, a_n\in A_n$, and $x_n\in F_n(X)$. Moreover, since $x\in F^{(0)}(X)$ and $a\in A^{(0)}$, there exists $N\in\mathbb{N}$ such that $h_n^X(x_n)= x_{n+1}$ and $\phi_n(a_n)=a_{n+1}$ for all $n\geq N$. Hence, $$h_n^X(x_n\lhd a_n)=x_{n+1}\lhd a_{n+1}$$ for all $n\geq N$, as required.

That \eqref{Aactions} is independent of the choice of representative sequences, follows exactly as in the proof of Lemma \ref{lemma: ConstructionCheck}.\ Thus, \eqref{Aactions} gives a well-defined right action of $A^{(0)}$. 

We now check that \eqref{eq: InnerProdIndLim} gives a well-defined right pre-inner product. Firstly, if $x=(x_n)_{n\geq 1},y=(y_n)_{n\geq 1}\in F^{(0)}(X)$, then $\langle x,y\rangle$ is an element of $A^{(0)}$. By \ref{item:linearmap3} of Lemma \ref{linearmapspicture} applied to each map $h_n^X$, we have that $$\phi_n(\langle x_n,y_n\rangle_{A_n})=\langle h_n^X(x_n),h_n^X(y_n)\rangle_{A_{n+1}}.$$ Moreover, since $x,y\in F^{(0)}(X)$, there exists $N\in\mathbb{N}$ such that for any $n\geq N$, $\langle h_n^X(x_n),h_n^X(y_n)\rangle_{A_{n+1}}=\langle x_{n+1},y_{n+1}\rangle_{A_{n+1}}$. Hence, $$\phi_n(\langle x_n,y_n\rangle_{A_n})=\langle x_{n+1},y_{n+1}\rangle_{A_{n+1}}$$ for all $n\geq N$, as required.

That \eqref{eq: InnerProdIndLim} is independent of the choice of representative sequences, follows as in the proof of Lemma \ref{lemma: ConstructionCheck}. Thus, \eqref{eq: InnerProdIndLim} gives a well-defined $A^{(0)}$-valued map. It is now straightforward to check that this function is right linear, left conjugate linear, and antisymmetric. Finally, it is clear that $\langle x,x\rangle_{A^{(0)}} \geq 0$ and $\langle x,x\rangle_{A^{(0)}} = 0$ if and only if $\langle x_n,x_n\rangle_{A_n}$ converges to $0$ i.e. $x=0$ in $F^{(0)}(X)$. Thus, the conclusion follows.
\end{proof}

For any $X\in\cC$, since $A^{(0)}$ is a dense $^*$-subalgebra of $A$, combining Lemma \ref{checkconstruction} and {\cite[Lemma 2.16]{raeburnwilliamsbook}}, we form the completion of $F^{(0)}(X)$, denoted by $F(X)$. This is a right-Hilbert $A$-module.

\begin{lemma}\label{buildleftactionHilbertbimodule}
For any $X\in\cC$, with the notation above, $F(X)$ is a non-degenerate right Hilbert $A$-$A$-bimodule.
\end{lemma}

\begin{proof}

We start by defining a left $A$-action on $F(X)$. For any $x=(x_n)_{n\geq 1} \in F^{(0)}(X)$ and any $a=(a_n)_{n\geq 1}\in A^{(0)}$, we can define 
\begin{equation}\label{eq: LeftAaction}
 a\rhd x= (a_n\rhd x_n)_{n\geq 1}.
\end{equation} The fact that \eqref{eq: LeftAaction} is a well-defined left $A^{(0)}$-action on $F^{(0)}(X)$ follows as in the proof of Lemma \ref{checkconstruction}. Moreover, the left $A^{(0)}$-action is adjointable, since the left action of $A_n$ on $F_n(X)$ is adjointable for any $n\geq 1$. Thus, using \cite[Lemma 2.16]{raeburnwilliamsbook}, it extends to a left action of $A$ on $F(X)$ by density.
\end{proof}

We now show that the assignment $X\to F(X)$ extends to an action of $\cC$ on $A$.

\begin{lemma}\label{inductivelimitfunctor}
With the notation above, $F:\cC^{\rev}\to \Corr_0^{\sep}(A) $ is a $\C$-functor. Moreover, there exists a unitary natural isomorphism $$J:=\{J_{X,Y}: F(Y)\boxtimes F(X)\to F(X\otimes Y): X,Y\in \cC\}$$ such that the pair $(F,J)$ is an action of $\cC$ on $A$.
\end{lemma}

\begin{proof}
Since $F_n(X)\in\Corr_0^{\sep}(A_n)$ for any $X\in\cC$ and any $n\in\mathbb{N}$, it follows that $F(X)\in\Corr_0^{\sep}(A)$ for any $X\in\cC$. 

Suppose $X,Y\in\cC$ and $f:X\to Y$ is a morphism in $\cC$. We define $F(f)(x)=(F_n(f)(x_n))_{n\geq 1}$ for any $x\in F^{(0)}(X)$. As each $F_n$ is a C$^*$-functor, the assignment $f\mapsto F_n(f)$ is contractive (see for example \cite[Definition~1.5]{GLR85}) and so $F(f)$ is bounded on $F^{(0)}(X)$. Therefore, one may extend $F(f)$ to a bounded $A$-$A$-bimodule map on $F(X)$. This operator is further adjointable as each $F_n(f)$ is adjointable. It is straightforward to check that $F$ defines a $\C$-functor. 

Furthermore, for any $X,Y\in\cC$, $x=(x_n)_{n\geq 1}\in F^{(0)}(X)$, and $y=(y_n)_{n\geq 1}\in F^{(0)}(Y)$, let $$J_{X,Y}(y\boxtimes x)=(J_{X,Y}^{(n)}(y_n\boxtimes x_n))_{n\geq 1}.$$We claim that $J_{X,Y}$ is well-defined. Firstly, note that $J_{X,Y}$ is independent of the choice of representative sequences as $J_{X,Y}^{(n)}$ is an isometry for any $n\in\mathbb{N}$ and $\lim_{n\to\infty}(F_n(Y)\boxtimes F_n(X))\cong F(Y)\boxtimes F(X)$.

Secondly, there exists $N\in\mathbb{N}$ such that for any $n\geq N$, $$h_n^X(x_n)=x_{n+1}\ \text{and} \ h_n^Y(y_n)=y_{n+1}.$$ Then, Lemma \ref{linearmapspicture} gives that $$h_n^{X\otimes Y}(J_{X,Y}^{(n)}(y_n\boxtimes x_n))=J_{X,Y}^{(n+1)}(h_n^Y(y_n)\boxtimes h_n^X(x_n))=J_{X,Y}^{(n+1)}(y_{n+1}\boxtimes x_{n+1}),$$ so the image of $J_{X,Y}$ is indeed contained in $F^{(0)}(X\otimes Y)$. By density, we can now extend $J_{X,Y}$ to $F(Y)\boxtimes F(X)$. 

Moreover, since each $J^{(n)}$ is a unitary natural isomorphism, so is $J$. Finally, since each family $\{J_{X,Y}^{(n)}:X,Y\in\cC\}$ satisfies a commuting diagram as in \eqref{diagramJmaps}, so does the collection of maps $\{J_{X,Y}:X,Y\in\cC\}$. Hence, the pair $(F,J)$ gives an action of $\cC$ on $A$.
\end{proof}

We now show the existence of sequential inductive limits in $\C_{\cC}$.

\begin{prop}
The triple $(A,F,J)$ constructed above defines the inductive limit of the inductive system in \eqref{inductivelimitconnectingmaps}.  
\end{prop}

\begin{proof}
For any $n\geq 1$, we need to define cocycle morphisms $$(\phi,h)_{n,\infty}:(A_n,F_n,J^{(n)})\to (A,F,J)$$ such that for any $n\geq 1$, the diagram 
\begin{equation}\label{univpropertyindlimit}
\begin{tikzcd}[row sep=4em, column sep=2em]
    (A_n,F_n,J^{(n)}) \arrow{rr}{(\phi,h)_{n,\infty}} \arrow[swap]{dr}{(\phi_n,h_n)} & & (A,F,J) \\
    & (A_{n+1},F_{n+1},J^{(n+1)}) \arrow[swap]{ur}{(\phi,h)_{n+1,\infty}}, \\[-4.6em]
\end{tikzcd}
\end{equation}commutes.

Recalling \eqref{eq: InfConnMaps}, for any $n\geq 1$, we have a $^*$-homomorphism $\phi_{n,\infty}:A_n\to A$. To define a cocycle morphism from $(A_n,F_n,J^{(n)})$ to $(A,F,J)$, it suffices to find a collection of linear maps $\{h_{n,\infty}^X:F_n(X)\to F(X)\}_{X\in\cC}$ satisfying the conditions in Lemma \ref{linearmapspicture}. If $X\in\cC$ and $x_n\in F_n(X)$ for $n\in\mathbb{N}$, we define 
\begin{equation}\label{eq: ConnLinMaps}
(h_{n,\infty}^X(x_n))_k=
\begin{cases}
 h_{n,k}^X(x_n) & \ k\geq n \\
0 & k<n.
\end{cases}
\end{equation} By convention, $h_{n,n}^X(x_n)=x_n$ for any $X\in\cC$ and any $x_n\in F_n(X)$. Then, $h_{n,\infty}^X$ is a linear map into $F(X)$. Moreover, the conditions of Lemma \ref{linearmapspicture} are checked pointwise since $(\phi_{n,k},h_{n,k})$ is a cocycle morphism for any $k>n$. Thus, for any $n\geq 1$, the pair $(\phi_{n,\infty},h_{n,\infty})$ defines a cocycle morphism from $(A_n,F_n,J^{(n)})$ to $(A,F,J)$.

We now check that the triple $(A,F,J)$ is the inductive limit of the sequence of triples $(A_n,F_n,J^{(n)})$, together with connecting morphisms $(\phi_n,h_n)$. Firstly, it is clear that \eqref{univpropertyindlimit} commutes. It remains to check that $(A,F,J)$ satisfies the universal property. Precisely, let $B$ be a $\C$-algebra and $(B,G,I)$ defining an action of $\cC$ on $B$. Suppose there exists a sequence of cocycle morphisms $(\psi,l)_{n,\infty}:(A_n,F_n,J^{(n)})\to (B,G,I)$ such that the diagram
\begin{equation}\label{eq: UnivProp}
\begin{tikzcd}[row sep=4em, column sep=2em]
    (A_n,F_n,J^{(n)}) \arrow{rr}{(\psi,l)_{n,\infty}} \arrow[swap]{dr}{(\phi_{n,m},h_{n,m})} & & (B,G,I) \\
    & (A_m,F_m,J^{(m)}) \arrow[swap]{ur}{(\psi,l)_{m,\infty}}\\[-4.6em]
\end{tikzcd}
\end{equation} commutes for any $m>n\geq 1$. We claim that there exists a unique cocycle morphism $(\Phi,r):(A,F,J)\to (B,G,I)$ such that for any $n\geq 1$, $$(\Phi,r)\circ (\phi,h)_{n,\infty}= (\psi,l)_{n,\infty}.$$

For any $X\in\cC$, the union $\bigcup_{k\geq 1}h_{k,\infty}^X(F_k(X))$ is dense in $F(X)$. Since $l_{k,\infty}^X$ is contractive, \eqref{eq: UnivProp} yields that $\Ker\ h_{k,\infty}^X\subseteq \Ker\ l_{k,\infty}^X$ for any $X\in\cC$ and any $k\in\mathbb{N}$. Therefore, for any $x_k\in F_k(X)$, define $r^X(h_{k,\infty}^X(x_k))=l_{k,\infty}^X(x_k)$. As $l_{k,\infty}^X$ is contractive, we may extend $r^X$ to a well-defined linear map $r^X:F(X)\to G(X)$.\ The fact that $r^X$ satisfies the conditions appearing in Lemma \ref{linearmapspicture} is routine and essentially follows from the fact that for each $k\geq 1$, $(\psi_{k,\infty},l_{k,\infty})$ and $(\phi_{k,\infty},h_{k,\infty})$ are cocycle morphisms. Therefore, $(\Phi,r)$ is the unique cocycle morphism such that $(\Phi,r)\circ (\phi,h)_{n,\infty}= (\psi,l)_{n,\infty}$.\ Hence, the triple $(A,F,J)$ is the inductive limit of the system in \eqref{inductivelimitconnectingmaps}. 
\end{proof}

\begin{rmk}\label{rmk: ExtMor}
Note that if the connecting maps $\phi_n:A_n\to A_{n+1}$ of the inductive system are extendible, then the inductive limit connecting map $\phi_{k,\infty}:A_k\to A$ is extendible for any $k\geq 1$.
\end{rmk}

\section{Approximate unitary equivalence}

Let $\mathcal{C}$ be a semisimple $\C$-tensor category with countably many isomorphism classes of simple objects.
To define approximate unitary equivalence, we first introduce a topology on the space of cocycle representations between two actions $\mathcal{C}\overset{F}{\curvearrowright} A$ and $\mathcal{C}\overset{G}{\curvearrowright} B$. For this we employ the same approach as in \cite[Section 2.1]{cocyclecategszabo}; introducing a family of pseudometrics which measure the distance between cocycle representations.

\begin{rmk}
Recall that for any fixed $X\in\cC$, $\mathbbm{v}_X:F(X)\boxtimes {}_{\phi}B\to {}_{\phi}B\boxtimes G(X)$ and $\mathbbm{w}_X:F(X)\boxtimes {}_{\psi}B\to {}_{\psi}B\boxtimes G(X)$. By the construction of the tensor product, both $F(X)\boxtimes {}_{\phi}B$ and $F(X)\boxtimes {}_{\psi}B$ are completions of quotients of the algebraic tensor product $F(X)\odot B$. To compare the difference between the maps, it suffices to do so on their respective images of elementary tensors $x\odot b$ in the right Hilbert $B$-module $B\boxtimes G(X)$. 
\end{rmk}

\begin{defn}\label{topologycocrep}
For any finite set $K\subset \Irr(\cC)$ containing $1_{\cC}$ and any compact sets $\mathcal{F}^B\subset B$, and $\mathcal{F}^X\subset F(X)$ for any $X\in K$, denote $\mathcal{F}= \mathcal{F}^B\times\left(\coprod_{X\in K} \mathcal{F}^X\right)$. Then, we define the pseudometric 

\begin{align*}
d_{\mathcal{F}}\big((\phi,\mathbbm{v}), (\psi,\mathbbm{w})\big)&= \max\limits_{(b,\xi_X)\in\mathcal{F}}\|\mathbbm{v}_X(\xi_X\boxtimes_{\phi} b)-\mathbbm{w}_X(\xi_X\boxtimes_{\psi} b)\|,
\end{align*} where the norm is induced by the right inner product on the right Hilbert $B$-module $B\boxtimes G(X)$.
\end{defn}

\begin{rmk}\label{rmk: IsomRecoverMorphTop}
Note that $F(1_{\cC})=A$ and for any $a\in A$ and $b\in B$ $\mathbbm{v}_{1_{\cC}}(a\boxtimes_{\phi}b)=\phi(a)b$ and $\mathbbm{w}_{1_{\cC}}(a\boxtimes_{\psi}b)=\psi(a)b$. 
\end{rmk}

We will only use the topology generated by the family of pseudometrics $d_\mathcal{F}$ in the setting of cocycle morphisms, so let us derive the relevant subset topology.

\begin{lemma}\label{lemma: topologycocyclemor}
Let $(\phi_{\lambda},\mathbbm{v}_{\lambda}), (\phi,\mathbbm{v}):(A,F,J)\to (B,G,I)$ be cocycle morphisms with associated families of linear maps $\{h_{\lambda}^X\}$ and $\{h^X\}$. Then $(\phi_{\lambda},\mathbbm{v}_{\lambda})$ converges to $(\phi,\mathbbm{v})$ if and only if for any $X\in\Irr(\cC)$, $h_{\lambda}^X$ converges pointwise to $h^X$in the norm induced by the right inner product.
\end{lemma}

\begin{proof}
Suppose that for any $X\in\Irr(\cC)$, $h_{\lambda}^X$ converges pointwise to  $h^X$ in the norm induced by the right inner product. Let $\zeta_\mu$ be an approximate unit of $A$ and recall from \eqref{eq: fromHtoV} that for any $X\in\cC$, $x\in F(X)$, and $b\in B$ $$\mathbbm{v}_X(x\boxtimes_{\phi} b)=\lim\limits_{\mu}\phi(\zeta_\mu)\boxtimes h^X(x)\lhd b$$ and $$(\mathbbm{v}_{\lambda})_X(x\boxtimes_{\phi_{\lambda}} b)=\lim\limits_{\mu}\phi_{\lambda}(\zeta_\mu)\boxtimes h_{\lambda}^X(x)\lhd b.$$ Fix $X\in\cC$, $x\in F(X)$, and $b\in B$ and recall that since $F(X)$ is non-degenerate, there exists $a\in A$ and $y\in F(X)$ such that $a\rhd y=x$. In particular, \ref{item:linearmap1} of Lemma \ref{linearmapspicture} yields that $h^X(x)=\phi(a)\rhd h^X(y)$ and $h^X_\lambda(x)=\phi_\lambda(a)\rhd h^X_\lambda(y)$. Therefore, by the definition of the tensor product, one has that
\begin{equation}\label{eq: Top1}
\mathbbm{v}_X(x\boxtimes_{\phi} b)=\lim\limits_{\mu}\phi(\zeta_\mu)\lhd \phi(a)\boxtimes h^X(y)\lhd b= \phi(a)\boxtimes h^X(y)\lhd b.
\end{equation}
Similarly, 
\begin{equation}\label{eq: Top2}
(\mathbbm{v}_{\lambda})_X(x\boxtimes_{\phi_{\lambda}} b)=\phi_\lambda(a)\boxtimes h^X_\lambda(y)\lhd b.
\end{equation}

Moreover, by \ref{item:linearmap5} of Lemma \ref{linearmapspicture}  $h^{1_{\cC}}(a)=\phi(a)$ and $h^{1_{\cC}}_\lambda(a)=\phi_\lambda(a)$ for any $a\in A$. By \eqref{eq: Top1} and \eqref{eq: Top2}, it follows that for any $X\in \Irr(\cC)$, $(\mathbbm{v}_{\lambda})_X$ converges pointwise to $\mathbbm{v}_X$. Therefore, $(\phi_{\lambda},\mathbbm{v}_{\lambda})$ converges to $(\phi,\mathbbm{v})$ by definition of the pseudometrics in Definition \ref{topologycocrep}.

Conversely, suppose that $(\phi_{\lambda},\mathbbm{v}_{\lambda})$ converges to $(\phi,\mathbbm{v})$. Then $(\phi_{\lambda},\mathbbm{v}_{\lambda})$ satisfies the Cauchy criterion with respect to every pseudometric $d_{\mathcal{F}}$ above. Recall from \eqref{hcomp} that $h_{\lambda}^X(x)=f((\mathbbm{v}_{\lambda})_X(\iota(x))$ for any $X\in\cC$ and $x\in F(X)$, where $f$ and $\iota$ are defined in \eqref{hcomp}. Since $(\mathbbm{v}_{\lambda})_X$ is Cauchy in the point-norm topology and $f$ is continuous, it follows that $h_{\lambda}^X$ is Cauchy for any $X\in\Irr(\cC)$. But the norm induced by the right inner product is complete, so $h_{\lambda}^X$ converges pointwise to $h^X$ for any $X\in\Irr(\cC)$.
\end{proof}

\begin{example}\label{exmp: DiffTop}
Let $\Gamma$ be a countable discrete group. Let $(\alpha,\mathfrak{u}): \Gamma\curvearrowright A$ and $(\beta,\mathfrak{v}): \Gamma\curvearrowright B$ be two twisted actions on $\C$-algebras and $(\phi,\mathbbm{v}')$, $(\psi,\mathbbm{w}'):(A,\alpha,\mathfrak{u})\to (B,\beta,\mathfrak{v})$ be two extendible cocycle morphisms as in Definition \ref{def: cocyclemorGroups}. Recall from Example \ref{example: intertHilb(G)} that $(\phi,\mathbbm{v}')$ and $(\psi,\mathbbm{w}')$ induce cocycle morphisms $(\phi,\mathbbm{v}),(\psi,\mathbbm{w}):(A,\alpha,\mathfrak{u})\to(B,\beta,\mathfrak{v})$ in the sense of Definition \ref{cocyclemorphism}, where we view $(\alpha,\mathfrak{u})$ and $(\beta,\mathfrak{v})$ as $\Hilb(\Gamma)$-actions. So,
\begin{equation}\label{grouptop}
\mathbbm{v}_g([a\boxtimes b]_{\phi})=\lim\limits_{\mu}\eta_\mu\boxtimes \mathbbm{v}_g'\phi(a)b, \quad \mathbbm{w}_g([a\boxtimes b]_{\psi})=\lim\limits_{\mu}\eta_\mu\boxtimes \mathbbm{w}_g'\psi(a)b,
\end{equation} for all $a\in A$, $b\in B$, $g\in \Gamma$, and $\eta_\mu$ an approximate unit for $B$. In this case, the topology is induced by the pseudometrics 
\begin{equation}\label{eq: TopGroup1}
d_{\mathcal{F}}\big((\phi,\mathbbm{v}), (\psi,\mathbbm{w})\big)= \max\limits_{(b,a_g)\in \mathcal{F}}\|\mathbbm{v}_g'\phi(a_g)b-\mathbbm{w}_g'\psi(a_g)b\|,
\end{equation} where $\mathcal{F}=\mathcal{F}^B\times\left(\coprod_{g\in K}\mathcal{F}^A\right)$, for compact sets $\mathcal{F}^A\subset A, \mathcal{F}^B\subset B$, and finite $K\subset \Gamma$ containing $1_\Gamma$.
Equivalently, by taking an approximate unit for $B$, the topology is generated by the pseudometrics
\begin{equation}\label{eq: TopGroup2}
d_{\mathcal{F}^A\times K}\big((\phi,\mathbbm{v}), (\psi,\mathbbm{w})\big)= \max\limits_{a_g\in \mathcal{F}}\|\mathbbm{v}_g'\phi(a_g)-\mathbbm{w}_g'\psi(a_g)\|.
\end{equation}

In \cite[Definition 2.5]{cocyclecategszabo}, Szabó defines a topology on the space of cocycle morphisms which is generated by the family of pseudometrics
\begin{equation}\label{eq: GaborTop}
d_{\mathcal{F}}\big((\phi,\mathbbm{v}'), (\psi,\mathbbm{w}')\big)=\max\limits_{a\in\mathcal{F}^A}\|\phi(a)-\psi(a)\|+ \max\limits_{g\in K}\max\limits_{b\in \mathcal{F}^B}\|b((\mathbbm{v}_g')^*-(\mathbbm{w}_g')^*)\|.  
\end{equation} The convergence with respect to the family of pseudometrics in \eqref{eq: GaborTop} implies convergence with respect to the family of pseudometrics in \eqref{eq: TopGroup1}.\ If $\phi$ and $\psi$ are non-degenerate, then $\phi(A)B$ and $\psi(A)B$ are dense in $B$. Moreover, the case $g=1$ in \eqref{eq: TopGroup1} recovers the pointwise difference of the morphisms.\ Thus, convergence with respect to the family of pseudometrics in \eqref{eq: TopGroup1} implies convergence with respect to the family of pseudometrics in \eqref{eq: GaborTop}.\ However, these topologies are different outside the non-degenerate setting, with the topology induced by the family of  pseudometrics in \eqref{eq: TopGroup1} being coarser than the topology induced by the family of pseudometrics in \eqref{eq: GaborTop}. Finally, we would like to point out that the topology induced by \eqref{eq: TopGroup2} coincides with that used in \cite[Section 5]{nawata}.
\end{example}

We finish our discussion by noticing that the composition of cocycle morphisms is jointly continuous with respect to the topology defined above. This fact will be used in Section \ref{section: onesidedintert} in the context of asymptotic unitary equivalence.

\begin{lemma}\label{lemma: compjointlycts}
Let $(\phi,h):(A,F,J)\to (B,G,I)$ and $(\psi,l):(B,G,I)\to (C,H,K)$ be two cocycle morphisms. Then the composition map given by $$[(\phi,h),(\psi,l)]\mapsto (\psi,l)\circ (\phi,h)$$ is jointly continuous.
\end{lemma}

We can now introduce a notion of approximate unitary equivalence that will be crucial to perform equivariant Elliott intertwining arguments. We start by defining unitary equivalence for cocycle morphisms and then use the topology on the space of morphisms to obtain an approximate notion of unitary equivalence.

Suppose $\cC$ is a semisimple $\C$-tensor category with countably many isomorphism classes of simple objects acting on a $\C$-algebra $B$. We denote this action by the triple $(B,G,I)$. For any $u$ unitary in $\mathcal{M}(B)$, we consider $\Ad(u):B\to B$ to be the $^*$-homomorphism given by $b\mapsto ubu^*$. Then $\Ad(u)$ induces a Hilbert $B$-bimodule ${}_{\Ad(u)}B$.\ The map $T_u:{}_{\Ad(u)}B\to B$ given by $b\mapsto u^*b$ is a bimodule isomorphism. But for any $X\in\cC$, $G(X)\boxtimes B\cong B\boxtimes G(X)$, so there exists a unitary isomorphism $(\mathbbm{v}_{u})_X: G(X)\boxtimes {}_{\Ad(u)}B\to {}_{\Ad(u)}B\boxtimes G(X)$.\ It follows that $(\Ad(u),\mathbbm{v}_{u}):(B,G,I)\to (B,G,I)$ is a cocycle morphism..\ We denote by $h_u=\{h_u^X:G(X)\to G(X)\}_{X\in\cC}$ the collection of linear maps corresponding to the cocycle morphism induced by $\Ad(u)$. 

\begin{lemma}\label{lemma: linearmapsAd(u)}
Let $\mathcal{C}\overset{F}{\curvearrowright} B$  be an action of $\mathcal{C}$ on a $\C$-algebra $B$, and let $u\in\mathcal{M}(B)$ be a unitary.\ Then $\Ad(u)$ induces a cocycle morphism $(\Ad(u),h_u):(B,G,I)\to(B,G,I)$, where $h_u^X(x)=u\rhd x\lhd u^*$ for any $X\in\cC$ and any $x\in G(X)$.\footnote{Recall that $\rhd$ denotes an action on the left, while $\lhd$ an action on the right.}
\end{lemma}

\begin{proof}
For any $X\in\cC$, consider the bimodule maps $L_X:G(X)\to B\boxtimes G(X)$ and $R_X:G(X)\to G(X)\boxtimes B$ given by $L_X(x)=\lim\limits_{\mu}\eta_\mu\boxtimes x$ and $R_X(x)=\lim\limits_{\mu}x\boxtimes\eta_\mu$ for some $\eta_\mu$ quasicentral approximate unit of $B$ with respect to $\mathcal{M}(B)$ (i.e. $\eta_{\mu}x-x\eta_{\mu}\to 0$ for all $x\in\mathcal{M}(B)$).

As the bimodule $G(X)$ is non-degenerate for any $X\in\cC$, it follows from Lemma \ref{lemma:ActionMultiplierAlg} that $$u\rhd x= L_X^{-1}(u\rhd L_X(x)) \quad \AND \quad x\lhd u^*=R_X^{-1}(R_X(x)\lhd u^*).$$ Then, consider the map $$\iota: G(X)\to G(X)\boxtimes B$$ given by  $x\mapsto \lim\limits_{\mu}x\boxtimes u^*\eta_\mu$. Since $\eta_\mu$ is quasicentral, $$\iota(x)=\lim\limits_{\mu}x\boxtimes\eta_\mu u^*=R_X(x)\lhd u^*.$$

Then, $L_X\circ R_X^{-1}\circ\iota: G(X)\to B\boxtimes G(X)$ is the map given by $$L_X(R_X^{-1}(\iota(x)))=L_X(x\lhd u^*).$$ Finally, we consider the map $$f:B\boxtimes G(X)\to G(X)$$ given by $f(b\boxtimes x)=L_X^{-1}(u\rhd (b\boxtimes x))$. 

Following the construction in Lemma \ref{linearmapspicture} and since the family of maps $\{h_u^X\}$ is independent of the chosen approximate unit, it follows that $$h_u^X=f\circ L_X\circ R_X^{-1}\circ\iota.$$ Therefore, for any $X\in\cC$ and any $x\in G(X)$, $$h_u^X(x)=f(L_X(x\lhd u^*))=L_X^{-1}(u\rhd L_X(x\lhd u^*))=u\rhd x\lhd u^*,$$ which finishes the proof.
\end{proof}

\begin{defn}\label{unitaryeq}
Let $\mathcal{C}\overset{F}{\curvearrowright} A$ and $\mathcal{C}\overset{G}{\curvearrowright} B$  be actions of $\mathcal{C}$ on $\C$-algebras $A$ and $B$ and let $(\phi,\mathbbm{v}),(\psi,\mathbbm{w}):(A,F,J)\to (B,G,I)$ be cocycle representations.
\begin{enumerate}
\item  We say that the pairs $(\phi,\mathbbm{v})$ and $(\psi,\mathbbm{w})$ are \emph{unitarily equivalent} if there exists a unitary $u\in\mathcal{U}(\mathcal{M}(B))$ such that $$(\Ad(u),\mathbbm{v}_{u})\circ (\phi,\mathbbm{v})= (\psi,\mathbbm{w}).\footnote{As mentioned in Remark \ref{rmk:NewComp}, the composition works since $\Ad(u)$ is a non-degenerate $^*$-homomorphism.}$$

\item  We say that $(\psi,\mathbbm{w})$ is an \emph{approximate unitary conjugate} of $(\phi,\mathbbm{v})$ if there exists a net of unitaries $u_\lambda\in\mathcal{U}(\mathcal{M}(B))$ such that $$\psi(a)=\lim\limits_{\lambda}u_\lambda\phi(a)u_\lambda^*,$$ and $$\|\mathbbm{w}_X(x\boxtimes_{\psi}b)-(\mathbbm{v}_{u_\lambda}*\mathbbm{v})_X(x\boxtimes_{\Ad(u_{\lambda})\phi}b)\|\overset{\lambda}{\longrightarrow} 0$$ for all $a\in A$, $b\in B$, $X\in\Irr(\cC)$, and any $x\in F(X)$. We denote this by $(\psi,\mathbbm{w})\precapprox_u(\phi,\mathbbm{v})$.\footnote{Note that this means precisely that $(\Ad(u_\lambda),\mathbbm{v}_{u_\lambda})\circ(\phi,\mathbbm{v})$ converges to $(\psi,\mathbbm{w})$ with respect to the topology in Definition \ref{topologycocrep}.}
\end{enumerate}
\end{defn}

The analogous notion of approximate unitary conjugacy in \cite[Definition 2.8]{cocyclecategszabo} is only a subequivalence relation in general. Indeed, symmetry fails when the cocycle representations considered are not extendible (\cite[Remark 2.9]{cocyclecategszabo}). However, the topology that we consider does indeed make it an equivalence relation.
\begin{lemma}\label{lem: conjisequivalence}
    Let $(\phi,\bv),(\psi,\bw):(A,F,J)\rightarrow (B,G,I)$ be cocycle representations between $\cC$-C$^*$-algebras. Then if $(\psi,\mathbbm{w})\precapprox_u(\phi,\mathbbm{v})$, one also has that $(\phi,\mathbbm{v})\precapprox_u(\psi,\mathbbm{w})$.
\end{lemma}
\begin{proof}
It can be computed from the definitions that for any unitary $u\in \cU(\cM(B))$, $X\in \cC$ and $x\in F(X)$
    \[
    (\bv_{u}* \bv)_X(x\boxtimes_{\Ad(u)\phi} b)=u\rhd \bv_X(x\boxtimes_{\phi} u^*b).
    \]
In particular 
\begin{align*}
&\|\mathbbm{w}_X(x\boxtimes_{\psi}b)-(\mathbbm{v}_{u_\lambda}*\mathbbm{v})_X(x\boxtimes_{\Ad(u_{\lambda})\phi}b)\|\\
&=\lVert \mathbbm{w}_X(x\boxtimes_{\psi}b)-u_{\lambda}\rhd\bv_{X}(x\boxtimes_{\phi}u_{\lambda}^*b)\rVert\\
&=\lVert u_{\lambda}^*\rhd \bw_{X}(x\boxtimes_{\psi} b)-\bv_{X}(x\boxtimes_{\phi} u_{\lambda}^*b)\rVert\\
&=\lVert (\bv_{u_{\lambda}^*}*\bw)_{X}(x\boxtimes_{\Ad(u_{\lambda}^*)\psi} u_{\lambda}^*b)-\bv_X(x\boxtimes_{\phi} u_{\lambda}^*b)\rVert.
\end{align*}
Thus the result follows.
\end{proof}
In light of Lemma \ref{lem: conjisequivalence}, we simply call two cocycle representations \[(\psi,\bw),(\phi,\bv):(A,F,J)\rightarrow (B,G,I)\] such that $(\psi,\bw)\precapprox_u (\phi,\bv)$ \emph{approximately unitarily equivalent} and we denote it rather by $(\phi,\bv)\approx_u (\psi,\bw)$.
In general, we will be interested in these notions when $(\phi,\mathbbm{v})$ and $(\psi,\mathbbm{w})$ are cocycle morphisms, so let us record the following immediate consequence of Lemma \ref{lemma: linearmapsAd(u)}.

\begin{lemma}\label{lemma: approxunitconj}
If $(\phi,\mathbbm{v}),(\psi,\mathbbm{w}):(A,F,J)\to (B,G,I)$ are cocycle morphisms with associated families of linear maps $\{h^X\}$ and $\{l^X\}$ respectively, then $(\psi,\mathbbm{w})$ is approximately unitarily equivalent to $(\phi,\mathbbm{v})$ if and only if there exists a net of unitaries $u_\lambda\in\mathcal{U}(\mathcal{M}(B))$ such that $$\|l^X(x)-u_\lambda\rhd h^X(x)\lhd u_\lambda^*\| \overset{\lambda}{\longrightarrow} 0$$ for any $X\in \Irr(\cC)$, and any $x\in F(X)$. 
\end{lemma}

We finish this section by defining asymptotic unitary equivalence for cocycle representations.

\begin{defn}\label{def: AsymptUnitaryEq}
Let $\mathcal{C}\overset{F}{\curvearrowright} A$ and $\mathcal{C}\overset{G}{\curvearrowright} B$  be actions of $\mathcal{C}$ on $\C$-algebras $A$ and $B$ respectively and let $(\phi,\mathbbm{v}),(\psi,\mathbbm{w}):(A,F,J)\to (B,G,I)$ be cocycle representations. We say that $(\psi,\mathbbm{w})$ is \emph{asymptotically unitarily equivalent} to $(\phi,\mathbbm{v})$ if there exists a strictly continuous map $u: [0,\infty) \to \mathcal{U}(\mathcal{M}(B))$ such that $$\psi(a)=\lim\limits_{t\to\infty}u_t\phi(a)u_t^*,$$ and $$\|\mathbbm{w}_X(x\boxtimes_{\psi}b)-(\mathbbm{v}_{u_t}*\mathbbm{v})_X(x\boxtimes_{\Ad(u_t)\phi}b)\|\overset{t\to\infty}{\longrightarrow} 0$$ for all $a\in A$, $b\in B$, $X\in\Irr(\cC)$, and any $x\in F(X)$. This will be denoted by $(\phi,\mathbbm{v})\cong_u(\psi,\mathbbm{w})$.
\end{defn}
The same argument as in Lemma \ref{lem: conjisequivalence} shows that asymptotic unitary equivalence is an equivalence relation.
Moreover, the same argument as in Lemma \ref{lemma: approxunitconj} gives the following equivalent characterisation for asymptotic unitary equivalence of cocycle morphisms.

\begin{lemma}\label{lemma: asymptunitconj}
If $(\phi,\mathbbm{v}),(\psi,\mathbbm{w}):(A,F,J)\to (B,G,I)$ are cocycle morphisms with associated families of linear maps $\{h^X\}_{X\in\cC}$ and $\{l^X\}_{X\in\cC}$ respectively, then $(\psi,\mathbbm{w})$ is \emph{asymptotically unitarily equivalent} to $(\phi,\mathbbm{v})$ if and only if there exists a strictly continuous map $u:[0,\infty)\to\mathcal{U}(\mathcal{M}(B))$ such that  $$\|l^X(x)-u_t\rhd h^X(x)\lhd u_t^*\| \overset{t\to\infty}{\longrightarrow} 0$$ for any $X\in \Irr(\cC)$, and any $x\in F(X)$.
\end{lemma}

\begin{rmk}\label{rmk: VerticalComp}
Let $(\phi_1,h_1):(A,F,J)\to (B,G,I), (\psi_1,l_1):(B,G,I)\to (C,H,K)$ be two cocycle morphisms. Suppose that $(\phi_1,h_1)$ is unitarily equivalent to $(\phi_2,h_2)$ and $(\psi_1,l_1)$ is unitarily equivalent to $(\psi_2,l_2)$. If $\psi_1$ is extendible, then $(\psi_1\circ\phi_1,l_1\circ h_1)$ is unitarily equaivalent to $(\psi_2\circ\phi_2,l_2\circ h_2)$. Indeed, let $u\in \mathcal{U}(\mathcal{M}(B))$ and $v\in \mathcal{U}(\mathcal{M}(C))$ be such that $u\rhd h_1^X(x)\lhd u^*=h_2^X(x)$ and $v\rhd l_1^X(y)\lhd v^*=l_2^X(y)$ for any $X\in\cC$, $x\in F(X)$, and $y\in G(X)$. Then, $v\psi_1^{\dagger}(u)\rhd l_1^X(h_1^X(x))\lhd (v\psi_1^{\dagger}(u))^*=l_2^X(h_2^X(x))$. In general, the vertical composition in the $2$-category of correspondences is not well defined outside of the extendible setting.
\end{rmk}

We will finish by recording the following lemma which will be used in Section \ref{section: onesidedintert}.

\begin{lemma}\label{lemma: CompConjIsConjAsympt}
Let $(\phi,h):(A,F,J)\to (B,G,I)$ and $(\psi,l):(B,G,I)\to (C,H,K)$ be two cocycle morphisms which are asymptotically unitarily equivalent to cocycle conjugacies. Then their composition $(\psi\circ\phi,l\circ h)$ is asymptotically unitarily equivalent to a cocycle conjugacy. 
\end{lemma}

\begin{proof}
Suppose that $(\Phi,H)$ and $(\Psi,L)$ are cocycle conjugacies such that $$(\Phi,H)\cong_u(\phi,h) \quad \AND \quad (\Psi,L)\cong_u(\psi,l).$$ By Lemma \ref{lemma: compjointlycts}, composition is jointly continuous. Moreover, since $\Psi$ is extendible, $(\Psi\circ\Phi,L\circ H)\cong_u(\psi\circ\phi,l\circ h)$ by Remark \ref{rmk: VerticalComp}.
\end{proof}

\section{Two-sided Elliott intertwining}\label{sect: TwoSidedInt}

With our setup we may now use Elliott's abstract framework from \cite{AbstractElliott} to show Theorem \ref{thm: IntroThmIntert}.
\begin{proof}
We check the conditions of the second theorem in \cite{AbstractElliott}.\ Our underlying category is the subcategory of $\C_{\cC}$ of separable C$^*$-algebras, together with extendible cocycle morphisms (see Definition \ref{cocyclecatgeneralised}). In this category, Lemma \ref{lemma: linearmapsAd(u)} gives us a notion of inner automorphisms in the sense of \cite{AbstractElliott}. Moreover, the topology of the morphism spaces in this category, as in Lemma \ref{lemma: topologycocyclemor}, is induced by a complete metric. Indeed, let $K_n$ be an increasing sequence of finite sets containing $1_{\cC}$ such that $\bigcup_{n\in \N} K_n=\Irr(\mathcal{C})$. For any $X\in\Irr(\cC)$, choose a sequence of contractions $\mu_n^X$ which is dense in the unit ball of $F(X)$. Then, the assignment
\begin{equation*}
    \left((\phi,h),(\psi,l)\right)\mapsto\sum_{n=1}^{\infty}2^{-n}\max_{X\in K_n}\max_{m\leq n}\|h^X(\mu_m^X)-l^X(\mu_m^X)\|.
\end{equation*}
yields a complete metric recovering the topology on the space of morphisms.\ Furthermore, composition with an inner automorphism is isometric. Since the composition of cocycle morphisms is jointly continuous by Lemma \ref{lemma: compjointlycts}, all the conditions of the second theorem in \cite{AbstractElliott} are satisfied. Hence, the conclusion follows.
\end{proof}
However, we also decide to give the technical argument in full, picking up some more general results in the process. We start by introducing the setup for performing approximate intertwining arguments in the spirit of {\cite[Definition 3.3]{cocyclecategszabo}}.

\begin{defn}\label{def: approx-int}
Let $\cC$ be a semisimple $\C$-tensor category with countably many isomorphism classes of simple objects. Let $(F_n,J^{(n)}): \cC\curvearrowright A_n$ and $(G_n,I^{(n)}): \cC\curvearrowright B_n$ be sequences of actions on separable $\C$-algebras. Let
\[
(\phi_n,\{h_n^X\}_{X\in\cC}): (A_n,F_n,J^{(n)}) \to (A_{n+1},F_{n+1},J^{(n+1)})
\]
and
\[
(\psi_n,\{l_n^X\}_{X\in\cC}): (B_n,G_n,I^{(n)}) \to (B_{n+1},G_{n+1},I^{(n+1)})
\]
be sequences of cocycle morphisms which we view as two inductive systems in the category $\C_{\cC}$.\footnote{As in previous sections, we will denote these morphisms by $(\phi_n,h_n)$ and $(\psi_n,l_n)$ for ease of notation.}

Consider two sequences of cocycle morphisms 
\[
(\kappa_n,\{r_n^X\}_{X\in\cC}): (B_n,G_n,I^{(n)}) \to (A_n,F_n,J^{(n)})
\]
and 
\[
(\theta_n,\{s_n^X\}_{X\in\cC}): (A_n,F_n,J^{(n)}) \to (B_{n+1},G_{n+1},I^{(n+1)})
\]
fitting into the family of not necessarily commutative diagrams
\begin{equation} \label{intertwiningdiagram}
\xymatrix{
\dots\ar[rr] && F_n(X) \ar[rd]^{s_n^X} \ar[rr]^{h_n^X} && F_{n+1}(X) \ar[r] \ar[rd] & \dots\\
\dots\ar[r] & G_n(X) \ar[ru]^{r_n^X} \ar[rr]^{l_n^X} && G_{n+1}(X) \ar[ru]^{r_{n+1}^X} \ar[rr] && \dots \quad .
}
\end{equation}

We will call the collection of diagrams \eqref{intertwiningdiagram} an \emph{approximate cocycle intertwining}, if the following hold: 
There exist an increasing sequence of finite sets $K_n\subset \Irr(\cC)$ containing $1_{\cC}$, finite sets $\mathcal{F}_n^X\subset F_n(X)$ and $\mathcal{G}_n^X\subset G_n(X)$ for any $X\in K_n$, and numbers $\delta_n>0$ satisfying 
\begin{enumerate}[label=\textit{(\roman*)}]
\item $\|l_n^X(x)-s_n^X(r_n^X(x))\|\leq\delta_n$
for all $X\in K_n$ and all $x\in \mathcal{G}_n^X$; \label{intert:1}
\item $\|h_n^X(x)-r_{n+1}^X(s_n^X(x))\|\leq \delta_n$ 
for all $X\in K_n$ and all $x\in\mathcal{F}_n^X$; \label{intert:2}
\item $h_n^X(\mathcal{F}_n^X)\subseteq \mathcal{F}_{n+1}^X$, $l_n^X(\mathcal{G}_n^X)\subseteq \mathcal{G}_{n+1}^X$, $r_n^X(\mathcal{G}_n^X)\subseteq \mathcal{F}_n^X$, and 
$s_n^X(\mathcal{F}_n^X)\subseteq \mathcal{G}_{n+1}^X$; \label{intert:3} 
\item $\bigcup\limits_{m>n} (h_{n,m}^X)^{-1}(\mathcal{F}_m^X)\subset F_n(X)$ and $\bigcup\limits_{m>n} (l_{n,m}^X)^{-1}(\mathcal{G}_m^X)\subset G_n(X)$ are dense for all $X\in K_n$ and all $n$; \label{intert:4}
\item $\bigcup\limits_{n\in\mathbb{N}}K_n=\Irr(\cC)$;\label{intert:5}
\item $\sum\limits_{n\in\mathbb{N}} \delta_n <\infty$. \label{intert:6}
\end{enumerate}
\end{defn}

\begin{rmk}
The conditions listed above are in the spirit of Elliott approximate intertwining arguments, as seen for example in {\cite[Definition $2.3.1$]{rordambook}} or {\cite[Definition $3.3$]{cocyclecategszabo}}. Note that applying $X=1_{\cC}$ in conditions \ref{intert:1}-\ref{intert:4} above, together with the \ref{intert:6} recover the usual assumptions in the two-sided Elliott intertwining argument.\ However, to boost the Elliott intertwining in the equivariant setting, one needs to preserve the equivariant structure.\ This is why conditions \ref{intert:1}-\ref{intert:4} above are phrased for irreducible elements in the category. These conditions resemble the assumptions needed in \cite[Definition 3.3]{cocyclecategszabo} in the case of a twisted group action, although recall that the topology on cocycle morphisms is different.

Moreover, \ref{intert:1}-\ref{intert:2} encode that the diagrams in \ref{intertwiningdiagram} approximately commute in the topology described in Lemma \ref{lemma: topologycocyclemor}. Furthermore, in \ref{intert:4} we are making crucial use of our definition of the functors $F_n$ and $G_n$ to assume that $F_n(X)$ and $G_n(X)$ have a countable dense subset for any $X\in\cC$ and any $n\in\mathbb N$.

Note that, in the light of Remark \ref{rmk: IrredDetLinearMaps}, it is enough to state the conditions above for $X\in\Irr(\cC)$.\ This will uniquely determine the family of linear maps $\{h^X\}_{X\in\cC}$. 
\end{rmk}

\begin{theorem}\label{twosidedintertwining}
Let $(A,F,J)$ and $(B,G,I)$ be inductive limits in $\C_{\cC}$ given by $$(A,F,J)=\lim\limits_{\longrightarrow}\Big\{(A_n,F_n,J^{(n)}),(\phi_n,\{h_n^X\}_{X\in\cC})\Big\}$$ and $$(B,G,I)=\lim\limits_{\longrightarrow}\Big\{(B_n,G_n,I^{(n)}),(\psi_n,\{l_n^X\}_{X\in\cC})\Big\}.$$ Let \[
(\kappa_n,\{r_n^X\}_{X\in\cC}): (B_n,G_n,I^{(n)}) \to (A_n,F_n,J^{(n)})
\]
and 
\[
(\theta_n,\{s_n^X\}_{X\in\cC}): (A_n,F_n,J^{(n)}) \to (B_{n+1},G_{n+1},I^{(n+1)})
\] be sequences of cocycle morphisms making \eqref{intertwiningdiagram} an approximate cocycle intertwining.

Then there exist mutually inverse cocycle conjugacies $(\theta,\{s^X\}_{X\in\cC}):(A,F,J)\to (B,G,I)$ and $(\kappa,\{r^X\}_{X\in\cC}):(B,G,I)\to (A,F,J)$ given by the formulae 
\begin{equation} \label{intertwining-theta}
\theta(\phi_{n,\infty}(a)) = \lim_{k\to\infty} (\psi_{k+1,\infty}\circ\theta_k\circ\phi_{n,k})(a),\quad a\in A_n
\end{equation}
\begin{equation} \label{intertwining-s}
s^X(h_{n,\infty}^X(x))=\lim_{k\to\infty}(l_{k+1,\infty}^X\circ s_k^X\circ h_{n,k}^X)(x),\quad X\in\cC,\quad x\in F_n(X)
\end{equation}
and
\begin{equation} \label{intertwining-kappa}
\kappa(\psi_{n,\infty}(b)) = \lim_{k\to\infty} (\phi_{k,\infty}\circ\kappa_k\circ\psi_{n,k})(b),\quad b\in B_n,
\end{equation}
\begin{equation} \label{intertwining-r}
r^X(l_{n,\infty}^X(x))=\lim_{k\to\infty}(h_{k,\infty}^X\circ r_k^X\circ l_{n,k}^X)(x),\quad X\in\cC,\quad x\in G_n(X).
\end{equation} The limits in the formulae \eqref{intertwining-s} and \eqref{intertwining-r} are taken in the topologies induced by the respective right inner products.
\end{theorem}

\begin{proof}
We will show that the limits in \eqref{intertwining-theta} and \eqref{intertwining-s} exist and that the pair $(\theta,\{s^X\}_{X\in\cC})$ is a cocycle morphism. Firstly, note that the limit in \eqref{intertwining-theta} exists and $\theta:A\to B$ is a $^*$-homomorphism by the general Elliott intertwining, as can be found for example in {\cite[Proposition 2.3.2]{rordambook}}.

We first show that the limit in \eqref{intertwining-s} exists for all $X\in\Irr(\cC).$ We can employ a similar argument to the one in {\cite[Proposition 2.3.2]{rordambook}}. Let $X\in K_n$ and given $x_n\in F_n(X)$, by condition \ref{intert:4} in Definition \ref{def: approx-int}, we may assume $x_n$ is contained in $\bigcup_{m>n} (h_{n,m}^X)^{-1}(\mathcal{F}_m^X)$. Moreover, we may assume that $h_{n,m}^X(x_n)\in \mathcal{F}_m^X$ for all $m$ greater than or equal to some $m_0$. 

Consider the possibly non-commutative diagram 
\begin{equation} 
\xymatrix{
F_n(X)\ar[r]^{h_{n,m}^X} & F_m(X) \ar[d]_{s_m^X} \ar[r]^{h_m^X} & F_{m+1}(X) \ar[d]^{s_{m+1}^X} & \\
&  G_{m+1}(X) \ar[ru]^{r_{m+1}^X} \ar[r]_{l_{m+1}^X} & G_{m+2}(X) \ar[r]_{\quad l_{m+2,\infty}^X} & G(X).
}
\end{equation} Using triangle inequality, it follows that 
\begin{align*}
\|(s_{m+1}^X\circ h_m^X)(x_m)-(l_{m+1}^X\circ s_m^X)(x_m)\| &\leq \|s_{m+1}^X(h_m^X(x_m)-(r_{m+1}^X\circ s_m^X)(x_m))\|\\ &+\|(s_{m+1}^X\circ r_{m+1}^X-l_{m+1}^X)(s_m^X(x_m))\|.
\end{align*} Also, by \ref{intert:3} $s_m^X(\mathcal{F}_m^X)\subset \mathcal{G}_{m+1}^X$, so we may combine \ref{intert:1}, \ref{intert:2}, and that $s_{m+1}^X$ is contractive to get that
 $$\|(s_{m+1}^X\circ h_m^X)(x_m)-(l_{m+1}^X\circ s_m^X)(x_m)\|\leq \delta_m+\delta_{m+1}, \quad \forall \ x_m\in \mathcal{F}_m^X\subseteq F_m(X). $$ 
Moreover, since $l_{m+2,\infty}^X$ is contractive, it follows that  
\begin{align*}
&\|(l_{m+2,\infty}^X\circ s_{m+1}^X\circ h_{n,m+1}^X)(x_n)-(l_{m+1,\infty}^X\circ s_m^X\circ h_{n,m}^X)(x_n)\|=\\ &\|l_{m+2,\infty}^X((s_{m+1}^X\circ h_m^X)(h_{n,m}^X(x_n))-(l_{m+1}^X\circ s_m^X)(h_{n,m}^X(x_n)))\|\\&\leq\delta_m+\delta_{m+1},
\end{align*} for all $m\geq m_0$ and all $x_n\in F_n(X)$.\ Using \ref{intert:6}, the sequence in \eqref{intertwining-s} is Cauchy and therefore convergent. Moreover \ref{intert:5} yields that $s^X:F(X)\to G(X)$ is a well-defined linear map for all $X\in\Irr(\cC)$. Extending by linearity we may define $s^X$ for any $X\in \cC$.
\par We claim that $(\theta,\{s^X\}_{X\in\Irr(\cC)})$ induces a cocycle morphism. For this, we check the conditions in Lemma \ref{linearmapspicture}. Given $X,Y\in\Irr(\cC)$, consider a morphism $f\in\Hom(X,Y)$. Since $\Hom(X,Y)$ is $0$ if $X\neq Y$ and $\Hom(X,X)\cong\mathbb C$, it follows immediately that $G(f)\circ s^X=s^Y\circ F(f)$. 

Next, we want to show that $\theta(\langle x,y\rangle_A)=\langle s^X(x),s^X(y)\rangle_B$ for any $x,y\in F(X)$. It suffices to check this condition on the dense subset $\bigcup_{n\geq 1}h_{n,\infty}^X(F_n(X))$. If $x_n,y_n\in F_n(X)$ and $k>n$, as $l_{k+1,\infty}^X\circ s_k^X\circ h_{n,k}^X$ satisfies condition \ref{item:linearmap3} of Lemma \ref{linearmapspicture} one has 

\begin{adjustbox}{max width=\textwidth}
$\langle(l_{k+1,\infty}^X\circ s_k^X\circ h_{n,k}^X)(x_n), (l_{k+1,\infty}^X\circ s_k^X\circ h_{n,k}^X)(y_n)\rangle=(\psi_{k+1,\infty}\circ\theta_k\circ\phi_{n,k})(\langle x_n,y_n\rangle).$
\end{adjustbox}

Taking the limit as $k$ goes to infinity and using the formulae in \eqref{intertwining-theta} and \eqref{intertwining-s},
$$\theta(\phi_{n,\infty}(\langle x_n,y_n\rangle))=\langle s^X(h_{n,\infty}^X(x_n)),s^X(h_{n,\infty}^X(y_n))\rangle.$$Moreover, $\phi_{n,\infty}(\langle x_n,y_n\rangle)= \langle h_{n,\infty}^X(x_n),h_{n,\infty}^X(y_n)\rangle$, so $\theta(\langle x,y\rangle)=\langle s^X(x),s^X(y)\rangle$ for any $x,y$ in the dense subset $\bigcup_{n\geq 1}h_{n,\infty}^X(F_n(X))$.

Then, we check \ref{item:linearmap4} of Lemma \ref{linearmapspicture}. Let $X,Y\in\cC$. Since $(\psi_{k+1,\infty}\circ\theta_k\circ\phi_{n,k},l_{k+1,\infty}\circ s_k\circ h_{n,k})$ is a cocycle morphism, the diagram \[
\begin{tikzcd}
F_n(Y)\boxtimes F_n(X)
\arrow[swap]{d}{(l_{k+1,\infty}^Y\circ s_k^Y\circ h_{n,k}^Y)\boxtimes (l_{k+1,\infty}^X\circ s_k^X\circ h_{n,k}^X)}
\arrow{r}{J_{X,Y}^{(n)}}
& F_n(X\otimes Y)\arrow{d}{l_{k+1,\infty}^{X\otimes Y}\circ s_k^{X\otimes Y}\circ h_{n,k}^{X\otimes Y}}
 \\
G(Y)\boxtimes G(X)
\arrow{r}{I_{X,Y}}
& G(X\otimes Y)
\end{tikzcd}
\] commutes. Then, by taking the limit as $k$ goes to infinity, it follows that $I_{X,Y}\circ ((s^Y\circ h_{n,\infty}^Y)\boxtimes (s^X\circ h_{n,\infty}^X))=s^{X\otimes Y}\circ h_{n,\infty}^{X\otimes Y}\circ J_{X,Y}^{(n)}.$ Note that by the construction of the map $J_{X,Y}$ in Lemma \ref{inductivelimitfunctor}, $$J_{X,Y}\circ(h_{n,\infty}^Y\boxtimes h_{n,\infty}^X)=h_{n,\infty}^{X\otimes Y}\circ J_{X,Y}^{(n)}.$$ Thus, $$I_{X,Y}\circ(s^Y\boxtimes s^X)\circ (h_{n,\infty}^Y\boxtimes h_{n,\infty}^X)=s^{X\otimes Y}\circ J_{X,Y}\circ (h_{n,\infty}^Y\boxtimes h_{n,\infty}^X).$$ By density it now follows that $I_{X,Y}\circ(s^Y\boxtimes s^X)=s^{X\otimes Y}\circ J_{X,Y}$. Hence, as $\cC$ is semisimple, $(\theta,\{s^X\}_{X\in\Irr(\cC)})$ induces a cocycle morphism (see Remark \ref{rmk: IrredDetLinearMaps}).\ This follows as the map $s^X$ will be given by a direct sum of linear maps corresponding to irreducible objects. Since the same holds for $h_{n,\infty}^X, l_{k+1,\infty}^X, s_k^X,$ and $h_{n,k}^X$, and the limit preserves this decomposition, the formula in \eqref{intertwining-s} holds for any $X\in\cC$. 

It follows in the same way that $(\kappa,\{r^X\}_{X\in\cC})$ given by the formulae in \eqref{intertwining-kappa} and \eqref{intertwining-r} yields a well-defined cocycle morphism. Moreover, the fact that $\theta$ and $\kappa$ are mutually inverse isomorphisms follows from {\cite[Proposition 2.3.2]{rordambook}}. 

Finally, it remains to check that for any $X\in\cC$, $r^X\circ s^X=\id_{F(X)}$ and $s^X\circ r^X=\id_{G(X)}$. It suffices to show that for any $n\geq 1$, $r^X\circ s^X\circ h_{n,\infty}^X=h_{n,\infty}^X$. For any $k>n$, by \eqref{intertwining-r}, we have that 
\begin{equation}\label{eqn:mistake}
r^X\circ l_{k+1,\infty}^X\circ s_k^X\circ h_{n,k}^X=\lim\limits_{m\to\infty}(h_{m,\infty}^X\circ r_m^X\circ l_{k+1,m}^X\circ s_k^X\circ h_{n,k}^X).
\end{equation}
Using conditions \ref{intert:1},\ref{intert:2} and \ref{intert:6} of Definition \ref{def: approx-int}, it follows that the right hand side evaluated at $x\in (h_{n,k}^X)^{-1}(\mathcal{F}_k^X)$ converges to $h_{n,\infty}(x)$ uniformly on $k$ as $k$ tends to infinity.\ Now, taking the limit as $k$ goes to infinity in \eqref{eqn:mistake} and using condition \ref{intert:4} of Definition \ref{def: approx-int}, we get that $r^X\circ s^X\circ h_{n,\infty}^X=h_{n,\infty}^X$. Similarly, $s^X\circ r^X=\id_{G(X)}$, so $(\theta,\{s^X\}_{X\in\cC}):(A,F,J)\to (B,G,I)$ and $(\kappa,\{r^X\}_{X\in\cC}):(B,G,I)\to (A,F,J)$ are mutually inverse cocycle conjugacies.
\end{proof}

We now use Theorem \ref{twosidedintertwining} to show that if we assume that the diagrams in \eqref{intertwiningdiagram} commute up to approximate unitary equivalence, then there exist mutually inverse cocycle conjugacies as in Theorem \ref{twosidedintertwining}. The proof follows in a similar fashion to {\cite[Theorem 3.6]{cocyclecategszabo}} and {\cite[Corollary 2.3.3]{rordambook}}.

\begin{theorem}\label{intertapproxuniteq}
Let $\cC$ be a semisimple C$^*$-tensor category with countably many isomorphism classes of simple objects. Let $(F_n,J^{(n)}): \cC\curvearrowright A_n$ and $(G_n,I^{(n)}): \cC\curvearrowright B_n$ be sequences of actions on separable $\C$-algebras. Let
\[
(\phi_n,\{h_n^X\}_{X\in\cC}): (A_n,F_n,J^{(n)}) \to (A_{n+1},F_{n+1},J^{(n+1)})
\]
and
\[
(\psi_n,\{l_n^X\}_{X\in\cC}): (B_n,G_n,I^{(n)}) \to (B_{n+1},G_{n+1},I^{(n+1)})
\]
be sequences of cocycle morphisms, in the sense of Lemma \ref{linearmapspicture}, which we view as two inductive systems in the category $\C_{\cC}$.

Consider two sequences of extendible cocycle morphisms 
\[
(\kappa_n,\{r_n^X\}_{X\in\cC}): (B_n,G_n,I^{(n)}) \to (A_n,F_n,J^{(n)})
\]
and 
\[
(\theta_n,\{s_n^X\}_{X\in\cC}): (A_n,F_n,J^{(n)}) \to (B_{n+1},G_{n+1},I^{(n+1)})
\]
fitting into the not necessarily commutative collection of diagrams
\begin{equation}\label{eq: DiagThmUnitEq}
\xymatrix{
\dots\ar[rr] && F_n(X) \ar[rd]^{s_n^X} \ar[rr]^{h_n^X} && F_{n+1}(X) \ar[r] \ar[rd] & \dots\\
\dots\ar[r] & G_n(X) \ar[ru]^{r_n^X} \ar[rr]^{l_n^X} && G_{n+1}(X) \ar[ru]^{r_{n+1}^X} \ar[rr] && \dots \quad .
}
\end{equation}Suppose that $$(\psi_n,l_n)\approx_u (\theta_n, s_n)\circ (\kappa_n, r_n)\quad and \quad (\phi_n, h_n)\approx_u (\kappa_{n+1}, r_{n+1})\circ (\theta_n, s_n)$$ for all $n\in\mathbb{N}$.\ Then there exist mutually inverse cocycle conjugacies $(\theta,\{s^X\}_{X\in\cC}):(A,F,J)\to (B,G,I)$ and $(\kappa,\{r^X\}_{X\in\cC}):(B,G,I)\to (A,F,J)$. Moreover, if $\phi_n$ and $\psi_n$ are extendible for any $n\in\mathbb{N}$, then $$(\theta,s)\circ (\phi_{n,\infty},h_{n,\infty})\approx_u (\psi_{n+1,\infty},l_{n+1,\infty})\circ (\theta_n,s_n)$$ and $$(\kappa,r)\circ (\psi_{n,\infty},l_{n,\infty})\approx_u (\phi_{n,\infty},h_{n,\infty})\circ (\kappa_n,r_n).$$
\end{theorem}

\begin{proof}
This will follow as an application of Theorem \ref{twosidedintertwining}. For this, it suffices to show that we can obtain a collection of diagrams as in Definition \ref{def: approx-int}. The strategy is to replace the families of cocycle morphisms $(\kappa_n,r_n)$ and $(\theta_n,s_n)$ by unitary perturbations $(\eta_n, R_n)$ and $(\zeta_n, S_n)$ respectively such that the diagrams in \eqref{eq: DiagThmUnitEq} become an approximate cocycle intertwining.

Let $\delta_n=2^{-n}$ for any $n\in\mathbb{N}$ and $K_n$ an increasing sequence of finite sets containing $1_{\cC}$ such that $\bigcup_{n\in\mathbb{N}}K_n=\Irr(\cC)$. For any $n\in\mathbb{N}$, and any $X\in K_n$, we can choose $t_n^X\in\mathbb{N}$ and finite sets $\{f_{m,n}^X\}_{1\leq m\leq t_n^X}\subset F_n(X)$ and $\{g_{m,n}^X\}_{1\leq m\leq t_n^X}\subset G_n(X)$ such that the inclusions
\begin{equation}\label{eqn:density}
\begin{adjustbox}{max width=\textwidth}
    $\bigcup\limits_{k>n}(h_{n,k}^X)^{-1}(\{f_{m,k}^X\}_{1\leq m\leq t_k^X})\subset F_n(X)
\quad \AND \quad \bigcup\limits_{k>n}(l_{n,k}^X)^{-1}(\{g_{m,k}^X\}_{1\leq m\leq t_k^X})\subset G_n(X)$
\end{adjustbox}
\end{equation} are dense for all $n\in\mathbb{N}$. To simplify notation, we will write $\{f_{m,n}^X\}$ to denote the set $\{f_{m,n}^X\}_{1\leq m\leq t_n^X}$.

Set $(\eta_1, R_1)=(\kappa_1,r_1)$, $\mathcal{G}_1^X=\{g_{m,1}^X\}\subset G_1(X)$, and $\mathcal{F}_1^X=\{f_{m,1}^X\}\cup r_1^X(\mathcal{G}_1^X)$ for any $X\in K_1$. Since $$(\psi_1,l_1)\approx_u (\theta_1, s_1)\circ (\kappa_1, r_1)= (\theta_1, s_1)\circ (\eta_1, R_1),$$ we can find a unitary $u_1\in\mathcal{U}(\mathcal{M}(B_2))$ such that if we set $(\zeta_1, S_1)=\Ad(u_1)\circ (\theta_1,s_1)$, one has that for any $X\in K_1$
$$\max\limits_{x\in\mathcal{G}_1^X}\|l_1^X(x)-S_1^X(R_1^X(x))\|\leq\delta_1.$$

At the next stage, let $\mathcal{G}_2^X=\{g_{m,2}^X\}\cup S_1^X(\mathcal{F}_1^X)\cup l_1^X(\mathcal{G}_1^X)$ for any $X\in K_1$ and $\mathcal{G}_2^X=\{g_{m,2}^X\}$ for any $X\in K_2\setminus K_1$. Using the assumption $(\phi_1, h_1)\approx_u (\kappa_2, r_2)\circ (\theta_1, s_1)$, that $(\theta_1,s_1)$ is unitarily equivalent to $(\zeta_1,S_1)$, and that $\kappa_2$ is extendible, it follows that $(\phi_1, h_1)\approx_u (\kappa_2, r_2)\circ (\zeta_1, S_1)$ (see Remark \ref{rmk: VerticalComp}). Therefore, there exists a unitary $v_2\in\mathcal{U}(\mathcal{M}(A_2))$ such that if we set $(\eta_2,R_2)=\Ad(v_2)\circ (\kappa_2,r_2)$, we have that for any $X\in K_1$ $$\max\limits_{x\in\mathcal{F}_1^X}\|h_1^X(x)-R_2^X(S_1^X(x))\|\leq\delta_1.$$ Then, we let $\mathcal{F}_2^X=\{f_{m,2}^X\}\cup R_2^X(\mathcal{G}_2^X)\cup h_1^X(\mathcal{F}_1^X)$ for any $X\in K_1$ and $\mathcal{F}_2^X=\{f_{m,2}^X\}\cup R_2^X(\mathcal{G}_2^X)$ for any $X\in K_2\setminus K_1$. 

By continuing inductively we construct finite sets
\begin{equation}\label{eq: FiniteSets1}
\mathcal{G}_n^X=\{g_{m,n}^X\}\cup S_{n-1}^X(\mathcal{F}_{n-1}^X)\cup l_{n-1}^X(\mathcal{G}_{n-1}^X)\subset G_n(X),
\end{equation}
\begin{equation}\label{eq: FiniteSets2}
\mathcal{F}_n^X=\{f_{m,n}^X\}\cup R_n^X(\mathcal{G}_n^X)\cup h_{n-1}^X(\mathcal{F}_{n-1}^X)\subset F_n(X)
\end{equation} for any $X\in K_{n-1}$ and
\begin{equation}\label{eq: FiniteSets3}
\mathcal{G}_n^X=\{g_{m,n}^X\}\subset G_n(X),
\end{equation}
\begin{equation}\label{eq: FiniteSets4}
\mathcal{F}_n^X=\{f_{m,n}^X\}\cup R_n^X(\mathcal{G}_n^X)\subset F_n(X)
\end{equation} for any $X\in K_n\setminus K_{n-1}$ and unitaries $v_n\in\mathcal{U}(\mathcal{M}(A_n))$ and $u_n\in\mathcal{U}(\mathcal{M}(B_{n+1}))$ such that setting $(\eta_n,R_n)=\Ad(v_n)\circ (\kappa_n,r_n)$ and $(\zeta_n,S_n)=\Ad(u_n)\circ (\theta_n,s_n)$ one has for any $X\in K_n$
\begin{equation}\label{eq: CommTri1}
\max\limits_{x\in\mathcal{G}_n^X}\|l_n^X(x)-S_n^X(R_n^X(x))\|\leq\delta_n   
\end{equation} and
\begin{equation}\label{eq: CommTri2}
\max\limits_{x\in\mathcal{F}_n^X}\|h_n^X(x)-R_{n+1}^X(S_n^X(x))\|\leq\delta_n.    
\end{equation}

We claim that the diagram 
\begin{equation}\label{eq: NewIntertDiagram}
\xymatrix{
\dots\ar[rr] && F_n(X) \ar[rd]^{S_n^X} \ar[rr]^{h_n^X} && F_{n+1}(X) \ar[r] \ar[rd] & \dots\\
\dots\ar[r] & G_n(X) \ar[ru]^{R_n^X} \ar[rr]^{l_n^X} && G_{n+1}(X) \ar[ru]^{R_{n+1}(X)} \ar[rr] && \dots \quad .
}
\end{equation} is an approximate cocycle intertwining in the sense of Definition \ref{def: approx-int}. Conditions \ref{intert:1} and \ref{intert:2} is ensured by \eqref{eq: CommTri1} and \eqref{eq: CommTri2} respectively. Then, condition \ref{intert:3} follows by \eqref{eq: FiniteSets1}, \eqref{eq: FiniteSets2},\eqref{eq: FiniteSets3}, and \eqref{eq: FiniteSets4}. Moreover, condition \ref{intert:4} follows by \eqref{eq: FiniteSets1}, \eqref{eq: FiniteSets2},\eqref{eq: FiniteSets3}, \eqref{eq: FiniteSets4}, and the choice of finite sets $\{f_{m,n}^X\}, \{g_{m,n}^X\}$ from \eqref{eqn:density}.\ Then, \ref{intert:5} is ensured by the choice of finite sets $K_n$, while \ref{intert:6} is satisfied because the sum $\sum_n2^{-n}$ converges.

Hence, by Theorem \ref{twosidedintertwining} applied to the family of diagrams in \eqref{eq: NewIntertDiagram}, there exists mutually inverse cocycle conjugacies $(\theta,\{s^X\}_{X\in\cC}):(A,F,J)\to (B,G,I)$ and $(\kappa,\{r^X\}_{X\in\cC}):(B,G,I)\to (A,F,J)$. Moreover, Theorem \ref{twosidedintertwining} also gives that \begin{equation} \label{eq: intertwining-theta}
\theta(\phi_{n,\infty}(a)) = \lim_{k\to\infty} (\psi_{k+1,\infty}\circ\zeta_k\circ\phi_{n,k})(a),\quad a\in A_n
\end{equation} and
\begin{equation} \label{eq: intertwining-s}
s^X(h_{n,\infty}^X(x))=\lim_{k\to\infty}(l_{k+1,\infty}^X\circ S_k^X\circ h_{n,k}^X)(x),\quad X\in\cC,\quad x\in F_n(X).
\end{equation}

Now, assume that $\phi_n$ and $\psi_n$ are extendible. Recall that $$(\psi_m,l_m)\approx_u (\theta_m, s_m)\circ (\kappa_m, r_m)\ \text{and} \ (\phi_m, h_m)\approx_u (\kappa_{m+1}, r_{m+1})\circ (\theta_m, s_m)$$ and $(\zeta_m,S_m)=\Ad(u_m)\circ (\theta_m,s_m)$ for any $m\in\mathbb{N}$. Then, for all $n\geq 1$ and all $k>n$ we have that 
\[
\begin{array}{cl}
\multicolumn{2}{l}{ (\psi_{n+1,k+1},l_{n+1,k+1})\circ (\theta_n,s_n) } \\
\approx_u& (\theta_k,s_k)\circ(\kappa_k,r_k)\circ(\theta_{k-1},s_{k-1})\circ\ldots\circ(\kappa_{n+1},r_{n+1})\circ (\theta_n,s_{n}) \\
\approx_u& (\theta_k,s_{k})\circ (\phi_{n,k},h_{n,k}) \\
\approx_u& (\zeta_k,S_{k})\circ (\phi_{n,k},h_{n,k}).
\end{array}
\] Composing with $(\psi_{k+1,\infty},l_{k+1,\infty})$ then by \ref{rmk: ExtMor} we have that $$(\psi_{n+1,\infty},l_{n+1,\infty})\circ (\theta_n,s_n)\approx_u (\psi_{k+1,\infty},l_{k+1,\infty})\circ (\zeta_k, S_k)\circ(\phi_{n,k},h_{n,k}).$$ Then, by \eqref{eq: intertwining-theta} and \eqref{eq: intertwining-s}, one gets that $$(\theta,s)\circ (\phi_{n,\infty},h_{n,\infty})\approx_u(\psi_{n+1,\infty},l_{n+1,\infty})\circ (\theta_n,s_n).$$ 
The fact that $$(\kappa,r)\circ (\psi_{n,\infty},l_{n,\infty})\approx_u (\phi_{n,\infty},h_{n,\infty})\circ (\kappa_n,r_n)$$ follows analogously.
\end{proof}

As a corollary of Theorem \ref{intertapproxuniteq}, we obtain the following result.

\begin{cor}\label{intertidentity}
Let $\cC$ be a semisimple C$^*$-tensor category with countably many isomorphism classes of simple objects. Let $(F,J): \cC\curvearrowright A$ and $(G,I): \cC\curvearrowright B$ be actions on separable $\C$-algebras. Let
\[
(\phi, h): (A,F,J) \to (B,G,I) \quad
and \quad
(\psi, l): (B,G,I) \to (A,F,J)
\] be two extendible cocycle morphisms such that $$\id_A\approx_u (\psi,l)\circ (\phi, h)\quad and \quad \id_B\approx_u (\phi,h)\circ (\psi, l).$$ Then there exist mutually inverse cocycle conjugacies \[
(\Phi, H): (A,F,J) \to (B,G,I) \quad
and \quad
(\Psi, L): (B,G,I) \to (A,F,J)
\] such that $$(\Phi, H)\approx_u (\phi,h) \quad and \quad (\Psi, L)\approx_u (\psi,l).$$
\end{cor}

\begin{proof}
This is a direct application of Theorem \ref{intertapproxuniteq} with $A_n=A$, $B_n=B$, $\phi_n=\id_A$, $\psi_n=\id_B$, $k_n=\psi$, and $\theta_n=\phi$ for all $n\in\mathbb{N}$.
\end{proof}

\begin{rmk}
Note that in Corollary \ref{intertidentity}, if the approximate unitary equivalences are realised by unitaries in the minimal unitisations, then we may drop the assumption of extendibility on the connecting morphisms. This follows as extendibility in Theorem \ref{intertapproxuniteq} is only required to evaluate a morphism on a given unitary in the multiplier algebra. However, there is no issue in doing so when the unitary is in the minimal unitisation.
\end{rmk}

\subsection{Intertwining through reparametrisation}

Intertwining through reparametrisation is a type of intertwining argument commonly employed in the classification programme of $\C$-algebras and $\C$-dynamics. In broad terms, if we want to prove that a morphism $\theta:A\to B_\infty$ is unitarily equivalent to a morphism $\psi:A\to B_\infty$ which factors through $B$, then it suffices to check that $\theta$ is invariant under reparametrisations. This type of result appears for example in \cite{O2embthm,cocyclecategszabo} and it is used in successful classification results in \cite{O2embthm, DynamicalKP, ClassifMorphisms23}.
\par If $A$ is a separable $C^*$-algebra and $\eta: \mathbb{N}\to \mathbb{N}$ is any map with $\lim\limits_{n\to\infty}\eta(n)=\infty$, then it induces an endomorphism $\eta^*$ on $A_\infty$ via $\eta^*((a_n)_n)=(a_{\eta(n)})_n$. Moreover, if a C$^*$-tensor category $\cC$ acts on a C$^*$-algebra $A$ by a finite rank triple $(A,F,J)$, it follows from Lemma \ref{lemma: TensorFunctor} that there is an induced action on $A_\infty$ by the triple $(A_\infty, F_\infty, J^{\infty})$ (see also Remark \ref{rmk: FinGen}). 

Then, a straightforward checking of the conditions in Lemma \ref{linearmapspicture} shows that $\eta$ induces a cocycle morphism $(\eta^*, r):(A_\infty, F_\infty, J^{\infty})\to (A_\infty, F_\infty, J^{\infty})$, where $r^X:F_\infty(X)\to F_\infty(X)$ is given by $r^X((\xi_n)_n)=((\xi_{\eta(n)})_n)$ for any $X\in \cC$ and any $(\xi_n)_n\in F_\infty(X)$. Using this construction, we will prove an intertwining argument concerning maps into sequence algebras. First we need a preparatory lemma.

\begin{lemma}\label{lemma: IntThroughRep}
Let $(F,J): \cC\curvearrowright A$ and $(G,I): \cC\curvearrowright B$ be actions of a C$^*$-tensor category $\cC$ on separable $\C$-algebras with $B$ unital and $(G,I)$ finite rank. Suppose that $(\phi,h):(A,F,J)\to (B_\infty, G_\infty, I^{\infty})$ is a cocycle morphism such that for any map $\eta: \mathbb{N}\to \mathbb{N}$ with $\lim\limits_{n\to\infty}\eta(n)=\infty$, the cocycle morphisms $(\phi,h)$ and $(\eta^*,r)\circ(\phi,h)$ are approximately unitarily equivalent. For each $X\in\cC$, let $(h_n^X)_n\colon F(X)\rightarrow l^{\infty}(\N,G(X))$ be a linear lift of $h^X$. Then, for every finite set $K\subseteq \Irr(\cC)$ containing $1_{\cC}$, finite sets $\mathcal{F}^X\subseteq F(X)$ for $X\in K$, $\epsilon>0$, and $m\in\mathbb{N}$, there is an integer $k\geq m$ such that for every integer $n\geq k$ there is a unitary $u\in B$ for which $$\|u\rhd h_n^X(x)\lhd u^* - h_k^X(x)\|<\epsilon, \quad X\in K, \quad x\in\mathcal{F}^X.$$
\end{lemma}

\begin{proof}
We prove this by contradiction. Suppose that there exists a finite set $K\subseteq \Irr(\cC)$ containing $1_{\cC}$, finite sets $\mathcal{F}^X\subseteq F(X)$ for $X\in K$, $\epsilon>0$, and $m\in\mathbb{N}$ such that for every $k\geq m$, there exists $n_k\geq k$ for which 
\begin{equation}\label{eq: IntRep1}
\max\limits_{x\in \mathcal{F}^X}\|u_k\rhd h_{n_k}^X(x)\lhd u_k^*-h_k^X(x)\|\geq \epsilon,
\end{equation}
for any $X\in K$, $x\in \mathcal{F}^X$, and every unitary $u_k\in B$. Let $\eta:\mathbb{N}\to\mathbb{N}$ be the map $\eta(k)=n_k$ whenever $k\geq m$ and $\eta(k)=1$ for $k<m$. As $\eta(k)\geq k$ for $k\geq m$, it follows that $\lim\limits_{k\to\infty}\eta(k)=\infty$. Moreover, $(\phi,h)$ and $(\eta^*,r)\circ(\phi,h)$ are approximately unitarily equivalent, so there exists a unitary $u\in B_\infty$ for which $$\|u\rhd r^X(h^X(x))\lhd u^* - h^X(x)\|<\epsilon, \quad X\in K, \quad x\in\mathcal{F}^X.$$ If we let $(u_k)_{k\geq 1}$ to be a sequence of unitaries lifting $u$, then $$\limsup\limits_{k\to\infty}\|u_k\rhd h_{n_k}^X(x)\lhd u_k^*-h_k^X(x)\|<\epsilon$$ for all $X\in K$ and $x\in \mathcal{F}^X$. But this contradicts \eqref{eq: IntRep1}, so we reach the conclusion.
\end{proof}

\begin{theorem}\label{thm: IntThroughRepar}
Let $\cC$ be a semisimple C$^*$-tensor category with countably many isomorphism classes of simple objects. Let $(F,J): \cC\curvearrowright A$ and $(G,I): \cC\curvearrowright B$ be actions on separable $\C$-algebras with $B$ unital and $(G,I)$ finite rank. Let $(\phi,h):(A,F,J)\to (B_\infty, G_\infty, I^{\infty})$ be a cocycle morphism. Then the following are equivalent:

\begin{enumerate}
    \item $(\phi,h)$ is unitarily equivalent to a cocycle morphism $(\psi,l):(A,F,J)\to (B,G,I)$;\footnote{We view $\psi$ as a cocycle morphism into $B_\infty$ after composing with the canonical inclusion $\iota:B\rightarrow B_{\infty}$, which itself is canonically a cocycle morphism.}
    \item for any map $\eta: \mathbb{N}\to \mathbb{N}$ with $\lim\limits_{n\to\infty}\eta(n)=\infty$, the cocycle morphisms $(\phi,h)$ and $(\eta^*,r)\circ(\phi,h)$ are approximately unitarily equivalent.
\end{enumerate}
\end{theorem}

\begin{proof}
Let us first show that (i) implies (ii). Suppose that $u\in B_\infty$ is a unitary such that $(\Ad(u),h_u)\circ (\psi,l)=(\phi,h)$.\footnote{Recall the definition of the family of linear maps $h_u$ from Lemma \ref{lemma: linearmapsAd(u)}.} Let $\eta: \mathbb{N}\to \mathbb{N}$ be any map with $\lim\limits_{n\to\infty}\eta(n)=\infty$. Since $r^X\circ l^X=l^X$ for any $X\in \cC$, it follows that $$r^X\circ h^X= r^X\circ h_u^X\circ l^X= h_{\eta^*(u)}^X\circ r^X\circ l^X= h_{\eta^*(u)u^*}^X\circ h^X.$$ Since this holds for any $X\in \cC$, we get that $$(\Ad(\eta^*(u)u^*), h_{\eta^*(u)u^*})\circ (\phi, h)= (\eta^*,r)\circ (\phi, h).$$ Hence, $(\phi,h)$ and $(\eta^*,r)\circ(\phi,h)$ are unitarily equivalent. 

Suppose now that $(\phi,h)$ and $(\eta^*,r)\circ(\phi,h)$ are approximately unitarily equivalent for any $\eta:\N\rightarrow \N$ such that $\eta(n)$ converges to infinity. Let $K_n\subseteq \Irr(\cC)$ be an increasing sequence of finite sets containing $1_{\cC}$ such that $\bigcup_{n=1}^{\infty}K_n=\Irr(\cC)$, and $\mathcal{F}_n^X\subseteq F(X)$ be an increasing sequence of finite sets such that $\bigcup_{n=1}^{\infty}\mathcal{F}_n^X$ is dense in $F(X)$ for any $X\in\Irr(\cC)$. Let $(h_n^X)_n:F(X)\rightarrow l^{\infty}(\N,G(X))$ be a linear lifting. Recursively applying Lemma \ref{lemma: IntThroughRep} to $K=K_n$, $\mathcal{F}^X=\mathcal{F}_n^X$, $\epsilon= \frac{1}{2^n}$, one may pick $k_0=1<k_1<k_2<\ldots$ and unitaries $u_n\in B$ for $n\in \N$ such that
\begin{equation}\label{eq: Cauchy}
\|u_n\rhd h_{k_n}^X(x)\lhd u_n^*-h_{k_{n-1}}^X(x)\|<\frac{1}{2^n}, \quad X\in K_n, \quad x\in \mathcal{F}_n^X.
\end{equation}

Let $v_n=u_1u_2\ldots u_n$, we claim that $l^X(x):=\lim\limits_{n\to\infty}v_n\rhd h_{k_n}^X(x)\lhd v_n^*$ is well-defined for any $X\in\Irr(\cC)$ and any $x\in F(X)$. First let $y\in \bigcup_{n\in \N}\F_n^X$ for $X\in \Irr(\cC)$ and $\varepsilon>0$. Then there exists $n\in \N$ such that $X\in K_n$, $y\in \F_n^X$ and $2^{-n}<\varepsilon$. In particular, for any $j>l\geq n$ we may use \eqref{eq: Cauchy} repeatedly, together with triangle inequality, to achieve that
\begin{align*}
\|v_j\rhd h_{k_j}^X(y)\lhd v_j^*-v_l\rhd h_{k_l}^X(y)\lhd v_l^*\| &\leq \sum\limits_{i=l+1}^j 2^{-i} <\epsilon.
\end{align*}
As $\varepsilon$ is arbitrary $v_n\rhd h_{k_n}^X(y)\lhd v_n^*$ is Cauchy for any $X\in\Irr(\cC)$ and any $y\in \bigcup_{n\in \N}\F_n^X$.  This implies that, letting $(v_n)_n=V\in \U(B_\infty)$, $\eta:\N\rightarrow \N$ be given by $\eta(n)=k_n$ and their induced cocycle morphisms by $(\Ad(V),h_V)$ and $(\eta^*,r)$, 
\[
(\Ad(V),h_V)\circ(\eta^*,r)\circ(\phi,h):(A,F,J)\rightarrow (B_\infty,G_\infty,I^\infty)
\]
is a cocycle morphism such that it sends a dense subspace of $F(X)$ into the closed subspace $G(X)\subset G(X)_\infty$ for each $X\in \Irr(\cC)$. Thus 
\[
(h_V)^X\circ r^X\circ h^X:F(X)\to G(X)\subset G(X)_\infty
\]
by continuity for any $X\in \Irr(\cC)$.
This precisely shows that \[\lim_{n\rightarrow\infty} v_n\rhd h_{k_n}^X(x)\lhd v_n^*\] exists for all $x\in F(X)$. Let $\psi:A\to B$ be the linear map given by $\psi=l^{1_\cC}$. By construction we have that 
$$(\psi,l)=(\Ad(V),h_V)\circ (\eta^*,r)\circ (\phi,h)$$ 
for any $X\in\Irr(\cC)$.\ Hence, $\psi$ is a $^*$-homomorphism and the pair $(\psi,l)$ induces a cocycle morphism.\ Extending by linearity, we can define $l^X$ in the same way for any $X\in \cC$ and it is straightforward to see that these maps are linear.\ Moreover, $(\psi, l)$ is approximately unitarily equivalent to $(\phi,h)$. Finally, to show that the approximate unitary equivalence of $(\psi,l)$ and $(\phi,h)$ implies unitary equivalence follows as in \cite[Lemma~4.1]{O2embthm}.
\end{proof}

\begin{rmk}
Note that in Theorem~\ref{thm: IntThroughRepar}, if the approximate unitary equivalences are realised by unitaries in the minimal unitisations, then we may drop the assumption of unitality on $B$.
Moreover, as noted in Remark \ref{rmk: FinGen}, the assumption that $G(X)$ are finite rank in Theorem \ref{thm: IntThroughRepar} is immediate whenever the acting category $\cC$ is a unitary tensor category and $B$ is unital.
\end{rmk}

\section{One sided intertwining arguments}\label{section: onesidedintert}

We start this section by showing a tensor category equivariant adaptation of the classical one-sided intertwining argument (see \cite[Proposition 2.3.5]{rordambook}). A group equivariant version of this result can be found in \cite[Proposition 4.3]{cocyclecategszabo}. We end this section by proving an asymptotic Elliott two-sided intertwining (see Theorem \ref{thm: AsymptIntertIdentity}). First note that as $\Ad(v)$ is an isometric $^*$-homomorphism for any unitary $v\in \U(\M(A))$, injectivity of a cocycle morphism, which coincides with the injectivity of it's underlying $^*$-homomorphism, is preserved under both approximate and asymptotic unitary equivalence.

\begin{theorem}\label{thm: onesidedintert}
    Let $\cC$ be a semisimple C$^*$-tensor category with countably many isomorphism classes of simple objects.\ Let $(F,J):\cC\curvearrowright A$ and $(G,I):\cC\curvearrowright B$ be actions on separable C$^*$-algebras and $(\varphi,h):(A,F,J)\rightarrow (B,G,I)$ an injective cocycle morphism.\ Then $(\varphi,h)$ is asymptotically unitarily equivalent to a cocycle conjugacy if and only if:
    \par For all $\varepsilon>0$ and finite sets $K\subset\Irr(\cC)$ containing $1_{\cC}$, $\F^X\subset F(X)$ and $\G^X\subset G(X)$ there exists a strictly continuous path $z:[0,1]\rightarrow \U(\M(B))$ with $z_0=1$ such that
    \begin{enumerate}[label=\textbf{A.\arabic*}]
        \item $\sup\limits_{0\leq t\leq 1}\|z_t\rhd h^X(x)\lhd z_t^*-h^X(x)\|\leq \varepsilon$ for all $X\in K, x\in \F^X$,\label{item:1}
        \item $\dist(z_1^*\rhd y\lhd z_1,h^X(F(X)))\leq \varepsilon$ for all $X\in K$ and $y\in \G^X$.\footnote{Note that the injectivity assumption is redundant for the only if statement as it follows from the hypothesis. However, this assumption it is still required for the if direction.}\label{item:2}        
    \end{enumerate}
\end{theorem}
\begin{proof}
The proof of this Theorem will follow the proof of \cite[Proposition 4.3]{cocyclecategszabo} closely. We start by showing the ``only if" statement. Let $(\Phi,H)$ be a cocycle conjugacy such that $(\Phi,H)\asym (\varphi,h)$. By Lemma \ref{lemma: asymptunitconj} there exists a strictly continuous map $w:[0,\infty)\rightarrow \U(\M(B))$ such that
\begin{equation}\label{eqn:asym1}
    H^X(x)=\lim\limits_{t\rightarrow \infty}w_t\rhd h^X(x)\lhd w_t^*
\end{equation}
for all $X\in \Irr(\cC)$ and $x\in F(X)$. Therefore, for any finite sets $K\subset\Irr(\cC)$ containing $1_{\cC}$, $\F^X\subset F(X)$, and $\G^X\subset G(X)$ we may choose $n_1\geq 1$ sufficiently large such that 
\begin{equation}\label{eqn:asym2}
 \sup\limits_{t\geq n_1}\|H^X(x)-w_t\rhd h^X(x)\lhd w_t^*\|\leq \varepsilon/2, \ X\in K,\ x\in \F^X.
\end{equation}
Similarly, one may pick $n_2>n_1$ such that
\begin{equation}\label{eqn:asym3}
 \sup\limits_{t\geq n_2}\|H^X(x)-w_t\rhd h^X(x)\lhd w_t^*\|\leq \varepsilon, \ X\in K, \ x\in (H^X)^{-1}(w_{n_1}\rhd\G^X\lhd w_{n_1}^*).
\end{equation}
We claim that the unitary path $z_t=w_{n_1}^*w_{(1-t)n_1+tn_2}$ satisfies \ref{item:1} and \ref{item:2}. Indeed, it is a strictly continuous path $z:[0,1]\rightarrow \U(\M(B))$ with $z_0=1$. Moreover, for $X\in K$, $t\in [0,1]$ and $x\in \F^X$, using (\ref{eqn:asym2}) we have that
\begin{align*}
\|h^{X}(x)-z_t&\rhd h^X(x)\lhd z_t^*\|\\
&=\|h^{X}(x)-w_{n_1}^*w_{(1-t)n_1+tn_2}\rhd h^X(x)\lhd w_{(1-t)n_1+tn_2}^*w_{n_1}\|\\
&=\|w_{n_1}\rhd h^X(x)\lhd w_{n_1}^*-w_{(1-t)n_1+tn_2}\rhd h^X(x)\lhd w_{(1-t)n_1+tn_2}^*\|\\
&\leq \|w_{n_1}\rhd h^X(x)\lhd w_{n_1}^*-H^X(x)\|\\
&\quad +\|H^X(x)-w_{(1-t)n_1+tn_2}\rhd h^X(x)\lhd w_{(1-t)n_1+tn_2}^*\|\\
&\leq\varepsilon.
\end{align*}
This shows condition \ref{item:1}. We now turn to condition \ref{item:2}. Let $X\in K$, $y\in \G^X$, and $x=(H^X)^{-1}(w_{n_1}\rhd y\lhd w_{n_1}^*)$. Then, we get that
\begin{align*}
    &\|z_1^*\rhd y\lhd z_1-h^X(x)\|\\
    &= \|w_{n_2}^*\rhd(w_{n_1}\rhd y\lhd w_{n_1}^*)\lhd w_{n_2}-h^X(x)\|\\
    &= \|H^X(H^X)^{-1}(w_{n_1}\rhd y\lhd w_{n_1}^*)-w_{n_2}\rhd h^X(x)\lhd w_{n_2}^*\|\\
    &=\|H^X(x)-w_{n_2}\rhd h^X(x)\lhd w_{n_2}^*\|\\
    &\stackrel{\eqref{eqn:asym3}} {\leq}\epsilon.
\end{align*}
As $x\in F(X)$, condition \ref{item:2} holds.
\par We now turn to the if direction. To prove this statement we will use \ref{item:1}-\ref{item:2} to construct a strictly continuous path of unitaries $v_t\in \mathcal{U}(\mathcal{M}(B))$ for $t\in[0,\infty)$ such that $(\Ad(v_t),h_{v_t})\circ (\varphi,h)$ converges to a cocycle conjugacy.
   \par Let $\{x_n^{X}\}_{n\in\mathbb{N}}\subset F(X)$, $\{y_n^X\}_{n\in\mathbb{N}}\subset G(X)$ be countable dense subsets for any $X\in\Irr(\cC)$ and $ K_n$ increasing finite subsets of $\Irr(\mathcal{C})$ containing $1_{\cC}$ such that $\Irr(\mathcal{C})=\bigcup_{n\in\N} K_n$. Firstly, use \ref{item:1}-\ref{item:2} to find $x_{1,1}^X\in F(X)$ for $X\in K_1$ and $z^{(1)}:[0,1]\rightarrow \U(\M(B))$ such that $z_0^{(1)}=1$ and for $0\leq t\leq 1$
\begin{enumerate}
   \item $\|z_t^{(1)}\rhd h^X(x_1^{X})\lhd (z_t^{(1)})^*-h^X(x_1^{X})\|\leq 1/2$ for $X\in K_1$,\\
   \item $\|(z_1^{(1)})^*\rhd y_1^{X}\lhd z_1^{(1)}-h^X(x_{1,1}^{X})\|\leq 1/2$ for $X\in K_1$.
\end{enumerate}
Again use \ref{item:1}-\ref{item:2} to find $x_{2,1}^X$, $x_{2,2}^X$ in $F(X)$ for $X\in K_2$, $z^{(2)}:[0,1]\rightarrow \U(\M(B))$ such that $z^{(2)}_0=1$ and for every $0\leq t\leq 1$
\begin{enumerate}
    \item $\|z_t^{(2)}\rhd h^X(x_j^{X})\lhd (z_t^{(2)})^*-h^X(x_j^{X})\|\leq 1/4$ for $X\in K_2$ and $1\leq j \leq 2$,\\
    \item $\|z_t^{(2)}\rhd h^X(x_{1,1}^{X})\lhd (z_t^{(2)})^*-h^X(x_{1,1}^{X})\|\leq 1/4$ for $X\in K_2$,\\
    \item $\|(z_1^{(2)})^*\rhd((z_1^{(1)})^*\rhd y_j^{X}\lhd z_1^{(1)})\lhd z_1^{(2)}-h^X(x_{2,j}^X)\|\leq 1/4$ for $X\in K_2$ and $1\leq j\leq 2$.
\end{enumerate}
Now suppose you have $z^{(k)}:[0,1]\rightarrow \U(\M(B))$ for $1\leq k\leq n$ with $z_0^{(k)}=1$ and $x_{m,j}\in F(X)$ for $X\in K_m$ with $1\leq j\leq m \leq n$ such that for any $t\in[0,1]$
\begin{enumerate}[label=\textbf{S.\arabic*}]
    \item $\|z_t^{(n)}\rhd h^X(x_j^{X})\lhd (z_t^{(n)})^*-h^X(x_j^{X})\|\leq 2^{-n}$ for $X\in K_n$ and $1\leq j \leq n$,\label{n-thstep1}\\
    \item $\|z_t^{(n)}\rhd h^X(x_{m,j}^{X})\lhd (z_t^{(n)})^*-h^X(x_{m,j}^{X})\|\leq 2^{-n}$ for $X\in K_n$ and $1\leq j \leq m<n$,\label{n-thstep2}\\
    \item $\lVert\left((z_1^{(n)})^*\ldots(z_1^{(1)})^*\rhd y_j^{X}\lhd z_1^{(1)}\ldots z_1^{(n)}\right)-h^X(x_{n,j}^X)\rVert\leq 2^{-n}$ for $X\in K_n$ and $1\leq j\leq n$.\label{n-thstep3}
\end{enumerate}
Then use \ref{item:1}-\ref{item:2} to get $\{x_{n+1,j}^X\}_{j\leq n+1}\in F(X)$ for $X\in K_{n+1}$ and $z^{(n+1)}:[0,1]\rightarrow \U(\M(B))$ such that for all $t\in[0,1]$
\begin{enumerate}
    \item $\|z_t^{(n+1)}\rhd h^X(x_j^{X})\lhd (z_t^{(n+1)})^*-h^X(x_j^{X})\|\leq 2^{-(n+1)}$ for $X\in K_{n+1}$ and $1\leq j \leq n+1$,\\
    \item $\|z_t^{(n+1)}\rhd h^X(x_{m,j}^{X})\lhd (z_t^{(n+1)})^*-h^X(x_{m,j}^{X})\|\leq 2^{-(n+1)}$ for $X\in K_{n+1}$ and $1\leq j \leq m<n+1$,\\
    \item $\|\left((z_1^{(n+1)})^*\ldots(z_1^{(1)})^*\rhd y_j^{X}\lhd z_1^{(1)}\ldots z_1^{(n+1)}\right)-h^X(x_{n+1,j}^X)\|\leq 2^{-(n+1)}$ for $X\in K_{n+1}$ and $1\leq j\leq n+1$.
\end{enumerate}
We carry on inductively to construct $z^{(n)}$ and $x_{m,j}^X$ for every $n,m\in \N$ and $j\leq m$ satisfying \ref{n-thstep1}-\ref{n-thstep3}.\ We may now define the path $v_t:[0,\infty)\rightarrow \U(\M(B))$ by $v_t=z_1^{(1)}\ldots z_1^{(n)}z_{t-n}^{(n+1)}$ for every $t\in [n,n+1]$. This path is strictly continuous on every open interval $(n,n+1)$ for $n\in \N$ as the paths $z^{(k)}$ are strictly continuous for each $k\in\N$. Moreover, the path $v_t$ is strictly continuous at each $n\in \N$ as $z_0^{(k)}=1$ for every $k\in \N$. Adjoining by $v_t$ we obtain a continuous path of cocycle morphisms $(\Ad(v_t),h_{v_t})\circ(\varphi,h)=(\psi_t,h_t)$ where $\psi_t=\Ad(v_t)\circ \varphi$ and $h_t^{X}=h_{v_t}^X\circ h^X$ for any $t\in[0,\infty)$ and $X\in \Irr(\cC)$. 

We claim that as $t\rightarrow \infty$ the path $\psi_t$ converges to an isomorphism $\Psi$ and that the path $h_t^X$ converges to a bijective linear map $H^X$ for all $X\in \Irr(\cC)$ such that the pair $(\Psi,H)$ induces a cocycle morphism (recall Remark \ref{rmk: IrredDetLinearMaps}).\ For any $X\in \Irr(\cC)$ and $j\in \N$ the net $(h_t^X(x_j^X))_{t\geq 0}$ is Cauchy by \ref{n-thstep1}.\ Since the set $\{x_j^X\}_{j\in \N}$ is dense in $F(X)$ for any $X\in \Irr(\cC)$ a standard triangle inequality argument shows that the net $(h_t^X(x))_{t\geq 0}$ converges for any $x\in F(X)$ and $X\in \Irr(\cC)$. We may hence define $H^X(x)=\lim_{t\rightarrow \infty}h_t^X(x)$ for all $x\in F(X)$ and $X\in \Irr(\cC)$. The maps $H^X:F(X)\rightarrow G(X)$ are linear as inherited by the linearity of $h_t^X$. In light of Remark \ref{rmk: IrredDetLinearMaps}, it suffices to check that the family of maps $\{H^X\}_{X\in\Irr(\cC)}$ satisfies the conditions \ref{item:linearmap2}-\ref{item:linearmap4} of Lemma \ref{linearmapspicture} to conclude that $(\Psi, H)$ is a cocycle morphism. They are easily verified as $(\psi_t,h_t)$ is a cocycle morphism for every $t\in [0,\infty)$ and one may approximate $H$ pointwise by $h_t$ for some large enough $t$. 
\par It remains to check that $H^X$ is bijective for every $X\in\cC$. By Remark \ref{rmk: IrredDetLinearMaps}, it suffices to check that $H^X$ is bijective for every $X\in\Irr(\cC)$. Firstly, each $H^X$ is isometric. To see this, note that $\Psi$ is an injective $^*$-homomorphism and hence isometric, so as $(\Psi,H)$ is a cocycle morphism 
\begin{align*}
    \|H^X(x)\|^2&=\|\langle H^X(x),H^X(x)\rangle_B\|\\
    &=\|\Psi(\langle x,x\rangle_A)\|\\
    &=\|\langle x,x\rangle_A\|\\
    &=\|x\|^2
\end{align*}
for any $X\in \cC$ and $x\in F(X)$.\ Therefore $H^X$ is injective for every $X\in \cC$. We now turn to surjectivity. As $H^X$ are isometric, it suffices to show that $H^X$ have dense image. Fix $X\in \Irr(\cC)$ and $j\in \N$. There is a large enough $n_0\in \N$ satisfying $n_0>j$ and $X\in K_n$ for all $n\geq n_0$. So by \ref{n-thstep3} $\|h_n^X(x_{n,j}^X)-y_j^{X}\|\leq 2^{-n}$ for any $n\geq n_0$. Moreover, by \ref{n-thstep2}
\begin{align*}
    \|H^X(x_{n,j}^X)-h_n^X(x_{n,j})\|&\leq \sum_{k=n}^\infty\|h_{k+1}^X(x_{n,j}^X)-h_{k}^X(x_{n,j}^X)\|\\
    &\leq\sum_{k=n}^\infty \|z_1^{(k+1)}\rhd h^X(x_{n,j}^X)\lhd (z_1^{(k+1)})^*-h^X(x_{n,j}^X)\|\\
    &\leq \sum_{k=n}^\infty 2^{-(k+1)}\\
    &= 2^{-n},
\end{align*}
for any $n\geq n_0$. So $\|y_j^{X}-H^X(x_{n,j}^X)\|\leq 2^{1-n}$. Therefore, as $n$ may be chosen arbitrarily and $\{y_j^{X}\}_{j\in\mathbb{N}}$ is dense in $G(X)$, it follows that $H^X$ is surjective.
\end{proof}

Theorem \ref{thm: onesidedintert} also holds in the setting of approximate unitary equivalence by replacing the path of unitaries with a single unitary and dropping the assumption that $z_0=\mathbf{1}$. 

\begin{theorem}\label{thm: onesidedintertApprox}
    Let $\cC$ be a semisimple C$^*$-tensor category with countably many isomorphism classes of simple objects.\ Let $(F,J):\cC\curvearrowright A$ and $(G,I):\cC\curvearrowright B$ be actions on separable C$^*$-algebras and $(\varphi,h):(A,F,J)\rightarrow (B,G,I)$ an injective cocycle morphism.\ Then $(\varphi,h)$ is approximately unitarily equivalent to a cocycle conjugacy if and only if:
    \par For all $\varepsilon>0$ and finite sets $K\subset\Irr(\cC)$ containing $1_{\cC}$, $\F^X\subset F(X)$ and $\G^X\subset G(X)$ there exists a unitary $z\in\mathcal{M}(B)$ such that
    \begin{enumerate}[label=\textbf{A.\arabic*}]
        \item $\|z\rhd h^X(x)\lhd z^*-h^X(x)\|\leq \varepsilon$ for all $X\in K$ and $x\in \F^X$,\label{item:1approx}
        \item $\dist(z^*\rhd y\lhd z,h^X(F(X)))\leq \varepsilon$ for all $X\in K$ and $y\in \G^X$.\label{item:2approx}        
    \end{enumerate}
\end{theorem}

\begin{lemma}\label{lemma: AsymCocConj}
Let $\cC$ be a semisimple C$^*$-tensor category with countably many isomorphism classes of simple objects. Let $(F_n,J^{(n)}): \cC\curvearrowright A_n$ be a  sequence of actions on separable $\C$-algebras. Let
\[
(\phi_n,h_n): (A_n,F_n,J^{(n)}) \to (A_{n+1},F_{n+1},J^{(n+1)})
\] be a sequence of injective and extendible cocycle morphisms with inductive limit $(A,F,J)=\lim\limits_{\longrightarrow}(A_n,F_n,J^{(n)})$. Suppose that for every $n\geq 1$, $(\phi_n,h_n)$ is asymptotically unitarily equivalent to a cocycle conjugacy. Then it follows that $$(\phi_{1,\infty},h_{1,\infty}):(A_1,F_1,J^{(1)})\to (A,F,J)$$ is asymptotically unitarily equivalent to a cocycle conjugacy.
\end{lemma}

\begin{proof}
Let $\varepsilon>0$ and finite sets $K\subset\Irr(\cC)$ containing $1_{\cC}$, $\F_1^X\subset F_1(X),\ \F^X\subset F(X)$. We will check the conditions in Theorem \ref{thm: onesidedintert} for $(\phi_{1,\infty},h_{1,\infty})$. Perturbing $\F^X$ by an arbitarily small tolerance, we may assume that there exists $n\geq 1$ large enough and a finite set $\F_n^X\subset F_n(X)$ such that $\F^X=h_{n,\infty}^X(\F_n^X)$ for any $X\in K$.

By Lemma \ref{lemma: CompConjIsConjAsympt}, $(\phi_{1,n},h_{1,n})$ is asymptotically unitarily equivalent to a cocycle conjugacy. Then, by Theorem \ref{thm: onesidedintert}, there exists a unitary path $y:[0,1]\to \mathcal{U}(\mathcal{M}(A_n))$ such that $y_0=1$ and
 \begin{enumerate}
        \item $\sup\limits_{0\leq t\leq 1}\|y_t\rhd h_{1,n}^X(\mu)\lhd y_t^*-h_{1,n}^X(\mu)\|\leq\varepsilon$ for any $X\in K, \mu\in \F_1^X$,
        \item $\dist(y_1^*\rhd\eta\lhd y_1, h_{1,n}^X(F_1(X)))\leq\varepsilon$ for any $X\in K$, $\eta\in \F_n^X$ .
\end{enumerate}

Let $z_t=\phi_{n,\infty}^{\dagger}(y_t)$ be a unitary in $\mathcal{M}(A)$ for any $t\in[0,1]$, where $\phi_{n,\infty}^{\dagger}:\U(\mathcal{M}(A_n))\to \U(\mathcal{M}(A)) $ is as in Definition \ref{defn: extendible}. Now we claim that the path of unitaries $z_t$ satisfies the conditions of Theorem \ref{thm: onesidedintert} for $(\phi_{1,\infty},h_{1,\infty})$. 

By \ref{item:linearmap1} and \ref{item:linearmap3} of Lemma \ref{linearmapspicture} (see also Remark \ref{rmk: ExtLinMaps}) we get that
\begin{equation*}
\|z_t\rhd h_{1,\infty}^X(\mu)\lhd z_t^*-h_{1,\infty}^X(\mu)\|\leq \|y_t\rhd h_{1,n}^X(\mu)\lhd y_t^*-h_{1,n}^X(\mu)\|\leq\varepsilon
\end{equation*}
for any $X\in K$, $\mu\in \F_1^X$, and any $t\in[0,1]$.\ Moreover, recall that $\F^X=h_{n,\infty}^X(\F_n^X)$ and $$\dist(y_1^*\rhd\eta\lhd y_1, h_{1,n}^X(F_1(X)))\leq\varepsilon$$ for any $X\in K$ and any $\eta\in \F_n^X$.\ Hence, by applying $h_{n,\infty}^X$ to the formula above, it follows that for any $X\in K$, $\xi\in \F^X$, $$\dist(z_1^*\rhd\xi\lhd z_1, h_{1,\infty}^X(F_1(X)))\leq\varepsilon.$$ Therefore, $(\phi_{1,\infty},h_{1,\infty})$ is asymptotically unitarily equivalent to a cocycle conjugacy by Theorem \ref{thm: onesidedintert}.
\end{proof}

The following result can be seen as the asymptotic version of Corollary \ref{intertidentity}.

\begin{theorem}\label{thm: AsymptIntertIdentity}
 Let $\cC$ be a semisimple C$^*$-tensor category with countably many isomorphism classes of simple objects. Let $(F,J): \cC\curvearrowright A$ and $(G,I): \cC\curvearrowright B$ be actions on separable $\C$-algebras. Let
\[
(\phi, h): (A,F,J) \to (B,G,I) \quad
and \quad
(\psi, l): (B,G,I) \to (A,F,J)
\] be two extendible cocycle morphisms such that $$\id_A\cong_u (\psi,l)\circ (\phi, h)\quad and \quad \id_B\cong_u (\phi,h)\circ (\psi, l).$$ Then there exist mutually inverse cocycle conjugacies \[
(\Phi, H): (A,F,J) \to (B,G,I) \quad
and \quad
(\Psi, L): (B,G,I) \to (A,F,J)
\] such that $$(\Phi, H)\cong_u (\phi,h) \quad and \quad (\Psi, L)\cong_u (\psi,l).$$
\end{theorem}

\begin{proof}
Consider the cocycle morphisms $$(\kappa,r)= (\psi,l)\circ (\phi,h) : (A,F,J)\to (A,F,J)$$ and $$(\theta,s) = (\phi,h)\circ (\psi,l) :(B,G,I)\to (B,G,I)$$ fitting into the family of commuting diagrams
\begin{equation}\label{eq: AsymptInt}
\xymatrix{
\dots\ar[rr] && F(X) \ar[rd]^{h^X} \ar[rr]^{r^X} && F(X) \ar[r] \ar[rd]^{h^X} & \dots\\
\dots\ar[r] & G(X) \ar[ru]^{l^X} \ar[rr]^{s^X} && G(X) \ar[ru]^{l^X} \ar[rr]^{s^X} && \dots \quad .
}
\end{equation}

We can then form the inductive limits $$(A^{(\infty)},F^{(\infty)},J^{(\infty)})=\lim\limits_{\longrightarrow}\{(A,F,J),(\kappa,r)\}$$ and $$(B^{(\infty)},G^{(\infty)},I^{(\infty)})=\lim\limits_{\longrightarrow}\{(B,G,I),(\theta,s)\}$$ in the category $\C_{\cC}$. Consider the universal embeddings $$(\kappa_\infty,r_\infty):(A,F,J)\to (A^{(\infty)},F^{(\infty)},J^{(\infty)})$$ and $$(\theta_\infty,s_\infty):(B,G,I)\to (B^{(\infty)},G^{(\infty)},I^{(\infty)}).$$ Since the collection of diagrams in \eqref{eq: AsymptInt} commutes, the universal properties of both inductive limits yield that there exist mutually inverse cocycle conjugacies $$(\phi_\infty,h_\infty): (A^{(\infty)},F^{(\infty)},J^{(\infty)})\to (B^{(\infty)},G^{(\infty)},I^{(\infty)})$$ and $$(\psi_\infty, l_\infty): (B^{(\infty)},G^{(\infty)},I^{(\infty)})\to (A^{(\infty)},F^{(\infty)},J^{(\infty)})$$ such that $(\phi_\infty,h_\infty)\circ (\kappa_\infty,r_\infty)=(\theta_\infty,s_\infty)\circ (\phi,h)$ and $(\psi_\infty, l_\infty)\circ (\theta_\infty,s_\infty) =(\kappa_\infty,r_\infty) \circ (\psi,l)$.

Furthermore, the cocycle morphism $(\kappa,r)$ is asymptotically unitarily equivalent to the cocycle morphism induced by the identity map on $A$, which in particular implies that $(\kappa,r)$ is injective and thus Lemma \ref{lemma: AsymCocConj} gives that $(\kappa_\infty,r_\infty)$ is asymptotically unitarily equivalent to a cocycle conjugacy $(K, R) : (A,F,J)\to (A^{(\infty)},F^{(\infty)},J^{(\infty)})$. Likewise, $(\theta_\infty,s_\infty)$ is asymptotically unitarily equivalent to a cocycle conjugacy $(\Theta, S):(B,G,I)\to (B^{(\infty)},G^{(\infty)},I^{(\infty)})$.

Then, taking $$(\Phi, H)=(\Theta, S)^{-1}\circ (\phi_\infty, h_\infty)\circ(K, R)$$ yields that 
\begin{align*}
(\Phi, H) &\cong_u (\Theta, S)^{-1}\circ (\phi_\infty, h_\infty)\circ (\kappa_\infty,r_\infty)\\ & = (\Theta, S)^{-1}\circ (\theta_\infty,s_\infty)\circ (\phi,h) \\ &\cong_u (\phi,h).
\end{align*} Similarly, if we take $(\Psi, L)=(\Phi,H)^{-1}=(K, R)^{-1}\circ (\psi_\infty,l_\infty)\circ (\Theta, S)$, we get that $(\Psi, L)\cong_u(\psi,l)$, which finishes the proof.
\end{proof}

\begin{rmk}
Note that in Theorem~\ref{thm: AsymptIntertIdentity}, if the asymptotic unitary equivalences are realised by unitaries in the minimal unitisations, then we may drop the assumption of extendibility on the connecting morphisms. This follows as extendibility is only required to evaluate a morphism on a given unitary in the multiplier algebra.
However, there is no issue in doing so when the unitary is in the minimal unitisation.
\end{rmk}

\subsection*{Acknowledgements} We would like to thank Stuart White and Samuel Evington for their supervision on this project.\ We would also like to thank Corey Jones, George Elliott, and Gabór Szabó for useful discussions on the topic of this paper.\ Part of this work was completed during the authors' stay at the Fields Institute for Research in Mathematical Sciences for the 'Thematic Program on Operator Algebras and Applications' in Autumn 2023.\ We thank the Fields Institute and the organisers for the hospitality.

The first named author was supported by the Ioan and Rosemary James Scholarship awarded by St John's College and the Mathematical Institute, University of Oxford, as well as by projects G085020N and 1249225N funded by the Research Foundation Flanders (FWO).
The second named author was supported by the EPSRC grant EP/R513295/1, by the European Research Council under the European Union's Horizon Europe research and innovation programme (ERC grant AMEN--101124789), and by the postdoctoral fellowship 1204626N of the Research Foundation Flanders (FWO).

The authors' stay at the Fields Institute was partially funded by a Special Grant awarded by St John's College, Oxford.\ The first named author's stay was also partially supported by the Fields Institute while the travel costs of the second named author were supported by the EPSRC grant EP/X026647/1.

For the purpose of open access, the authors have applied a CC BY public copyright license to any author accepted manuscript version arising from this submission.

\allowdisplaybreaks

\bibliographystyle{abbrv}
\bibliography{categories}
\end{document}